\documentclass[a4paper,10pt]{amsart}
\usepackage[english]{babel}
\usepackage{amsmath,tikz-cd}
\usepackage{amssymb}
\usepackage[T1]{fontenc}
\usepackage[utf8]{inputenc}
\usepackage{longtable}
\usepackage{array}
\usepackage{lmodern}
\usepackage{mathrsfs}
\usepackage{enumitem}
\usepackage{mathtools}
\usetikzlibrary{arrows}
\usepackage{xcolor}
\definecolor{DarkRed}{RGB}{173,0,0}
\definecolor{LightRed}{RGB}{201,0,0}
\usepackage[
    colorlinks=true,
    linkcolor=DarkRed,
    urlcolor=LightRed,
    citecolor=LightRed
]{hyperref}

\newtheorem{thm}{Theorem}[section]
\newtheorem{Lemma}[thm]{Lemma}
\newtheorem{Proposition}[thm]{Proposition}

\theoremstyle{definition}

\newtheorem{Definition}[thm]{Definition}
\newtheorem{Remark}[thm]{Remark}

\newtheoremstyle{introthmstyle} {5pt}{5pt} {\itshape} {} {\bfseries} {.} { } {} \theoremstyle{introthmstyle} \newtheorem{introthm}{Theorem}

\definecolor{wwwwww}{rgb}{0.4,0.4,0.4}

\newcommand{\PP}{\mathbb{P}}

\newcommand{\QQ}{\mathbb{Q}}

\newcommand{\CC}{\mathbb{C}}
\newcommand{\kk}{k}
\newcommand{\OO}{\mathcal{O}}

\DeclareMathOperator{\Cr}{Cr}
\DeclareMathOperator{\Proj}{Proj}

\DeclareMathOperator{\Bir}{Bir}
\DeclareMathOperator{\rk}{rk}
\DeclareMathOperator{\codim}{codim}
\DeclareMathOperator{\Hilb}{Hilb}
\DeclareMathOperator{\Bl}{Bl}

\DeclareMathOperator{\Gr}{Gr}
\DeclareMathOperator{\PGL}{PGL}
\DeclareMathOperator{\Bs}{Bs}
\DeclareMathOperator{\Supp}{Supp}
\DeclareMathOperator{\multideg}{multideg}
\DeclareMathOperator{\length}{length}

\newcommand{\I}{\mathcal{I}}

\newcommand{\vf}{\varphi}

\setlist[enumerate]{label=(\roman*),before=\normalfont,font=\normalfont}
\hypersetup{pdfpagemode=UseNone}
\hypersetup{pdfstartview=FitH}

\begin{document}

\title{Quadratic Cremona transformations of $\PP^4$}

\author[Gianluca Grassi]{Gianluca Grassi}
\address{\sc Gianluca Grassi\\ Dipartimento di Matematica e Informatica, Universit\`a di Ferrara, Via Machiavelli 30, 44121 Ferrara, Italy}
\email{gianluca.grassi@unife.it}

\author[Alex Massarenti]{Alex Massarenti}
\address{\sc Alex Massarenti\\ Dipartimento di Matematica e Informatica, Universit\`a di Ferrara, Via Machiavelli 30, 44121 Ferrara, Italy}
\email{msslxa@unife.it}

\date{\today}
\subjclass[2020]{Primary 14E07, 14E05; Secondary 14N05, 14Q10.}
\keywords{Cremona transformation, quadratic map, multidegree, Segre class, base locus, ribbon, de Jonqui\`eres net.}

\begin{abstract}
We prove that $\Bir_2(\PP^4)$ has exactly nine irreducible components and that the locus of transformations with zero-dimensional base scheme is irreducible.
\end{abstract}

\maketitle
\tableofcontents

\section*{Introduction}

Throughout the paper we work over the field $\kk=\CC$, and write
$\PP^n=\PP^n_{\CC}$.  This convention is needed only when we invoke the
classification of Pan--Ronga--Vust in dimension three; all elementary
constructions and intersection-theoretic computations below are valid over
any algebraically closed field of characteristic zero. Recall that the
Cremona group of rank $n$ is
$
\Cr_n(\kk)=\Bir(\PP^n),
$
that is the group of birational self-maps of projective space, or equivalently the automorphism group of the purely transcendental field $\kk(x_1,\ldots,x_n)$. The group $\Cr_n(\kk)$ contains $\PGL_{n+1}(\kk)$ as the subgroup of projective automorphisms and coincides with it only when $n=1$. For $n\ge 2$ it is an infinite-dimensional object, central in birational geometry and still only partially understood.

The first non-trivial case is $n=2$, the classical one. By the Noether--Castelnuovo theorem the group $\Cr_2(\kk)$ is generated by $\PGL_3(\kk)$ and by the standard quadratic involution
$
[x_0:x_1:x_2]\dashrightarrow [x_1x_2:x_0x_2:x_0x_1].
$
Therefore, plane Cremona transformations, and in particular quadratic ones, are quite well understood: up to linear equivalence, a general quadratic transformation is the standard involution, with base locus given by three non-collinear points, while the cases with infinitely near base points appear as degenerations. We refer to \cite{Alberich02,Dolgachev12,Hudson27} for the classical theory.

On the other hand, even in the plane the algebraic geometry of the Cremona group is subtle. If $\Bir_d(\PP^2)$ denotes the variety of plane Cremona transformations of degree $d$, the number and the distribution of its irreducible components form a difficult problem. The study of the algebraic growth of $\Cr_2$ gives two-sided asymptotic bounds for the number of irreducible components of $\Bir_d(\PP^2)$ as $d$ grows; more precisely, after taking two logarithms, the cumulative number of components grows as $\sqrt{\log d}$ \cite{CCMMM25}.

The group-theoretic structure of Cremona groups has also attracted much
attention. S. Cantat and S. Lamy proved that $\Cr_2(\kk)$ is not simple when $\kk$ is
algebraically closed \cite{CantatLamy13}. J. Blanc proved connectedness results
for Cremona groups and topological simplicity for $\Cr_2$ with the Zariski
topology \cite{Blanc10}. J. Blanc and S. Zimmermann later proved topological simplicity
of Cremona groups in arbitrary dimension over infinite fields, for the Zariski
topology and also for the Euclidean topology over local fields
\cite{BlancZimmermann18}. On the other hand, J. Blanc, S. Lamy and S. Zimmermann proved
that, over subfields of $\CC$, $\Cr_n$ is not simple for $n\ge 3$ and has many
non-trivial quotients \cite{BlancLamyZimmermann21}. These results show that
Cremona groups in higher dimension are large and rigid from some viewpoints,
but at the same time far from being completely accessible.

For $n\ge 3$ there is no analogue of the Noether--Castelnuovo theorem. In
particular, the group is not generated by transformations of bounded degree,
as already shown classically by H. Hudson \cite{Hudson27}. Therefore, one cannot hope
to understand $\Cr_n$ by reducing all transformations to the quadratic ones.
Still, quadratic transformations form the first non-linear stratum, and their
classification is a natural testing ground for the geometry of the higher
Cremona groups.

Quadratic transformations of $\PP^3$ were studied by I. Pan, F. Ronga and T. Vust
\cite{PanRongaVust01}. In that case there are three possible bidegrees,
namely $(2,2)$, $(2,3)$ and $(2,4)$, and the corresponding families and their
degenerations can be described explicitly. In higher dimension the situation
becomes more complicated for two independent reasons. First, the inverse degree
is no longer enough to describe the numerical behavior of a birational map:
one has to use the whole multidegree. Secondly, the base scheme may have
several two-dimensional components.

The quadro-quadric case, that is the case in which both a Cremona
transformation and its inverse are defined by quadrics, has a particularly rich
structure. J. G. Semple and L. Roth studied the case of $\PP^4$ \cite{SempleRoth49}, and A. Bruno and A. Verra later gave a
modern treatment and described general base schemes in dimensions $4$ and $5$
\cite{BrunoVerra11}. L. Pirio and F. Russo developed the JC-correspondence, relating
quadro-quadric Cremona transformations to rank three Jordan algebras, and used
it to obtain complete explicit classifications in low dimensions
\cite{PirioRusso14}. In particular, the general quadro-quadric base schemes in
$\PP^4$ are known.

In this paper we study all quadratic Cremona transformations of $\PP^4$ from
the point of view of their multidegrees. Let
$
V=H^0(\PP^4,\OO_{\PP^4}(2)).
$
A quadratic rational map $\PP^4\dashrightarrow\PP^4$ is determined, up to a
projective change of coordinates on the target, by a point of $\Gr(5,V)$.
We say that a map represented by quadratic forms has \emph{algebraic degree
exactly two} if the forms have no non-constant common divisor. We denote by
$\Bir_2(\PP^4)\subset\Gr(5,V)$ the locally closed locus of five-dimensional
subspaces $W$ such that $\gcd(W)=1$ and the associated rational map is
birational. Thus $\Bir_2(\PP^4)$ parametrizes birational maps of minimal
algebraic degree exactly two. We work inside $\Gr(5,V)$ and we do not quotient
by $\PGL_5$; all dimensions in the sequel are dimensions of the parameter
spaces considered inside this Grassmannian.

The multidegree of a rational map $\vf:\PP^4\dashrightarrow\PP^4$ is
$
\multideg(\vf)=(d_0,d_1,d_2,d_3,d_4),
$
where $d_i$ is the degree of the inverse image of a general codimension $i$
linear subspace of the target. If $\vf$ is birational, then $d_0=d_4=1$; if its minimal algebraic
degree is exactly two, then $d_1=2$. Therefore, we write the multidegree in
the form
$
(1,2,d_2,d_3,1),
$
or, in abbreviated form, as $(2,d_2,d_3)$. Our first result is numerical.

\begin{introthm}\label{intro:numerics}
Let $\vf:\PP^4\dashrightarrow\PP^4$ be a quadratic Cremona transformation.
Then $\multideg(\vf)=(1,2,d_2,d_3,1)$ and $(d_2,d_3)$ belongs to
the following list: $
(2,2),\, (3,2),(3,3),(3,4),\,
(4,2),(4,3),(4,4),(4,5),(4,6),(4,7),(4,8)$.
\end{introthm}

The proof of Theorem~\ref{intro:numerics} follows from the Cremona inequalities and the Hodge inequalities for multidegrees. On the other hand, the realization of all the numerical possibilities is a geometric problem. We construct explicit families of base schemes imposing ten independent conditions on quadrics and compute their multidegrees via the Segre class formula. In all these cases the last projective degree is $d_4=1$, and hence the corresponding rational map is generically finite of degree one and therefore birational.

We summarize these constructions in the following theorem. In the table, a notation such as $p^\Lambda$ denotes a first order linear scheme supported at $p$ whose tangent space is the linear subspace $\Lambda\subset \PP^4$ through $p$. Therefore, $p^{\PP^2}$ imposes three independent conditions on quadrics, while $p^{\PP^3}$ imposes four. When such a first-order condition is supported on a positive-dimensional component, the displayed scheme is the natural integral-closure model of the generated base ideal. The five quadrics may generate a proper reduction of that ideal; reductions have the same normalized blow-up and hence the same Segre class. This distinction is made explicit in Proposition~\ref{prop:curve-families}.

\begin{introthm}\label{intro:families}
For every multidegree in Theorem~\ref{intro:numerics} there exists a quadratic
Cremona transformation of $\PP^4$ having this multidegree. More precisely,
the following table gives natural families of such transformations.

\begingroup
\normalfont
\setlength{\LTleft}{\fill}
\setlength{\LTright}{\fill}
\renewcommand{\arraystretch}{1.0}
\begin{footnotesize}
\begin{longtable}{
    >{\itshape\centering\arraybackslash}p{0.14\textwidth}|
    >{\itshape\centering\arraybackslash}p{0.67\textwidth}|
    >{\itshape\centering\arraybackslash}p{0.10\textwidth}
}
$(d_1,d_2,d_3)$
& general base configuration (or integral-closure model)
& dimension
\\
\hline
\endfirsthead

\multicolumn{3}{c}{\small\itshape }
\\[2pt]
$(d_1,d_2,d_3)$
& general base scheme of the family
& dimension
\\
\hline
\endhead

\hline
\multicolumn{3}{r}{\small\itshape }
\endfoot

\hline
\endlastfoot

$(2,4,8)$
& three reduced points, one point with Hilbert function $(1,2)$, and one point with Hilbert function $(1,3,2)$
& $27$
\\

$(2,4,7)$
& $L+p_L^{\PP^3}+q^{\PP^2}+r+s$
& $25$
\\

$(2,4,7)$
& reduced line $L+$ length-$7$ point $p+$ reduced point $q$ (de Jonqui\`eres)
& $24$
\\

$(2,4,6)$
& conic $C+p_C^{\PP^3}+q_1+q_2+q_3$
& $26$
\\

$(2,4,6)$
& genus $-2$ ribbon $R_{-2}+q_1+q_2+q_3$
& $26$
\\

$(2,4,6)$
& two skew lines $L_1,L_2+q^{\PP^2}+r$
& $24$
\\

$(2,4,5)$
& twisted cubic $C+p^{\PP^2}$
& $24$
\\

$(2,4,5)$
& conic $C+$ line $L+q_1+q_2$
& $25$
\\

$(2,4,4)$
& triple line $L^H+q_1+q_2+q_3$
& $20$
\\

$(2,4,4)$
& rational normal quartic $C+q$
& $25$
\\

$(2,4,3)$
& elliptic normal quintic $E$
& $25$
\\

$(2,4,2)$
& Semple--Roth / Bruno--Verra triple line plus conic
& $16$
\\

$(2,3,4)$
& plane $P+q_1+q_2+q_3+q_4$
& $22$
\\

$(2,3,3)$
& plane $P+$ line $L$ meeting $P$ $+q_1+q_2$
& $19$
\\

$(2,3,2)$
& plane $P+$ two skew lines meeting $P$
& $16$
\\

$(2,2,2)$
& smooth quadric surface $Q+q$
& $17$
\end{longtable}
\end{footnotesize}
\endgroup
\end{introthm}

Let us explain two names used in the table. The ribbon family is named after its one-dimensional base scheme, which is a non-reduced double structure $R_{-2}$ on a line $L$ whose square-zero ideal in $R_{-2}$ is an invertible sheaf; therefore, it is a ribbon in the usual sense, and the subscript records its arithmetic genus $-2$. The reduced-line de Jonqui\`eres family is obtained by lifting, from a smooth $(2,2)$ surface, the plane homaloidal net of type $(4;3,1^6)$, which is a classical de Jonqui\`eres system.

The zero-dimensional case is particularly rigid. Indeed, the first family in the table is the unique component of the locus of transformations with zero-dimensional base scheme.

\begin{introthm}\label{intro:zero-dimensional}
Let
$
\mathcal C^0_{248}\subset
\Gr\bigl(5,H^0(\PP^4,\OO_{\PP^4}(2))\bigr)
$
be the locus of quadratic Cremona transformations with zero-dimensional base
locus. Then $\mathcal C^0_{248}$ is irreducible of dimension $27$. The general
point of this unique component is the five-dimensional subspace
$H^0(\I_Z(2))$, where
$Z=p_1\cup p_2\cup p_3\cup q^\Pi\cup r^H$,
$p_1,p_2,p_3,q,r\in\PP^4$ are general pairwise distinct points,
$\Pi\subset\PP^4$ is a general plane through $q$, and
$H\subset\PP^4$ is a general hyperplane through $r$. The scheme $Z$ records the imposed first-order conditions. The actual base
scheme is supported at the same five points; its local Hilbert functions are
$(1),(1),(1),(1,2),(1,3,2)$, respectively. Its multidegree is
$(1,2,4,8,1)$.
\end{introthm}

Our main result can be stated as follows.

\begin{introthm}\label{intro:nine-components}
The space $\Bir_2(\PP^4)$ has exactly nine irreducible components, described in the following table.
\begingroup
\setlength{\LTleft}{\fill}
\setlength{\LTright}{\fill}
\renewcommand{\arraystretch}{1.0}
\begin{footnotesize}
\begin{longtable}{
    >{\itshape\centering\arraybackslash}p{0.14\textwidth}|
    >{\itshape\centering\arraybackslash}p{0.67\textwidth}|
    >{\itshape\centering\arraybackslash}p{0.10\textwidth}
}
$(d_1,d_2,d_3)$
& general base scheme
& dimension
\\
\hline
\endfirsthead

\multicolumn{3}{c}{\small\itshape }
\\[2pt]
$(d_1,d_2,d_3)$
& general base scheme
& dimension
\\
\hline
\endhead

\hline
\multicolumn{3}{r}{\small\itshape }
\endfoot

\hline
\endlastfoot

$(2,4,8)$
& three reduced points, one point with Hilbert function $(1,2)$, and one point with Hilbert function $(1,3,2)$
& $27$
\\

$(2,4,7)$
& $L+Z_7+q$
& $24$
\\

$(2,4,6)$
& $R_{-2}+q_1+q_2+q_3$
& $26$
\\

$(2,4,6)$
& $L_1\sqcup L_2+p^{\PP^2}+q$
& $24$
\\

$(2,4,5)$
& $T+p^{\PP^2}$
& $24$
\\

$(2,4,5)$
& $C\sqcup L+q_1+q_2$
& $25$
\\

$(2,4,4)$
& $R_4+q$
& $25$
\\

$(2,4,3)$
& $E_5$
& $25$
\\

$(2,4,6)$
& $C+p_C^{\PP^3}+q_1+q_2+q_3$
& $26$
\end{longtable}
\end{footnotesize}
\endgroup
Here $L,L_1,L_2$ are lines, with $L_1$ and $L_2$ skew, $Z_7$ is the length-$7$ punctual scheme of Samuel multiplicity $9$ occurring in the reduced-line de Jonqui\`eres family, and $R_{-2}$ is a locally complete-intersection ribbon of degree two and arithmetic genus $-2$. Moreover, $T$ is a twisted cubic, $C$ is a smooth conic, $R_4$ is a rational normal quartic, and $E_5$ is an elliptic normal quintic. The notation $p^{\PP^a}$ denotes the first-order linear subscheme encoding the conditions imposed at $p$ by a tangent space $\PP^a$; this subscheme need not coincide with the full local base scheme. The notation $p_C^{\PP^3}$ denotes the corresponding condition at a point $p\in C$, with tangent hyperplane containing $T_pC$. Finally, $q,q_i$ denote general reduced points, disjoint from the remaining components of the base scheme.
\end{introthm}

The seven one-dimensional components in the smooth-pencil locus arise from birational nets on a del Pezzo surface of degree four. The Noether equalities, together with nefness and effectivity, reduce the possible cases to seven Weyl orbits. Furthermore, projecting from the unique rank-one point of the ninth family we obtain the component of quadratic Cremona transformations of $\PP^3$ whose general base scheme is a line together with three points.

In order to prove the upper bound, we begin by classifying systems having a surface in the base and systems having a rank-one base point. Then, we study the quintic discriminant of the web of quadrics. Outside the smooth-pencil locus, a multiple factor has degree at most two. A multiple hyperplane yields a line or a conic of vertices, while a multiple irreducible quadric is ruled by the planes of quadrics singular at a fixed base point; the remaining low-rank cases are excluded. Finally, explicit determinantal and Pfaffian normal forms show that every boundary system belongs to one of the nine components.

Note that Theorem~\ref{intro:families} is a construction theorem: the rows which do not occur as general base schemes in Theorem~\ref{intro:nine-components} are boundary strata rather than general points of irreducible components. The Magma implementation at the end of the paper provides independent residual-degree checks for the displayed systems.

For the quadro-quadric cases, that is when $d_3=2$, we use the classification recalled in \cite{PirioRusso14}. In dimension four the general base schemes are given by a smooth quadric surface together with a point, a plane together with two skew lines meeting the plane, and the schematic union of the degree-three scheme $L^H$, supported on a line contained in a hyperplane, with a smooth conic whose tangent line at the intersection point is contained in the same hyperplane. These three types have multidegrees $(2,2,2)$, $(2,3,2)$ and $(2,4,2)$ respectively. This is the only part of the paper in which we use the classical work of Semple--Roth and the modern treatments by Bruno--Verra and Pirio--Russo.



\subsection*{Organization of the paper}
This paper is organized as follows. In Section~\ref{sec:generalities} we recall the basic facts on multidegrees and Segre classes, and prove Theorem~\ref{intro:numerics}. In Sections~\ref{sec:punctual} and \ref{sec:zero-dimensional-geometric} we study the locus of transformations with zero-dimensional base scheme. In Section~\ref{sec:curves} we construct the families with one-dimensional base scheme, while in Section~\ref{sec:curve-components} we determine the seven components of the smooth-pencil locus and the ninth component. In Section~\ref{sec:surfaces} we consider the constructions with a surface in the base scheme. Finally, in Section~\ref{sec:upper-bound} we prove the upper bound and Theorem~\ref{thm:nine-components}, and in Section~\ref{appendix:magma} we describe the computational checks.

\subsection*{Acknowledgments}
The authors would like to thank Alberto Calabri and Francesco Russo for several useful discussions and suggestions.

\section{Quadratic maps and Segre classes}\label{sec:generalities}

We work over $\kk=\CC$. Set
$V=H^0(\PP^4,\OO_{\PP^4}(2))$. Then $\dim V=15$. A point $M\in\Gr(5,V)$ determines a rational map
$\vf_M:\PP^4\dashrightarrow\PP(M^*)\cong\PP^4$.
If $q_0,\dots,q_4$ is a basis of $M$, then $\vf_M=[q_0:\dots:q_4]$. Its base scheme is the closed subscheme $B_M\subset\PP^4$ defined by the ideal generated by $M$.

\begin{Lemma}\label{lem:bir-two-locally-closed}
The exact-degree-two Cremona locus $\Bir_2(\PP^4)\subset\Gr(5,V)$ is locally
closed.
\end{Lemma}

\begin{proof}
Let $G=\Gr(5,V)$ and let $\mathscr S\subset V\otimes\OO_G$ be the
tautological bundle.  On $\PP^4\times G$ the universal evaluation map
$\mathscr S\boxtimes\OO_{\PP^4}(-2)\to\OO$ defines the universal base ideal
$\mathscr I$.  The relative Proj of its Rees algebra,
$$
 \Gamma=\Proj_{\PP^4\times G}
 \bigoplus_{m\geq0}\mathscr I^m,
$$
is the universal closure of the graphs, and the universal quotient
$\mathscr S\to\mathscr I(2)$ induces a projective morphism
$\Gamma\to\PP(\mathscr S^*)$ over $G$.

Apply flattening stratification simultaneously to $\Gamma\to G$, to its
scheme-theoretic image in $\PP(\mathscr S^*)$, and to the induced finite map
on the open strata where source and image have relative dimension four.  On
each resulting locally closed stratum, the Hilbert polynomial of the image
and the generic degree of the graph over the image are constant.  The union
of the strata on which the image is $\PP^4$ and the generic degree is one is
therefore locally closed.  Finally, the locus of systems having a common
linear factor is closed: it is the image of the projective incidence
$$
 \{([h],W):W\subset hH^0(\PP^4,\OO(1))\}
 \longrightarrow G.
$$
A five-dimensional quadratic system can have no common factor of degree two,
so removing this closed locus imposes $\gcd(W)=1$.  The resulting locally
closed set is exactly $\Bir_2(\PP^4)$.
\end{proof}

Let $\vf:\PP^4\dashrightarrow\PP^4$ be a rational map. The multidegree of $\vf$ is the sequence $\multideg(\vf)=(d_0,\dots,d_4)$ where $d_i$ is the degree of the inverse image of a general codimension $i$ linear subspace of the target. Therefore, $d_0=1$. If $\vf$ is birational then $d_4=1$. If its minimal algebraic degree is
exactly two, then $d_1=2$.

The multidegrees satisfy the Cremona inequalities and the Hodge inequalities. In the present case these read
\stepcounter{thm}\begin{equation}\label{eq:cremona-hodge}
1\leq d_{i+j}\leq d_id_j,\qquad d_i^2\geq d_{i-1}d_{i+1}
\end{equation}
for all indices for which the terms are defined \cite[Chapter 7]{Dolgachev12}.

Let $\mathscr I_M\subset\OO_{\PP^4}$ be the base ideal sheaf generated by
$M\otimes\OO_{\PP^4}(-2)$.  Its saturation defines the same closed base
scheme but need not have the same Rees algebra.  We therefore use the blow-up
of the generated ideal,
$$
 \pi:X=\Proj_{\PP^4}\bigoplus_{m\geq0}\mathscr I_M^m
 \longrightarrow\PP^4,
$$
which is the closure of the graph of $\vf_M$.  If $H=\pi^*\OO_{\PP^4}(1)$
and $\mathscr I_M\OO_X=\OO_X(-E)$, the tautological quotient resolving the
five sections has divisor class $2H-E$. Hence
\stepcounter{thm}\begin{equation}\label{eq:multidegree-blow-up}
d_i=(2H-E)^i\cdot H^{4-i}.
\end{equation}
The same formula holds on any resolution dominating this graph; the notation
$E$ then denotes the total divisor of the pulled-back base ideal.

We recall the convention on Segre classes that will be used throughout the paper. Let $Z\subset \PP^4$ be a closed subscheme, and let $h=c_1(\OO_{\PP^4}(1))$. We write the push-forward of the Segre class of $Z$ in $\PP^4$ as a polynomial in $h$, $s(Z,\PP^4)=s_0(Z)+s_1(Z)h+s_2(Z)h^2+s_3(Z)h^3+s_4(Z)h^4.$ The coefficient of $h^j$ will be denoted by $s_j(Z)$, in the usual convention of \cite[Chapter 4]{Fulton84}. If $Z$ is zero-dimensional and is locally defined at $p$ by the base ideal $I_p$, then
$
s(Z,\PP^4)=\left(\sum_p e(I_p,\OO_{\PP^4,p})\right)h^4,
$
where $e$ is the Samuel multiplicity. This equals the length for reduced points and zero-dimensional complete intersections, but not for arbitrary non-reduced schemes. If $Z=C$ is a smooth curve of degree $\delta$ and genus $g$, then
$
s(C,\PP^4)=\delta h^3+(2-2g-5\delta)h^4.
$ These coefficients enter the computation of the multidegree of a rational map defined by a linear system of hypersurfaces through $Z$. We use the following form of the standard Segre class formula.

\begin{Proposition}\label{prop:segre-formula}
Let $Z\subset\PP^4$ be the base scheme of a rational map $\vf=[q_0:\dots:q_4]$ defined by quadrics and with no divisorial fixed component. Write $s(Z,\PP^4)=s_2h^2+s_3h^3+s_4h^4$ with the convention that $s_i=0$ if $\dim Z<4-i$. Then the multidegrees are
\stepcounter{thm}\begin{equation}\label{eq:segre-degrees}
d_i=2^i-\sum_{j=1}^{i}\binom{i}{j}2^{i-j}s_j(Z),
\end{equation}
where $s_j(Z)$ denotes the coefficient of $h^j$ in the Segre class polynomial.
\end{Proposition}

\begin{proof}
By \eqref{eq:multidegree-blow-up}, $d_i=(2H-E)^iH^{4-i}$. Expanding the power gives $d_i=2^iH^4+\sum_{j=1}^{i}(-1)^j\binom{i}{j}2^{i-j}H^{4-j}E^j.$ Now $H^4=1$. Furthermore, by the definition of Segre classes of a closed subscheme, the push-forward of the powers of the exceptional divisor is expressed in terms of $s(Z,\PP^4)$; this is precisely the formula in \cite[Chapter 4, Section 4.2]{Fulton84}. Substituting these push-forwards yields \eqref{eq:segre-degrees}.
\end{proof}

\begin{Remark}\label{rem:segre-birationality-criterion}
In the explicit constructions below, the Segre computation is used as a birationality criterion as follows. On the blow-up of the base ideal, the moving divisor is $L=2H-E$ and the induced morphism is defined by the resolved five-dimensional system. Formula~\eqref{eq:segre-degrees} computes $L^4$. If $L^4>0$, the image has dimension four, since the top self-intersection of the pull-back of a hyperplane is zero for a morphism with lower-dimensional image. Hence the map is dominant and generically finite. In particular, $L^4=1$ implies that its degree is one and therefore that the map is birational. Therefore, no independent dominance assumption is needed when the Segre computation gives $d_4=1$; the residual projective-degree calculations in Section~\ref{appendix:magma} provide additional checks.
\end{Remark}

\begin{Remark}\label{rem:disjoint}
If $Z=Z_1\sqcup\dots\sqcup Z_r$ is a disjoint union, then $s(Z,\PP^4)=\sum_i s(Z_i,\PP^4)$. Therefore, the contributions to the numbers $d_i$ add. This elementary observation is used repeatedly.
\end{Remark}

We will need the following elementary Segre class computations. If $C\subset\PP^4$ is a smooth curve of degree $e$ and genus $g$, then $s(C,\PP^4)=eh^3+(2-2g-5e)h^4.$ Indeed $c_1(T_{\PP^4}|_C)=5e$ and $c_1(T_C)=2-2g$, so $c_1(N_{C/\PP^4})=5e+2g-2$; then $s(C,\PP^4)=[C]-c_1(N_{C/\PP^4})[pt]$.

If $P\subset\PP^4$ is a plane, then
$N_{P/\PP^4}\cong\OO_P(1)^{\oplus2}$ and
$s(P,\PP^4)=h^2-2h^3+3h^4$. If $Q\subset\PP^4$ is a smooth quadric surface
spanning a hyperplane, then
$N_{Q/\PP^4}\cong\OO_Q(1)\oplus\OO_Q(2)$ and
$s(Q,\PP^4)=2h^2-6h^3+14h^4$.

We also use first order punctual schemes. Let $p\in\PP^4$, and let $\Lambda\subset\PP^4$ be a linear subspace through $p$ of dimension $a$. We denote by $p^\Lambda$ the first order condition that all quadrics vanish at $p$ and have tangent hyperplane containing $T_p\Lambda$. It imposes $a+1$ independent linear conditions on $V$ for general
$p,\Lambda$. Under the transversality and regular-sequence hypotheses of
Lemma~\ref{lem:punctual-contribution}, its local contribution to the
intersection of four general quadrics is $2^a$.

\begin{Lemma}\label{lem:punctual-contribution}
Let $p\in\PP^4$, and let $\Lambda\subset\PP^4$ be a linear subspace through $p$ of dimension $a$, where $0\le a\le 3$. Let $f_1,\ldots,f_4$ be four general quadrics in the given linear system. Assume that, at $p$, the linear parts of the $f_i$ span the conormal space of $\Lambda$, that is the subspace of $T_p^*\PP^4$ annihilating $T_p\Lambda$. Assume moreover that, after killing these linear parts, the remaining initial quadratic forms along $T_p\Lambda$ are general and form a regular sequence. Then the local intersection multiplicity at $p$ of $f_1,\ldots,f_4$ is $2^a$.
\end{Lemma}

\begin{proof}
We work in the completed local ring $\widehat{\mathcal O}_{\PP^4,p}$. Choose affine coordinates $u_1,\ldots,u_4$ centered at $p$ such that $\Lambda=\{u_{a+1}=\cdots=u_4=0\}.$ Therefore, $u_1,\ldots,u_a$ are coordinates along $T_p\Lambda$, while $u_{a+1},\ldots,u_4$ are normal coordinates. By assumption, the linear parts of the four quadrics span the conormal space of $\Lambda$. Hence the image of the first jet map $\langle f_1,\ldots,f_4\rangle \longrightarrow \mathfrak m_p/\mathfrak m_p^2$ is the vector space generated by $u_{a+1},\ldots,u_4$, and has dimension $4-a$. After replacing $f_1,\ldots,f_4$ by another basis of the same four-dimensional vector space, we may therefore write the equations in the form
$
g_{a+1}=u_{a+1}+r_{a+1},\, \ldots,\,
g_4=u_4+r_4,\, h_1,\ldots,h_a,
$
where each $r_j\in \mathfrak m_p^2$, and each $h_i$ has zero linear part. In other words, $h_i\in \mathfrak m_p^2,\, i=1,\ldots,a$. Now consider the first $4-a$ equations $g_{a+1},\ldots,g_4$. Their Jacobian matrix with respect to the normal variables $u_{a+1},\ldots,u_4$ is the identity at $p$. Hence, the quotient $
\kk[[u_1,\ldots,u_4]]/(g_{a+1},\ldots,g_4)$ is isomorphic to $\kk[[u_1,\ldots,u_a]]$. Concretely, these equations eliminate the normal variables $u_{a+1},\ldots,u_4$ as formal power series in $u_1,\ldots,u_a$.

Let $\overline h_1,\ldots,\overline h_a$ be the images of $h_1,\ldots,h_a$ in this quotient. Then $\overline h_i\in (u_1,\ldots,u_a)^2\subset \kk[[u_1,\ldots,u_a]]$. By the hypothesis of no further special local degeneracy, the initial homogeneous terms $q_i=\operatorname{in}(\overline h_i)$ are general quadrics in $u_1,\ldots,u_a$, and they form a regular sequence. Therefore $
(q_1,\ldots,q_a)\subset \kk[u_1,\ldots,u_a]$ is a complete intersection of type $(2,\ldots,2)$. Its length is $
\dim_\kk \kk[u_1,\ldots,u_a]/(q_1,\ldots,q_a)=2^a$.

Finally, the family obtained by deforming $\overline h_i$ to its initial form $q_i$ is flat, since the initial forms form a regular sequence. Hence the length of $\kk[[u_1,\ldots,u_a]]/(\overline h_1,\ldots,\overline h_a)$ is equal to the length of the homogeneous complete intersection defined by $q_1,\ldots,q_a$, namely $2^a$. Note that this is the local intersection multiplicity at $p$ of the four original quadrics. Therefore the local contribution is $2^a$.
\end{proof}


We collect here the numerical possibilities for quadratic Cremona transformations of $\PP^4$. The list is obtained from the projective degrees and the Segre classes of the base scheme. It provides the numerical framework for the geometric constructions and for the classification of the irreducible components carried out in the following sections.

\begin{Proposition}\label{prop:numerical-list}
Let $\vf:\PP^4\dashrightarrow\PP^4$ be a quadratic Cremona transformation. Then
$\multideg(\vf)=(1,2,d_2,d_3,1)$ and $(d_2,d_3)$ is one of the eleven pairs listed in Theorem~\ref{intro:numerics}.
\end{Proposition}

\begin{proof}
By definition $d_0=1$. Since $\vf$ is quadratic, $d_1=2$, and since $\vf$ is birational, $d_4=1$. Therefore, $
\multideg(\vf)=(1,2,d_2,d_3,1).
$
The Hodge inequalities in \eqref{eq:cremona-hodge} give $
d_1^2\ge d_0d_2,\, d_2^2\ge d_1d_3,\, d_3^2\ge d_2d_4$. From the first inequality we get $d_2\le 4$. Moreover $d_2\ge 2$: indeed $d_2=1$ would contradict $d_2^2\ge d_1d_3$, since $d_1=2$ and $d_3\ge 1$. Hence $d_2\in\{2,3,4\}$.

We are left with the bound on $d_3$. First, note that $d_3\ge 2$. Indeed $d_3$ is the degree of the inverse map $\vf^{-1}$, since the multidegree of $\vf^{-1}$ is $(1,d_3,d_2,2,1)$. If $d_3=1$, then $\vf^{-1}$ is linear, hence $\vf$ is linear, contradicting $d_1=2$.

Now use the middle Hodge inequality $d_2^2\ge 2d_3$. If $d_2=2$, then $4\ge 2d_3$, hence $d_3\le 2$. Since $d_3\ge 2$, we get $d_3=2$. If $d_2=3$, then $9\ge 2d_3$, hence $d_3\le 4$. Together with $d_3\ge 2$, this gives $d_3\in\{2,3,4\}$.

If $d_2=4$, then $16\ge 2d_3$, hence $d_3\le 8$. Together with $d_3\ge 2$, this gives $
d_3\in\{2,3,4,5,6,7,8\}$. Therefore the possible pairs are precisely $
(2,2),\, (3,2),(3,3),(3,4),\,
(4,2),(4,3),(4,4),(4,5),(4,6),(4,7),(4,8)$, as claimed.
\end{proof}

The numerical argument proves only necessity. The existence part is addressed in Sections \ref{sec:punctual}, \ref{sec:curves} and \ref{sec:surfaces}.

\section{The punctual family of multidegree \texorpdfstring{$(2,4,8)$}{(2,4,8)}}\label{sec:punctual}

Let $p_1,p_2,p_3,p_4,p_5\in\PP^4$ be general points. Let $\Pi_4\subset\PP^4$ be a plane through $p_4$, and let $H_5\subset\PP^4$ be a hyperplane through $p_5$. We set
\stepcounter{thm}\begin{equation}\label{eq:punctual-Z}
Z=p_1+p_2+p_3+p_4^{\Pi_4}+p_5^{H_5}.
\end{equation}
Here $Z$ is the first-order scheme encoding the imposed linear conditions; it
is generally a proper subscheme of the actual base scheme of
$H^0(\I_Z(2))$. The first three points impose one condition each on quadrics,
the point $p_4^{\Pi_4}$ imposes three conditions, and the point $p_5^{H_5}$
imposes four conditions. Therefore, $Z$ imposes ten independent conditions for
general data, and $h^0(\I_Z(2))=5$.

\begin{Lemma}\label{lem:zero-dimensional-base}
Let $\vf:\PP^4\dashrightarrow\PP^4$ be a quadratic Cremona transformation. If $\Bs(\vf)$ is zero-dimensional, then $\multideg(\vf)=(1,2,4,8,1)$.
\end{Lemma}

\begin{proof}
Let $Z=\Bs(\vf)$. Since $Z$ is zero-dimensional, two general quadrics in the system have no fixed surface component, and three general quadrics have no fixed curve component. Hence the first three projective degrees are the degrees of complete intersections of respectively one, two, and three quadrics in $\PP^4$. Therefore
$
d_1=2,\ d_2=4,\ d_3=8.
$
Since $\vf$ is birational, $d_4=1$, and the claim follows.
\end{proof}

\begin{Proposition}\label{prop:248-family}
For general data as in \eqref{eq:punctual-Z}, the linear system $H^0(\I_Z(2))$ defines a quadratic Cremona transformation of $\PP^4$ of multidegree $(1,2,4,8,1)$. The parameter space of these data is irreducible of dimension $27$.
\end{Proposition}

\begin{proof}
Write $Z=p_1\cup p_2\cup p_3\cup q^\Pi\cup r^H$, where $p_1,p_2,p_3$ are simple points, $q^\Pi$ is a first order point with tangent plane $\Pi$, and $r^H$ is a first order point with tangent hyperplane $H$. The parameter space is a non-empty open subset of
$$
(\PP^4)^3\times\{(q,\Pi) \: | \: q\in\Pi\subset\PP^4,\ \dim\Pi=2\}\times\{(r,H) \: | \: r\in H\subset\PP^4,\ \dim H=3\}.
$$
It is irreducible, since each factor is irreducible. Its dimension is
$
3\cdot 4+(4+4)+(4+\dim(\PP^3)^*)=12+8+7=27.
$
For a general member, the actual base scheme recovers the unique support
point with local Hilbert function $(1,3,2)$ and its tangent hyperplane, the
unique support point with local Hilbert function $(1,2)$ and its tangent
plane, and the three reduced points up to their finite permutation. Hence the
map from this parameter space to $\Gr(5,V)$ is generically finite, so its
image also has dimension $27$.

The scheme $Z$ imposes $10$ linear conditions on quadrics: each simple point imposes $1$ condition, $q^\Pi$ imposes $1+\dim\Pi=3$ conditions, and $r^H$ imposes $1+\dim H=4$ conditions. Independence of these conditions is an open condition on the above parameter space. The explicit example in Proposition~\ref{prop:explicit-248} has precisely this type and satisfies $h^0(\I_Z(2))=5$; hence the same holds for general data. Moreover, the absence of further support points or positive-dimensional base components is open, and it holds for the same explicit example. At the point with prescribed tangent plane, elimination of the two independent linear equations leaves three quadratic initial forms spanning the square of the maximal ideal in two variables, so the local Hilbert function is $(1,2)$. At the point with prescribed tangent hyperplane, elimination of the single linear equation leaves four general quadratic initial forms in three variables, whose quotient has Hilbert function $(1,3,2)$. These are open rank conditions on the corresponding initial forms. Therefore, for general data, the actual base scheme is supported at the five prescribed points and has local Hilbert functions $(1),(1),(1),(1,2),(1,3,2)$.

Let $M=H^0(\I_Z(2))$. Since the actual base scheme is zero-dimensional, a general member of $M$ has degree $2$, two general members have no fixed surface component, and three general members have no fixed curve component. Therefore the first projective degrees are the complete intersection degrees $d_1=2$, $d_2=4$, and $d_3=8$.

To conclude, it is enough to compute the topological degree. Four general quadrics in $M$ form a complete intersection of degree $2^4=16$. The part supported at the base scheme has total local intersection multiplicity $15$. The initial-form rank conditions verified above are precisely the hypotheses of Lemma~\ref{lem:punctual-contribution}. Hence the three simple points contribute $1+1+1$, the point $q^\Pi$ contributes $2^{\dim\Pi}=4$, and the point $r^H$ contributes $2^{\dim H}=8$. Hence the residual intersection has degree $16-15=1$. This residual degree is the fourth projective degree $d_4$. Therefore, $d_4=1$, so the associated rational map is generically finite of degree one. Therefore it is birational, and its multidegree is $(1,2,4,8,1)$.
\end{proof}

\begin{Proposition}\label{prop:explicit-248}
The map $\vf:\PP^4\dashrightarrow\PP^4$ given by
\stepcounter{thm}\begin{equation}\label{eq:explicit-248}
\vf=[x_2x_3:x_2x_4:x_3x_4:x_1(x_3+x_4-x_2):x_0(x_1+x_2+x_3+x_4)-x_1x_2]
\end{equation}
is a Cremona transformation of multidegree $(1,2,4,8,1)$. Its inverse is given by forms of degree eight.
\end{Proposition}

\begin{proof}
Set $F_0=x_2x_3$, $F_1=x_2x_4$, $F_2=x_3x_4$, $F_3=x_1(x_3+x_4-x_2)$ and $F_4=x_0(x_1+x_2+x_3+x_4)-x_1x_2$. We begin by proving birationality. Let $[y_0:\cdots:y_4]$ be coordinates on the target and set $A=y_0y_2+y_1y_2-y_0y_1,\, B=y_0y_1+y_0y_2+y_1y_2,\, P=y_0y_1y_2 .$ Consider the five degree eight forms
\stepcounter{thm}\begin{equation}\label{eq:inverse-248}
\begin{split}
G_0&=PA(Ay_4+y_3y_0y_1),\qquad G_1=Py_3(AB+Py_3),\\
G_2&=y_0y_1A(AB+Py_3),\qquad G_3=y_0y_2A(AB+Py_3),\\
G_4&=y_1y_2A(AB+Py_3).
\end{split}
\end{equation}
They have no common factor. A direct expansion gives proportionalities $[G_0(F):G_1(F):G_2(F):G_3(F):G_4(F)]=[x_0:x_1:x_2:x_3:x_4]$ on the open set where both sides are defined, and similarly $\vf([G_0:G_1:G_2:G_3:G_4])=[y_0:y_1:y_2:y_3:y_4].$ Therefore, \eqref{eq:inverse-248} defines the inverse of \eqref{eq:explicit-248}. In particular $\vf$ is birational, and the inverse has degree $8$.

Now, let us identify the base locus. Let $Z=V(F_0,\ldots,F_4)$. From $F_0=F_1=F_2=0$ at most one among $x_2,x_3,x_4$ is non-zero. If $x_2=x_3=x_4=0$, then $F_4=x_0x_1$, so we get the two points $e_0=[1:0:0:0:0]$ and $e_1=[0:1:0:0:0]$. If only $x_2$ is non-zero, then $F_3=-x_1x_2$ and $F_4=x_0x_1+x_0x_2-x_1x_2$. Hence $x_1=0$, and then $F_4=x_0x_2$, so $x_0=0$ and we get $e_2$. The cases where only $x_3$ or only $x_4$ is non-zero give respectively $e_3$ and $e_4$. Therefore $Z$ is supported on the five coordinate points. The local first order data are read from the linear parts of the $F_i$ in affine charts centered at these points:
\begin{footnotesize}
$$
\begin{array}{c|c|c}
\text{point} & \text{linear parts} & \text{type}\\ \hline
e_0 & x_1+x_2+x_3+x_4 & p^{\PP^3}\\
e_1 & x_3+x_4-x_2,\ x_0-x_2 & p^{\PP^2}\\
e_2 & x_3,\ x_4,\ -x_1,\ x_0-x_1 & \text{simple}\\
e_3 & x_2,\ x_4,\ x_1,\ x_0 & \text{simple}\\
e_4 & x_2,\ x_3,\ x_1,\ x_0 & \text{simple}
\end{array}
$$
\end{footnotesize}
Therefore, the base scheme is supported at the five points occurring in \eqref{eq:punctual-Z}: one point with tangent hyperplane, one point with tangent plane, and three simple points. Direct reduction of the local ideals gives Hilbert functions $(1,3,2)$ at $e_0$, $(1,2)$ at $e_1$, and $(1)$ at $e_2,e_3,e_4$. In particular it is zero-dimensional; the first-order scheme in \eqref{eq:punctual-Z} is strictly smaller at $e_0$. Hence the first three projective degrees are not affected by positive-dimensional fixed components, so $d_1=2$, $d_2=4$ and $d_3=8$. Since $\vf$ is birational, $d_4=1$. Therefore $\multideg(\vf)=(1,2,4,8,1)$.
\end{proof}

\begin{Remark}\label{rem:first-order-punctual}
Let $Z$ be a zero-dimensional first order scheme supported at distinct points of $\PP^4$. Write $n_a$ for the number of support points of type $p^{\PP^a}$. Such a point imposes $a+1$ linear conditions on quadrics and, under the transversality assumptions of Lemma~\ref{lem:punctual-contribution}, contributes $2^a$ to the local intersection multiplicity of four general quadrics.

Therefore, the numerical conditions for imposing ten conditions and contributing $15$ to the fourth projective degree are $n_0+2n_1+3n_2+4n_3=10$ and $n_0+2n_1+4n_2+8n_3=15$. Subtracting gives $n_2+4n_3=5$. If $n_3=0$, then $n_2=5$, and the first equation gives $n_0+2n_1=-5$, which is impossible. Hence $n_3=1$ and $n_2=1$. Therefore $n_0+2n_1=3$, and the only numerical possibilities are $(n_0,n_1,n_2,n_3)=(3,0,1,1)$ and $(n_0,n_1,n_2,n_3)=(1,1,1,1)$.

The first type is realized by the general zero-dimensional base scheme of the family of multidegree $(2,4,8)$. The second type is only a numerical first order possibility; we do not use it as a Cremona family. The type $(1,1,1,1)$ occurs as a flat limit of punctual schemes of type $(3,0,1,1)$, obtained by letting two simple points coalesce along a prescribed tangent direction. This is only a statement about the limiting base schemes; it does not by itself imply that the limiting linear system defines a Cremona transformation.
\end{Remark}

\begin{Definition}\label{def:first-order-linear}
Let $p\in\PP^4$ and let $\Lambda\subset\PP^4$ be a linear subspace through $p$ of dimension $a$, with $0\le a\le 3$. The first order linear scheme supported at $p$ with tangent space $\Lambda$ is the subscheme denoted by $p^\Lambda$ and defined locally as follows. Choose affine coordinates $u_1,\ldots,u_4$ centered at $p$ such that $\Lambda=\{u_{a+1}=\cdots=u_4=0\}$. Then
$
p^\Lambda
$
is defined by the ideal
$
(u_{a+1},\ldots,u_4)+(u_1,\ldots,u_4)^2.
$
That is, $\mathcal O_{p^\Lambda}$ has basis $1,u_1,\ldots,u_a$, and hence $\length(p^\Lambda)=a+1$.

A zero-dimensional scheme $Z\subset\PP^4$ is called first order linear if it is a disjoint union of schemes of the form $p_i^{\Lambda_i}$, with distinct support points $p_i$ and linear subspaces $\Lambda_i\subset\PP^4$ through $p_i$.
\end{Definition}

\begin{Remark}\label{rem:first-order-linear-conditions}
A quadric $Q$ contains $p^\Lambda$ if and only if $Q(p)=0$ and the differential $dQ_p$ vanishes on $T_p\Lambda$. Therefore, $p^\Lambda$ imposes $a+1$ linear conditions on quadrics. In the notation used below, a point of type $p^{\PP^a}$ means a scheme $p^\Lambda$ with $\dim\Lambda=a$.
\end{Remark}

\section{Zero-dimensional base locus}
\label{sec:zero-dimensional-geometric}

In this section we prove that every quadratic Cremona transformation of
$\PP^4$ with zero-dimensional base locus is a specialization of the punctual
family constructed in Proposition~\ref{prop:248-family}. 

Let
$
 W\subset H^0(\PP^4,\OO_{\PP^4}(2))
$
be a five-dimensional vector space defining a quadratic rational map $\vf_W\colon \PP^4\dashrightarrow \PP(W^*)\simeq\PP^4.$ Let $p\in B_W$ be a base point. After trivializing $\mathcal O_{\PP^4}(2)$ at $p$, the first jets of the quadrics define a linear map
$
j^1_p(W)\colon W\longrightarrow \mathfrak m_p/\mathfrak m_p^2
\simeq T_p^*\PP^4.
$
We set
$
r_p(W):=\operatorname{rk} j^1_p(W).
$

\begin{Lemma}\label{lem:first-jet-normal-form}
Let $p\in B_W$, and set $r=r_p(W)$. After a projective change of coordinates sending
$
p=[1:0:0:0:0]
$
and a change of basis of $W$, one may write
$
W=
\left\langle
x_0x_1+q_1,\ldots,x_0x_r+q_r,
q_{r+1},\ldots,q_5
\right\rangle,
$
where
$
q_i\in k[x_1,x_2,x_3,x_4]_2.
$
In particular,
$
0\le r_p(W)\le 4.
$
If the rational map defined by $W$ is generically finite, then
$
1\le r_p(W)\le 4.
$
\end{Lemma}

\begin{proof}
Every quadric passing through $p=[1:0:0:0:0]$ can be written uniquely as
$
Q=x_0\ell(x_1,x_2,x_3,x_4)+q(x_1,x_2,x_3,x_4),
$
where $\ell$ is linear and $q$ is quadratic. Under this decomposition, the first-jet map is
$
Q\longmapsto \ell.
$
After choosing coordinates on $T_p\PP^4$ and a basis of $W$ adapted to the image and kernel of this map, we obtain the asserted normal form.

If $r_p(W)=0$, all the quadrics are independent of $x_0$. Hence the associated rational map factors through the projection
$
\PP^4\dashrightarrow\PP^3
$
from $p$, and therefore cannot be generically finite.
\end{proof}

The Zariski tangent space of the base scheme at $p$ is the annihilator in
$T_p\PP^4$ of the image of the first-jet map. Consequently,
$
\dim T_pB_W=4-r_p(W).
$
Therefore, the four possible ranks for a quadratic Cremona transformation have the following normal forms:
\begin{footnotesize}
$$
\begin{array}{c|c|c}
r_p(W) & \text{normal form} & \dim T_pB_W\\
\hline
4 &
\left\langle
x_0x_1+q_1,
x_0x_2+q_2,
x_0x_3+q_3,
x_0x_4+q_4,
q_5
\right\rangle
&0
\\[2mm]
3 &
\left\langle
x_0x_1+q_1,
x_0x_2+q_2,
x_0x_3+q_3,
q_4,
q_5
\right\rangle
&1
\\[2mm]
2 &
\left\langle
x_0x_1+q_1,
x_0x_2+q_2,
q_3,
q_4,
q_5
\right\rangle
&2
\\[2mm]
1 &
\left\langle
x_0x_1+q_1,
q_2,
q_3,
q_4,
q_5
\right\rangle
&3
\end{array}
$$
\end{footnotesize}
If $r_p(W)=4$, then $p$ is a reduced point of the base scheme. Indeed, the linear parts of the local equations generate
$
\mathfrak m_p/\mathfrak m_p^2,
$
and Nakayama's lemma gives
$
I_{B_W,p}=\mathfrak m_p.
$

If $r_p(W)=3$ and $p$ is isolated in $B_W$, then the local algebra has embedding dimension one and is therefore curvilinear.

For $r_p(W)=2$ or $r_p(W)=1$, the rank does not determine the local algebra. In particular, it does not determine its length, Hilbert function, or Samuel multiplicity. For instance, the general punctual transformation considered in this paper has a rank-two block with Hilbert function
$
(1,2)
$
and a rank-one block with Hilbert function
$
(1,3,2),
$
but other local structures may have the same first-jet ranks.

The rank also admits a useful geometric interpretation. Let
$
\pi\colon \widetilde{\PP^4}=\operatorname{Bl}_p(\PP^4)\longrightarrow\PP^4
$
be the blow-up of $p$, with exceptional divisor
$
E\simeq \PP(T_p\PP^4)\simeq\PP^3.
$
The strict transforms of the quadrics belong to
$
|2H-E|,
$
whose restriction to $E$ is $\mathcal O_E(1)$. The restricted linear system is generated by the first jets of the quadrics. In the normal form of Lemma~\ref{lem:first-jet-normal-form}, it induces the linear projection
$
E\simeq\PP^3
\dashrightarrow
\PP^{r-1},
\,
[v_1:v_2:v_3:v_4]
\longmapsto
[v_1:\dots : v_r].
$
Its base locus is
$
\PP^{3-r}=\{v_1=\cdots=v_r=0\}\subset E,
$
with the convention that it is empty when $r=4$. After resolving this linear base locus, the image of $E$ is a linear space
$
\PP^{r-1}\subset\PP^4,
$
and the general fiber has dimension
$
4-r.
$

Therefore, $r_p(W)$ measures both the tangent dimension of the base scheme at $p$ and the dimension of the image of the exceptional divisor over $p$. It determines the linear part of the local normal form, but not the higher-order terms $q_i$ and hence not the complete local structure of $B_W$.

\begin{Lemma}\label{lem:strict-transform-general-plane}
Let
$
 \vf\colon \PP^4\dashrightarrow\PP^4
$
be a birational map, and let $\Pi\subset\PP^4$ be a general plane in the
target.  Then the strict transform $S_\Pi$ of $\Pi$ is an irreducible surface,
and
$
 \vf|_{S_\Pi}\colon S_\Pi\dashrightarrow\Pi
$
is birational.

Assume moreover that $\vf$ is defined by a linear system $W$, that $A\subset W$ is the two-dimensional subspace corresponding
to the two hyperplanes cutting out $\Pi$, and that
$
 X_A=\bigcap_{Q\in A}V(Q)
$
is an integral surface.  Then
$
 X_A=S_\Pi
$
scheme-theoretically.
In particular, the rational map induced on $X_A$ by any complement of $A$ in
$W$ is birational onto $\Pi$.
\end{Lemma}

\begin{proof}
Choose dense open subsets $U$ in the source and $U'$ in the target such that
$
 \vf|_U\colon U\xrightarrow{\sim}U'
$
is an isomorphism.  Since $\Pi$ is general, $\Pi\cap U'$ is a dense open
subset of $\Pi$.  Its inverse image in $U$ is therefore irreducible of
dimension two and isomorphic to $\Pi\cap U'$.  By definition,
$
 S_\Pi=\overline{\vf^{-1}(\Pi\cap U')},
$
so $S_\Pi$ is irreducible and the restriction of $\vf$ to $S_\Pi$ is
birational onto $\Pi$.

Now suppose that $\Pi$ is cut out by the two linear forms on $\PP(W^*)$
corresponding to a basis of $A$.  On the open set where $\vf$ is defined, the
inverse image of $\Pi$ is cut out by the two corresponding members of $A$.
Hence
$
 S_\Pi\subset X_A.
$
Both schemes have dimension two. Since $X_A$ is irreducible, the inclusion
identifies the two supports. The defining ideal of $S_\Pi$ inside the reduced
scheme $X_A$ therefore has zero radical and is nilpotent; reducedness of
$X_A$ forces it to be zero. Thus the equality is scheme-theoretic. The last
assertion follows since, on $X_A$, the two coordinates belonging to $A$
vanish and the remaining coordinates give precisely the restriction of
$\vf$ to $X_A$.
\end{proof}

\begin{Lemma}\label{lem:smooth-two-quadric-section}
Let $W\subset H^0(\PP^4,\OO_{\PP^4}(2))$ have zero-dimensional base locus.
Assume that
$
 r_p(W)\geq 2
$
for every $p\in\Supp(\Bs(W))$.  Then, for a general two-dimensional subspace
$A\subset W$, the complete intersection
$
 X_A=\bigcap_{Q\in A}V(Q)
$
is a smooth irreducible surface of type $(2,2)$ in $\PP^4$.
\end{Lemma}

\begin{proof}
For a fixed base point $p$, consider the closed determinantal locus
$$
 \Sigma_p=
 \left\{
 A\in\Gr(2,W)\mid
 \rk\bigl(A\longrightarrow\mathfrak m_p/\mathfrak m_p^2\bigr)<2
 \right\}.
$$
Since $r_p(W)\geq2$, the locus $\Sigma_p$ is a proper closed subset of
$\Gr(2,W)$.  The support of the base scheme is finite, hence a general
$A\subset W$ lies outside every $\Sigma_p$.  For such an $A$, the
differentials of two generators of $A$ are linearly independent at every base
point.  Therefore $X_A$ is smooth at the base locus.

Outside the base locus the linear system $\PP(W)$ is base-point-free.  By
Bertini, two general members meet transversely there.  Therefore, $X_A$ is smooth
everywhere.  It is a complete intersection of two ample divisors in $\PP^4$,
so it is connected.  Since a smooth connected scheme has only one irreducible
component, $X_A$ is irreducible.
\end{proof}

\begin{Proposition}\label{prop:geometric-rank-one-point}
Let
$
 \vf_W\colon\PP^4\dashrightarrow\PP^4
$
be a quadratic Cremona transformation with zero-dimensional base locus. Then
there exists a point $p\in\Bs(W)$ such that
$
 r_p(W)=1.
$
\end{Proposition}

\begin{proof}
First, note that a base point of rank zero is impossible.  Indeed, send
such a point to $p=[1:0:0:0:0]$.  Since every quadric in $W$ vanishes at $p$
and has zero differential there, no member of $W$ contains a monomial
$x_0x_i$.  Hence every member depends only on $x_1,\ldots,x_4$.  Geometrically,
all the quadrics are cones with vertex $p$, and the map is constant along the
lines through $p$.  It is therefore not generically finite.

Suppose now, by contradiction, that
$
 r_p(W)\geq2
$
for every base point. Choose a general two-dimensional subspace
$A\subset W$.  By Lemma~\ref{lem:smooth-two-quadric-section},
$
 X=X_A
$
is a smooth irreducible complete intersection of two quadrics.  The subspace
$A$ corresponds to a plane $\Pi\subset\PP(W^*)$ in the target.  We may choose
$A$ generally enough that $\Pi$ is general with respect to the isomorphism
locus of $\vf_W$.  Lemma~\ref{lem:strict-transform-general-plane} then gives
$
 \vf_W|_X\colon X\dashrightarrow\Pi\simeq\PP^2
$
birational.

Choose a direct sum decomposition
$
 W=A\oplus B,
 \, \dim B=3.
$
On $X$ the two quadrics in $A$ vanish, and the net $B|_X\subset
|\OO_X(2)|$ defines the above birational map to $\Pi$.  This net has no fixed
curve: a fixed curve would be contained in the zero locus of all five
quadrics in $W$, contradicting the zero-dimensionality of $\Bs(W)$.

Let
$
 \pi\colon Y\longrightarrow X
$
be a sequence of blow-ups resolving all proper and infinitely near base points
of the net, and $m_1,\ldots,m_s$ be their multiplicities and let
$E_1,\ldots,E_s$ denote the corresponding total exceptional classes.  If
$H=\OO_X(1)$, the movable transform of the net is
$
 M=\pi^*(2H)-\sum_{i=1}^s m_iE_i.
$
The induced morphism
$
 g\colon Y\longrightarrow\PP^2
$
is birational and $M=g^*\OO_{\PP^2}(1)$.  Consequently,
$
 M^2=1,
 \,
 K_Y\cdot M=-3.
$

We have
$
 H^2=4,
 \,
 K_X=-H,
$ and $
 K_Y=\pi^*K_X+\sum_{i=1}^sE_i.
$ The two preceding intersection equalities therefore become
$
 16-\sum_{i=1}^s m_i^2=1
$
and
$
 -8+\sum_{i=1}^s m_i=-3.
$
Therefore, \stepcounter{thm}\begin{equation}\label{eq:noether-del-pezzo-four}
 \sum_{i=1}^s m_i^2=15,
 \qquad
 \sum_{i=1}^s m_i=5.
\end{equation}
These equations have no solution in positive integers.  Indeed, no $m_i$ can
be at least four, while if one multiplicity is three, the remaining
multiplicities have sum two and their squares sum at most four, giving a total
at most thirteen.  If every multiplicity is at most two, then
$
 \sum_i m_i^2\leq2\sum_i m_i=10.
$
Both alternatives contradict \eqref{eq:noether-del-pezzo-four}.

Hence some base point has first-jet rank at most one.  Rank zero has already
been excluded, so its rank is exactly one.
\end{proof}

\begin{Proposition}\label{prop:geometric-reduction-p3}
Let $W$ define a quadratic Cremona transformation of $\PP^4$, and let
$p\in\Bs(W)$ satisfy $r_p(W)=1$. After a projective change of coordinates in
the source and a change of basis of $W$, one can write
\stepcounter{thm}\begin{equation}\label{eq:rank-one-normal-form-geometric}
 W=
 \left\langle
 x_0x_1+q_0,
 q_1,q_2,q_3,q_4
 \right\rangle,
 \qquad
 q_i\in\kk[x_1,x_2,x_3,x_4]_2.
\end{equation}
Moreover, $V=\langle q_1,q_2,q_3,q_4\rangle\subset
H^0(\PP^3,\OO_{\PP^3}(2))$ defines a quadratic Cremona transformation
$\psi_V\colon\PP^3\dashrightarrow\PP^3$. If $\Bs(W)$ is zero-dimensional,
then $\Bs(V)$ is zero-dimensional.
\end{Proposition}

\begin{proof}
The normal form \eqref{eq:rank-one-normal-form-geometric} follows from Lemma~\ref{lem:first-jet-normal-form}. Let
$
 \pi_p\colon\PP^4\dashrightarrow\PP^3
$
be the projection from $p$, and let
$
 o=[1:0:0:0:0]
$
in the target.  Projection from $o$ gives a commutative diagram
$$
 \begin{tikzcd}[column sep=1.5cm, row sep=0.7cm]
 \PP^4 \arrow[dashed]{r}{\vf_W} \arrow[dashed]{d}[swap]{\pi_p}
 & \PP^4 \arrow[dashed]{d}{\pi_o} \\
 \PP^3 \arrow[dashed]{r}{\psi_V}
 & \PP^3.
 \end{tikzcd}
$$
Indeed, the last four coordinates of $\vf_W$ depend only on the direction of
the line through $p$.

For a general direction
$
 \xi=[x_1:x_2:x_3:x_4]\in\PP^3
$
with $x_1\neq0$, the line $\ell_\xi$ through $p$ is mapped linearly and
birationally onto the line through $o$ corresponding to $\psi_V(\xi)$.  In
coordinates, on $\ell_\xi$ the first coordinate is linear in the parameter
along the line, with non-zero coefficient $x_1$, whereas the last four
coordinates have the common quadratic factor coming from that parameter.
Therefore, the restriction
$
 \vf_W|_{\ell_\xi}\colon\ell_\xi\dashrightarrow
 \ell_{\psi_V(\xi)}
$
has degree one.

The map $\psi_V$ is dominant, since otherwise the image of $\vf_W$ would be
contained in the cone with vertex $o$ over a subvariety of dimension at most
two, and would therefore have dimension at most three.  Since the maps on the
general fibers of $\pi_p$ and $\pi_o$ have degree one, the generic degrees of
$\vf_W$ and $\psi_V$ are equal.  The former is one, hence $\psi_V$ is
birational.

Assume now that $\Bs(W)$ is zero-dimensional. To conclude, it is enough to prove that
$\Bs(V)$ is zero-dimensional. Suppose that it
contains an irreducible curve $C$.  If $C$ is not contained in
$\{x_1=0\}$, then, over the dense open subset where $x_1\neq0$, the equation
$
 x_0x_1+q_0=0
$
selects one point on each line through $p$ with direction in $C$.  The closure
of these points is a curve contained in $\Bs(W)$, a contradiction.

Assume instead that $C\subset\{x_1=0\}$.  If $q_0$ vanishes identically on
$C$, the cone over $C$ with vertex $p$ is contained in $\Bs(W)$.  Otherwise,
$q_0|_C$ is a non-zero section of the positive-degree line bundle
$\OO_C(2)$, and therefore it vanishes at some point $c\in C$.  The whole line
$\overline{pc}$ is then contained in $\Bs(W)$.  This is again impossible.
Hence $\Bs(V)$ is finite.
\end{proof}

Let
$
 \mathscr H_{\mathrm{III}}
 \subset
 \Gr\bigl(4,H^0(\PP^3,\OO_{\PP^3}(2))\bigr)
$
denote the third irreducible component in the classification of quadratic Cremona
transformations of $\PP^3$ by Pan--Ronga--Vust \cite{PanRongaVust01}.  Its general member is the
complete system of quadrics through three reduced points and through a fourth
point with a prescribed tangent plane.

\begin{Lemma}\label{lem:prv-finite-component}
Let
$
 V\subset H^0(\PP^3,\OO_{\PP^3}(2))
$
define a quadratic Cremona transformation with zero-dimensional base locus.
Then
$
 V\in\mathscr H_{\mathrm{III}}.
$
Moreover, $\mathscr H_{\mathrm{III}}$ is irreducible, and its general member
has base data
$
 a_1+a_2+a_3+b^{\Pi_b},
$
where $a_1,a_2,a_3,b$ are distinct points and $\Pi_b\subset\PP^3$ is a plane
through $b$.
\end{Lemma}

\begin{proof}
This is precisely the finite-base case of
\cite[Propositions~2.1.1 and~2.4.1]{PanRongaVust01}.  Proposition~2.1.1
places every quadratic Cremona transformation of $\PP^3$ with finite base
locus in the third family, and Proposition~2.4.1 proves that this family is an
irreducible component.  The description of its general member is the type
$\mathrm{III}$ in their classification.
\end{proof}

\begin{Proposition}\label{prop:geometric-lift-prv}
Let $W$ define a quadratic Cremona transformation of $\PP^4$ with
zero-dimensional base locus.  Then
$
 W\in\overline{\mathcal F}_{248},
$
where $\mathcal F_{248}$ is the punctual family of
Proposition~\ref{prop:248-family}.
\end{Proposition}

\begin{proof}
By Proposition~\ref{prop:geometric-rank-one-point}, $W$ has a base point of
first-jet rank one.  By Proposition~\ref{prop:geometric-reduction-p3}, after
changing coordinates we may write
$
 W=\langle x_0x_1+q_0,V\rangle
$
as in \eqref{eq:rank-one-normal-form-geometric}, where $V$ is a quadratic
Cremona system of $\PP^3$ with finite base locus.  By
Lemma~\ref{lem:prv-finite-component},
$
 V\in\mathscr H_{\mathrm{III}}.
$

Consider the irreducible parameter space
$
 \mathscr P=
 \mathscr H_{\mathrm{III}}
 \times H^0(\PP^3,\OO_{\PP^3}(2)),
$
whose points will be written $(V',q'_0)$.  Fix the hyperplane
$
 H_1=\{x_1=0\}\subset\PP^3.
$
There is a non-empty open subset $\mathscr P^\circ\subset\mathscr P$ on which
the following conditions hold:
\begin{enumerate}
\item $V'$ is a general type $\mathrm{III}$ system, with base data
$
 a_1+a_2+a_3+b^{\Pi_b};
$
\item the four support points are distinct and lie outside $H_1$;
\item after lifting the four points along the lines through
$p=[1:0:0:0:0]$ by the equation
$
 x_0x_1+q'_0=0,
$
the resulting three reduced points, the resulting point with tangent plane,
and the point $p$ with tangent hyperplane $\{x_1=0\}\subset\PP^4$ impose ten
independent conditions on quadrics.
\end{enumerate}
All these conditions are open.  Their simultaneous non-emptiness follows, for
instance, from the explicit member in Proposition~\ref{prop:explicit-248},
after choosing its point of first-jet rank one as $p$.

Since $\mathscr P$ is irreducible, the point $(V,q_0)$ belongs to the closure
of $\mathscr P^\circ$.  By the algebraic curve-selection lemma, after an
eventual finite base change there exist a smooth pointed curve $(T,0)$ and a
family
$
 (V_t,q_0(t))\in\mathscr P,
 \,
 (V_0,q_0(0))=(V,q_0),
$
whose general member belongs to $\mathscr P^\circ$.  After shrinking $T$, the
pullback of the tautological rank-four bundle is trivial, so we may choose a
basis
$
 V_t=\langle q_1(t),q_2(t),q_3(t),q_4(t)\rangle.
$
Set
$
 W_t=
 \left\langle
 x_0x_1+q_0(t),
 q_1(t),q_2(t),q_3(t),q_4(t)
 \right\rangle.
$
The coefficient of $x_0x_1$ shows that these subspaces have dimension five,
and $W_0=W$.

For general $t$, the map defined by $V_t$ is a Cremona transformation of
$\PP^3$.  The same projection diagram used in
Proposition~\ref{prop:geometric-reduction-p3} shows that the map defined by
$W_t$ is birational: $V_t$ recovers the direction of the line through $p$,
and the first coordinate recovers the unique point on that line.

We will now describe the base data of $W_t$.  Since the support of $\Bs(V_t)$
avoids $H_1$, the equation $x_0x_1+q_0(t)=0$ cuts each of the four lines over
these base points in exactly one point.  At each of the three reduced base
points of $V_t$, the three independent horizontal first jets, together with
the transverse first jet in the direction of the line through $p$, give a
reduced base point of $W_t$.  At the point $b^{\Pi_b}$, the first jets of
$V_t$ span a one-dimensional conormal space in $T_b^*\PP^3$; adding the
transverse first jet of $x_0x_1+q_0(t)$ gives a two-dimensional conormal space
in $T^*\PP^4$, hence a point with a prescribed tangent plane.  Finally, at
$p$ the only non-zero first jet is $x_1$, so $p$ is a point with prescribed
tangent hyperplane $\{x_1=0\}$.

There are no further base points.  Away from the lines over $\Bs(V_t)$, one
of $q_1(t),\ldots,q_4(t)$ is non-zero.  On each line over a base point of
$V_t$, the first equation has a unique zero since $x_1\neq0$.  Since
$\Bs(V_t)$ does not meet $H_1$, no additional line through $p$ is contained in
the base locus.

Let $Z_t$ be the first-order punctual scheme consisting of these three reduced
points, the point with tangent plane, and the point with tangent hyperplane.
By construction,
$
 W_t\subset H^0(\I_{Z_t}(2)).
$
Condition (iii) gives $h^0(\I_{Z_t}(2))=5$, hence equality holds.  Therefore
$W_t$ belongs to the family $\mathcal F_{248}$ for general $t$.  Since
$W_0=W$, we conclude that
$
 W\in\overline{\mathcal F}_{248}.
$
\end{proof}

\begin{thm}\label{thm:zero-dimensional-unique-component-geometric}
Let
$
 \mathcal C^0_{248}
 \subset
 \Gr\bigl(5,H^0(\PP^4,\OO_{\PP^4}(2))\bigr)
$
be the locus of quadratic Cremona transformations with zero-dimensional base
locus.  Then
$
 \overline{\mathcal C^0_{248}}
 =
 \overline{\mathcal F}_{248}.
$
In particular, $\mathcal C^0_{248}$ is irreducible, and its general member is
the punctual transformation of Proposition~\ref{prop:248-family}. 
\end{thm}

\begin{proof}
Proposition~\ref{prop:geometric-lift-prv} gives
$
 \mathcal C^0_{248}\subset\overline{\mathcal F}_{248}.
$
Conversely, the general member of $\mathcal F_{248}$ is a quadratic Cremona
transformation with zero-dimensional base locus by
Proposition~\ref{prop:248-family}.  Hence
$
 \overline{\mathcal C^0_{248}}
 =
 \overline{\mathcal F}_{248}.
$
The family $\mathcal F_{248}$ is irreducible, hence
$\overline{\mathcal C^0_{248}}=\overline{\mathcal F}_{248}$ is
irreducible. Inside the locally closed space $\Bir_2(\PP^4)$, the locus
where the base dimension is at least one is closed by upper semicontinuity.
Hence $\mathcal C^0_{248}$ is open in its closure and is therefore
irreducible.
\end{proof}

\section{One-dimensional base schemes}\label{sec:curves}

In this section we construct the families with $d_2=4$ and one-dimensional base scheme. We begin with the Segre class computations that will be used throughout the section.

\begin{Lemma}\label{lem:curve-segre}
Let $C\subset\PP^4$ be a smooth curve of degree $e$ and genus $g$. Then
\stepcounter{thm}\begin{equation}\label{eq:curve-segre}
s(C,\PP^4)=e h^3+(2-2g-5e)h^4.
\end{equation}
In particular, if $C$ is the only positive-dimensional component of the base scheme of a quadratic map of $\PP^4$, then
\stepcounter{thm}\begin{equation}\label{eq:curve-degrees}
d_3=8-e,\qquad d_4=14-3e+2g-\delta,
\end{equation}
where $\delta$ is the contribution of the zero-dimensional part and of the imposed first order conditions away from the curve contribution.
\end{Lemma}

\begin{proof}
The normal sequence of $C\subset\PP^4$ gives
$
c(N_{C/\PP^4})=c(T_{\PP^4}|_C)c(T_C)^{-1}.
$
Since $c_1(T_{\PP^4}|_C)=5h|_C$ and $\deg c_1(T_C)=2-2g$, we get
$
\deg c_1(N_{C/\PP^4})=5e+2g-2.
$
Therefore, $
s(C,\PP^4)=c(N_{C/\PP^4})^{-1}\cap [C]
=[C]-c_1(N_{C/\PP^4})\cap [C],
$
which gives \eqref{eq:curve-segre}. Formula \eqref{eq:curve-degrees} follows directly from
\eqref{eq:segre-degrees}. For a base scheme with no surface component, only
the coefficients $s_3$ and $s_4$ enter in the last two degrees, and
\stepcounter{thm}\begin{equation}\label{eq:segre-degrees-curves}
d_3=8-s_3,\qquad d_4=16-8s_3-s_4.
\end{equation}
Therefore, $d_3=8-e$ and $d_4=14-3e+2g$ before adding the zero-dimensional contributions.
\end{proof}

We will also use the following first order contribution. If $p\in\PP^4$ and $\Lambda\simeq\PP^a\subset\PP^4$ is a linear subspace through $p$, then the first order linear scheme $p^\Lambda$ imposes $a+1$ conditions on quadrics and contributes $2^a$ to the fourth projective degree, as in Lemma~\ref{lem:punctual-contribution}. If $p$ lies on a smooth curve $C$ and $T_pC\subset\Lambda$, then the condition of containing $C$ already imposes the tangent direction $T_pC$, hence $p^\Lambda$ imposes only $a-1$ further conditions on quadrics containing $C$.

\begin{Lemma}\label{lem:curve-embedded-point}
Let $C\subset\PP^4$ be a smooth curve, let $p\in C$, and let $H\subset\PP^4$ be a hyperplane such that $T_pC\subset H$. Let $Z=C\cup p^H$, where $p^H$ is the first order linear scheme supported at $p$ with tangent space $H$. Then
$
 s(Z,\PP^4)=s(C,\PP^4)+4h^4.
$
Therefore, the embedded first order condition contributes $4$ to the coefficient of $h^4$.
\end{Lemma}

\begin{proof}
Choose completed local coordinates $(t,x,y,z)$ at $p$ such that
$
 I_C=(x,y,z),\qquad H=(z),
$
and $t$ is a parameter on $C$. Since $I_{p^H}=(z,(t,x,y,z)^2)$, one has
\stepcounter{thm}\begin{equation}\label{eq:embedded-curve-local-ideal}
 I_Z=I_C\cap I_{p^H}=(z,tx,ty,x^2,xy,y^2).
\end{equation}
Let $\pi_1:X_1=\Bl_C\PP^4\to\PP^4$ and denote its exceptional divisor by $E$. On the $x$-chart, where $y=xY$ and $z=xZ$, the pull-back of
\eqref{eq:embedded-curve-local-ideal} is
$
 x(Z,t,x).
$
On the $z$-chart the residual ideal is the unit ideal. Thus
$
 I_Z\mathcal O_{X_1}=\mathcal O_{X_1}(-E)I_\lambda,
$
where $\lambda\subset E_p\simeq\PP^2$ is the line determined by $H$.

Blow up $\lambda$ and write
$
 \pi_2:X_2=\Bl_\lambda X_1\longrightarrow X_1,
 \qquad F=\operatorname{Exc}(\pi_2),\qquad A=\pi_2^*E.
$
Then the divisor resolving $I_Z$ is $A+F$, whereas the divisor resolving
$I_C$ is $A$. For a blow-up of a smooth codimension-three center,
\stepcounter{thm}\begin{equation}\label{eq:codim-three-blowup-push}
 \pi_{2*}F=\pi_{2*}F^2=0,\qquad
 \pi_{2*}F^3=[\lambda],\qquad
 \pi_{2*}F^4=i_*c_1(N_{\lambda/X_1}).
\end{equation}
Expanding through codimension four and using \eqref{eq:codim-three-blowup-push} gives
\begin{align*}
\pi_{2*}\left(
 \frac{A+F}{1+A+F}-\frac{A}{1+A}
\right)
&=\pi_{2*}(F^3-4AF^3-F^4)\\
&=[\lambda]-4A\cdot[\lambda]-i_*c_1(N_{\lambda/X_1}).
\end{align*}
The one-cycle $[\lambda]$ pushes forward to zero in $\PP^4$. Moreover,
$E|_\lambda\simeq\mathcal O_\lambda(-1)$, so
$\deg(A|_\lambda)=-1$. Finally, the normal sequences give
$
0\longrightarrow\mathcal O_\lambda(1)\oplus\mathcal O_\lambda
\longrightarrow N_{\lambda/X_1}
\longrightarrow\mathcal O_\lambda(-1)\longrightarrow0,
$
whence $\deg c_1(N_{\lambda/X_1})=0$. After pushing to $\PP^4$ the only
remaining term is therefore $4[p]$. The formula for Segre classes through a
principalization, $s(Z,\PP^4)=\pi_*\bigl(D/(1+D)\bigr)$, now yields
$
 s(C\cup p^H,\PP^4)=s(C,\PP^4)+4h^4.
$
\end{proof}

We need one non-reduced curve. Let $L\subset H\subset\PP^4$ be a line contained in a hyperplane. We denote by $L^H$ the degree three scheme supported on $L$ with fixed tangent hyperplane $H$, defined by
$
\I_{L^H}=\I_H+\I_L^2.
$
For instance, if $L=\{x_0=x_1=x_2=0\}$ and $H=\{x_0=0\}$, then
$
\I_{L^H}=(x_0,x_1^2,x_1x_2,x_2^2).
$

\begin{Lemma}\label{lem:triple-line-segre}
For $L^H\subset\PP^4$ as above,
\stepcounter{thm}\begin{equation}\label{eq:triple-line-segre}
s(L^H,\PP^4)=4h^3-20h^4.
\end{equation}
Moreover $L^H$ imposes $7$ independent conditions on quadrics.
\end{Lemma}

\begin{proof}
In the coordinates above, the degree two part of $\I_{L^H}$ is generated by
$
x_0x_i,\ i=0,\ldots,4,\, x_1^2,\, x_1x_2,\, x_2^2.
$
Therefore, $h^0(\I_{L^H}(2))=8$, and $L^H$ imposes $15-8=7$ conditions on quadrics.

For the Segre class, blow up first the line $L$ and let $E_1$ be the exceptional divisor. Then $E_1\simeq L\times\PP^2$. The hyperplane $H$ determines inside $E_1$ the divisor
$
Q=E_1\cap \widetilde H\simeq L\times\PP^1.
$
Blow up $Q$ and let $E_2$ be the second exceptional divisor. If $E_1'$ denotes the strict transform of the first exceptional divisor, then the total transform of the first exceptional divisor is $\overline E_1=E_1'+E_2$, and the divisor resolving the scheme $L^H$ is $E_1'+2E_2=\overline E_1+E_2$. Let $i:Q\hookrightarrow X_1=\Bl_L\PP^4$ and let
$\pi_2:X_2\to X_1$ be the second blow-up. The standard blow-up formulas
give
$
\pi_{2*}E_2^2=-[Q],\quad
\pi_{2*}E_2^3=-i_*c_1(N_{Q/X_1}),\quad
\pi_{2*}E_2^4=-i_*\bigl(c_1(N_{Q/X_1})^2-c_2(N_{Q/X_1})\bigr).
$
Here
$N_{Q/X_1}=\OO_Q(\widetilde H)\oplus\OO_Q(E_1)$, while for the first
blow-up $N_{L/\PP^4}=\OO_L(1)^{\oplus3}$. Expanding
$(\overline E_1+E_2)^j$, applying these identities, and then pushing forward
through $\Bl_L\PP^4\to\PP^4$ gives
$
\pi_*((\overline E_1+E_2)^3)=4h^3,\, \pi_*((\overline E_1+E_2)^4)=20h^4.
$
With the usual Segre sign convention this gives \eqref{eq:triple-line-segre}.
\end{proof}

\begin{Remark}\label{rem:triple-line-samuel}
For a smooth reduced curve the coefficient of $h^3$ in its Segre class is
its degree. This need not hold for a non-reduced curve. At the generic point
of $L$, the ideal of $L^H$ is $(x_0,x_1^2,x_1x_2,x_2^2)$ in a
three-dimensional regular local ring, and three general elements have
colength, equivalently Samuel multiplicity, $4$. This explains why the
coefficient of $h^3$ in \eqref{eq:triple-line-segre} is $4$, although the
fundamental cycle of $L^H$ has degree $3$.
\end{Remark}

\begin{Lemma}\label{lem:triple-line-conic-segre}
Let $p\in L\subset H\subset\PP^4$, and let $C$ be a smooth conic meeting
$L$ only at $p$, with $T_pC\subset H$ and $T_pC\neq T_pL$. Then
$
 s(L^H\cup C,\PP^4)=6h^3-33h^4.
$
Moreover, for general such data the union imposes ten independent conditions
on quadrics.
\end{Lemma}

\begin{proof}
The only non-additive contribution is local at $p$. In completed local
coordinates $(a,b,c,z)$ one may take
$
 I_{L^H}=(a,b^2,bc,c^2),\qquad I_C=(c,z,a-b^2).
$
A direct intersection gives
\stepcounter{thm}\begin{equation}\label{eq:triple-line-conic-local}
 I_{L^H\cup C}=(a-b^2,bc,c^2,b^2z).
\end{equation}
Blow up $L=(a,b,c)$ and use the $b$-chart $a=bA$, $c=bC$. The residual
center resolving $L^H$ is $Q=(A,b)$, namely
$E_L\cap\widetilde H$. Blow up $Q$ and, on the $A$-chart, write
$b=A\beta$ and $\eta=\beta-1$. Near the strict transform $\widetilde C$,
where $\beta$ is a unit, \eqref{eq:triple-line-conic-local} becomes
$
 A^2(\eta,C,z).
$
Consequently, on the twofold blow-up $X_2$ one has
\stepcounter{thm}\begin{equation}\label{eq:triple-line-conic-residual}
 I_{L^H\cup C}\mathcal O_{X_2}
 =\mathcal O_{X_2}(-D_T)I_{\widetilde C},
 \qquad D_T\cdot\widetilde C=2,
\end{equation}
where $D_T$ is the divisor resolving $L^H$.

Blow up the smooth codimension-three curve $\widetilde C$ and let $F$ be the
exceptional divisor. Comparing the divisors $D_T+F$ and $D_T$, exactly as in
Lemma~\ref{lem:curve-embedded-point}, gives
\stepcounter{thm}\begin{equation}\label{eq:triple-line-conic-difference}
 s(L^H\cup C,\PP^4)-s(L^H,\PP^4)
 =[C]-\bigl(\deg c_1(N_{\widetilde C/X_2})
        +4D_T\cdot\widetilde C\bigr)h^4.
\end{equation}
For a smooth curve meeting a smooth blow-up center transversely at one point,
the normal bundle of the strict transform is the elementary transform of the
old normal bundle and its degree drops by $\codim(Z)-1$. The blow-up of $L$
has codimension three and the blow-up of $Q$ has codimension two; hence
$
 \deg c_1(N_{\widetilde C/X_2})
 =\deg c_1(N_{C/\PP^4})-2-1=8-3=5.
$
Together with \eqref{eq:triple-line-conic-residual}, formula
\eqref{eq:triple-line-conic-difference} becomes
$
 s(L^H\cup C)-s(L^H)=[C]-13h^4.
$
Since $[C]=2h^3$ and $s(L^H)=4h^3-20h^4$, this is
$6h^3-33h^4$. Relative to the disjoint sum
$s(L^H)+s(C)$, the local correction is indeed $-5h^4$.

For independence, take coordinates $[x_0:\dots:x_4]$ and set
$
 L=V(x_0,x_1,x_2),\quad H=V(x_0),\quad
 C:[s:t]\longmapsto[t^2:st:0:s^2:0].
$
Then
\begin{align*}
H^0(\I_{L^H\cup C}(2))
=\langle x_0x_2,x_0x_4,x_1x_2,x_2^2,
          x_0x_3-x_1^2\rangle,
\end{align*}
and one checks directly that
\begin{align*}
&\langle x_0x_2,x_0x_4,x_1x_2,x_2^2,x_0x_3-x_1^2\rangle^{\rm sat}\\
&\hspace{15mm}=(x_0,x_1^2,x_1x_2,x_2^2)
 \cap(x_2,x_4,x_0x_3-x_1^2).
\end{align*}
Thus the union imposes exactly ten conditions in this rational coordinate
model, and the same holds for general data by openness.
\end{proof}

\begin{Proposition}\label{prop:curve-families}
The following irreducible families consist generically of quadratic Cremona transformations of $\PP^4$ with the stated multidegree:
\begin{enumerate}
\item \textit{$L+p_L^{\PP^3}+q^{\PP^2}+r+s$, of multidegree $(2,4,7)$ and dimension $25$;}
\item \textit{$C+p_C^{\PP^3}+q_1+q_2+q_3$, where $C$ is a smooth conic, of multidegree $(2,4,6)$ and dimension $26$;}
\item \textit{$L_1\sqcup L_2+q^{\PP^2}+r$, where $L_1,L_2$ are skew lines, of multidegree $(2,4,6)$ and dimension $24$;}
\item \textit{$C+p^{\PP^2}$, where $C$ is a twisted cubic and $p$ is general, of multidegree $(2,4,5)$ and dimension $24$;}
\item \textit{$C\sqcup L+q_1+q_2$, where $C$ is a smooth conic and $L$ is a line skew to $\langle C\rangle$, of multidegree $(2,4,5)$ and dimension $25$;}
\item \textit{$L^H+q_1+q_2+q_3$, of multidegree $(2,4,4)$ and dimension $20$;}
\item \textit{$C+q$, where $C\subset\PP^4$ is a rational normal quartic and $q$ is general, of multidegree $(2,4,4)$ and dimension $25$;}
\item \textit{$E$, where $E\subset\PP^4$ is an elliptic normal quintic, of multidegree $(2,4,3)$ and dimension $25$.}
\item \textit{$L^H+C_{L_p}$, where $C_{L_p}$ is a conic through $p\in L$ whose tangent line at $p$ is contained in $H$, of multidegree $(2,4,2)$ and dimension $16$.}
\end{enumerate}
\end{Proposition}

\begin{proof}
We begin with the parameter spaces. Each family is an open subset of an
irreducible incidence variety. For (i), the choices give
$
6+1+2+(4+4)+4+4=25.
$
For (ii), a smooth conic depends on $6+5=11$ parameters, and the marked point,
the hyperplane containing its tangent, and the three further points give
$1+2+12$, for a total of $26$. For (iii) the count is
$12+8+4=24$; for (iv), $16+8=24$; for (v), $11+6+8=25$; for (vi),
$6+2+12=20$; for (vii), $21+4=25$; and elliptic normal quintics form an
irreducible family of dimension $25$. In (ix), the flag
$p\in L\subset H$ contributes $9$ parameters and the conic contributes
$7$, for a total of $16$. The support and the distinguished tangent
conditions recover the general datum, up to the finite permutation of the
unlabelled reduced points. Hence these are also the dimensions of the images
in the Grassmannian.

We next prove that the required open conditions are non-empty over $\QQ$.
This also explains precisely the meaning of the two marked embedded
conditions in (i) and (ii).

For (i), set
$
L=V(x_0,x_1,x_2),\quad p=[0:0:0:1:0],\quad H=V(x_0),
$
$
q=[1:0:0:0:0],\quad
\Lambda=\PP\langle e_0,e_1,e_2\rangle,
$
and take
$
r=[2:1:2:-2:1],\quad s=[1:3:1:-1:-1].
$
The five quadrics
\begin{align*}
&19x_0x_3+35x_0x_4+6x_1^2,
&&4x_0x_3+5x_0x_4+3x_1x_2,\\
&-5x_0x_3-13x_0x_4+6x_1x_4,
&&x_0x_3+x_2^2,
&&-x_0x_4+x_2x_4
\end{align*}
vanish on the marked configuration. If $x_0=0$, their reduced common zero
set is $L$. On the chart $x_0=1$, elimination gives
$
 x_4(x_4+1)(2x_4-1)=0
$
and the three corresponding points are $q,s,r$, respectively. The points
$r,s$ are reduced, while at $q$ the completed ideal is
$
(x_3,x_4,x_1^2,x_1x_2,x_2^2),
$
so the marked tangent plane is exact. At $p$, put $t=x_4/x_3$ and eliminate
$x_0$ by the fourth quadric. The ideal
$
I=(x_1^2,x_1x_2,x_2^2,tx_1,tx_2)
$
of $L\cup p^H$ contains the ideal generated by the four remaining initial
quadrics
\begin{align*}
J=(6x_1^2-19x_2^2,
   3x_1x_2-4x_2^2,
   6tx_1+5x_2^2,
   tx_2).
\end{align*}
A direct multiplication gives $I^2=JI$. Thus $J$ is a reduction of $I$ and
the two ideals have the same normalized blow-up and the same local Segre
class. Along $L\setminus\{p\}$ the first-jet matrix has rank three. This
proves non-emptiness of (i) and identifies its integral-closure model.

For (ii), take
$
C=V(x_3,x_4,x_0x_2-x_1^2)
$
and
\begin{align*}
W=\langle &x_0x_2-x_1^2,
 x_3(x_1-x_2),x_4(x_1-x_2),\\
 &x_3(x_3-x_2),x_4(x_4-x_2)\rangle.
\end{align*}
The last four quadrics define on the projection $\PP^3$ the line
$V(x_3,x_4)$ and the three reduced points
$
[1:1:1:0],\ [1:1:0:1],\ [1:1:1:1].
$
Over the line the first quadric cuts out $C$, and over these three points it
selects respectively
$
[1:1:1:1:0],\ [1:1:1:0:1],\ [1:1:1:1:1].
$
Thus the support is exactly the conic and the three points. At
$p=[1:0:0:0:0]\in C$, after eliminating $x_2-x_1^2$, the marked ideal is
$
I=(x_1x_3,x_1x_4,x_3^2,x_3x_4,x_4^2),
$
whereas the four remaining quadrics have initial ideal
$
J=(x_1x_3,x_1x_4,x_3^2,x_4^2).
$
The missing generator is integral over $J$, since
$(x_3x_4)^2=x_3^2x_4^2\in J^2$. Hence $J$ is a reduction of $I$. This is the
precise local meaning of $C+p_C^{\PP^3}$ in the table.

Cases (iii), (iv), (v), (vii), and (viii) have a uniform geometric witness.
Let
$
X\simeq\Bl_{z_1,\ldots,z_5}\PP^2\subset\PP^4,
\quad H=3\ell-e_1-\cdots-e_5,
$
be a general smooth complete intersection of two quadrics and let
$A=H^0(\I_X(2))$. Use, respectively, the following movable classes, fixed
divisors, and residual multiplicities:
\begin{center}
\begin{footnotesize}
\begin{tabular}{c|c|c}
case & $D$ & $2H-D$ and residual multiplicities \\
\hline
(iii) & $3\ell-e_1-e_2-e_3$
 & $L_0+(\ell-e_4-e_5)$, $(2,1)$ \\
(iv) & $3\ell-e_1-e_2-e_3-e_4$
 & $H-e_5$, $(2)$ \\
(v) & $2\ell-e_1$
 & $L_0+(2\ell-e_2-e_3-e_4-e_5)$, $(1,1)$ \\
(vii) & $2\ell-e_1-e_2$
 & $4\ell-e_1-e_2-2e_3-2e_4-2e_5$, $(1)$ \\
(viii) & $\ell$
 & $5\ell-2e_1-\cdots-2e_5$, $-$
\end{tabular}
\end{footnotesize}
\end{center}
Here $L_0=2\ell-e_1-\cdots-e_5$. For each row let
$\mathcal N\subset H^0(X,\mathcal O_X(2H))$ be the complete net obtained by
adding the displayed fixed divisor to the indicated proper homaloidal net.
After the marked blow-downs the movable parts have types
$(3;2,1^4)$, $(3;2,1^4)$, $(2;1^3)$,
$(2;1^3)$, and $(1;0)$; the Hudson reductions show that they are
birational. Since $X$ is a complete intersection of two quadrics, the
restriction map
$$
H^0(\PP^4,\mathcal O_{\PP^4}(2))
\longrightarrow H^0(X,\mathcal O_X(2H))
$$
is surjective with kernel $A$. Its inverse image $W$ of $\mathcal N$ has
vector dimension five and contains $A$. Consequently every base point of
$W$ lies on $X$, and the restriction of the base scheme to $X$ is exactly
the fixed divisor and residual cluster displayed in the corresponding row.
The fixed divisors are, respectively, two disjoint lines, a twisted cubic,
a line together with a conic, a rational quartic, and an elliptic quintic;
the residual multiplicities give precisely cases (iii), (iv), (v), (vii),
and (viii). This provides exact characteristic-zero witnesses without a
randomized computation.

For (vi), with $L=V(x_0,x_1,x_2)$ and $H=V(x_0)$, take
\begin{align*}
q_1&=[1:2:2:0:1],&q_2&=[1:0:2:2:1],&q_3&=[1:-1:0:-1:2].
\end{align*}
The complete quadratic system is generated by
\begin{align*}
&x_0(2x_0+x_1-2x_2+x_3),\quad
 x_0(4x_0-x_2-2x_4),\\
&6x_0^2+4x_0x_1-3x_0x_2-2x_1^2,\quad
 (x_0+x_1)(2x_0-x_2),\quad x_2(2x_0-x_2).
\end{align*}
Successive elimination gives exactly
$
(I_{L^H}\cap I_{q_1}\cap I_{q_2}\cap I_{q_3})
$
after saturation. For (ix), the exact coordinate model and saturation are
those in Lemma~\ref{lem:triple-line-conic-segre}. These witnesses prove that
all ten conditions are independent and that no additional support occurs.

We now compute the multidegrees. In (i) and (ii), replacing the generated
base ideal by the displayed reduction does not change the Segre class; in the
other cases the displayed scheme is the actual complete base scheme. Since
there is no surface component, $d_1=2$ and $d_2=4$. The relevant classes are
\begin{footnotesize}
$$
\begin{array}{c|c|c}
\text{case} & s(Z,\PP^4) & \multideg \\ \hline
\textup{(i)} & h^3+7h^4 & (1,2,4,7,1)\\
\textup{(ii)} & 2h^3-h^4 & (1,2,4,6,1)\\
\textup{(iii)} & 2h^3-h^4 & (1,2,4,6,1)\\
\textup{(iv)} & 3h^3-9h^4 & (1,2,4,5,1)\\
\textup{(v)} & 3h^3-9h^4 & (1,2,4,5,1)\\
\textup{(vi)} & 4h^3-17h^4 & (1,2,4,4,1)\\
\textup{(vii)} & 4h^3-17h^4 & (1,2,4,4,1)\\
\textup{(viii)} & 5h^3-25h^4 & (1,2,4,3,1)\\
\textup{(ix)} & 6h^3-33h^4 & (1,2,4,2,1)
\end{array}
$$
\end{footnotesize}
Cases (i)--(viii) follow from Lemmas~\ref{lem:curve-segre},
\ref{lem:curve-embedded-point}, \ref{lem:triple-line-segre}, and
\ref{lem:punctual-contribution}, together with additivity for disjoint
components; case (ix) is Lemma~\ref{lem:triple-line-conic-segre}. Substitution
in
$
 d_3=8-s_3,\qquad d_4=16-8s_3-s_4
$
gives $d_4=1$ in every row. Remark~\ref{rem:segre-birationality-criterion}
then proves dominance and birationality.
\end{proof}

Proposition~\ref{prop:curve-families} constructs families. It does not assert that these are the only families with the corresponding multidegrees. In particular, when two families have the same multidegree, we do not assert here that they are distinct irreducible components of the whole Cremona locus.

\begin{Remark}
We recall the standard description of family \textup{(viii)} in Proposition~\ref{prop:curve-families}. Let $\mathsf A$ be a $5\times5$ skew-symmetric matrix of linear forms on $\PP^4$. The five $4\times4$ Pfaffians of $\mathsf A$ are quadrics. For a general $\mathsf A$ they define a smooth elliptic normal quintic $E\subset\PP^4$; conversely, every elliptic normal quintic arises in this way. This is the Buchsbaum--Eisenbud structure theorem for codimension three Gorenstein ideals \cite{BuchEis77}.

Since $E$ has degree $5$ and genus $1$, Lemma~\ref{lem:curve-segre} gives
$
s(E,\PP^4)=5h^3-25h^4.
$
Therefore, \eqref{eq:segre-degrees-curves} gives $d_3=8-5=3$ and $d_4=16-40+25=1$. Therefore the five Pfaffians define a quadratic Cremona transformation of multidegree $(1,2,4,3,1)$.
\end{Remark}

\begin{Lemma}\label{lem:delpezzo-line-ribbon}
Let $X\subset\PP^4$ be a smooth complete intersection of two quadrics, let
$H=\mathcal O_X(1)=-K_X$, and let $L\subset X$ be a line. Then $H-L$ is
base-point-free, has square $1$, and defines a birational morphism $\beta_L:X\longrightarrow\PP^2.$ Consequently $|2H-2L|=\beta_L^*|\mathcal O_{\PP^2}(2)|$.
\end{Lemma}

\begin{proof}
After choosing a suitable blow-down marking, we may write
$
X\simeq\Bl_{z_1,\ldots,z_5}\PP^2,\,
H=3\ell-e_1-\cdots-e_5,
\, L=2\ell-e_1-\cdots-e_5.
$
Then $H-L=\ell$, and the complete linear system $|H-L|$ is the pull-back of
lines in $\PP^2$. This proves the statement.
\end{proof}

\begin{Proposition}\label{prop:ribbon-family}
There exists an irreducible family of dimension $26$ whose general base
scheme is $B=R_{-2}\sqcup\{q_1,q_2,q_3\},$ where $R_{-2}$ is a locally
complete-intersection ribbon of degree two and arithmetic genus $-2$
supported on a line. The corresponding complete system of quadrics defines
a quadratic Cremona transformation with
$P_{R_{-2}}(t)=2t+3,\, P_B(t)=2t+6,\,
\multideg=(1,2,4,6,1)$.
\end{Proposition}

\begin{proof}
Fix a line $L\subset\PP^4$. A ribbon $R$ supported on $L$ with square-zero
ideal $\mathcal O_L(1)$ is determined by a surjection
\stepcounter{thm}\begin{equation}\label{eq:ribbon-surjection}
N^*_{L/\PP^4}=\mathcal O_L(-1)^{\oplus3}
\twoheadrightarrow\mathcal O_L(1),
\end{equation}
up to scalar. For a general surjection its kernel has degree $-4$ and the
balanced splitting
$
\mathcal K\simeq\mathcal O_L(-2)^{\oplus2}.
$
The defining ideal fits into
$$
0\longrightarrow\I_L^2(2)\longrightarrow\I_R(2)
\longrightarrow\mathcal K(2)\longrightarrow0.
$$
The Koszul resolution of $\I_L^2$ gives
$H^1(\I_L^2(2))=0$ and $h^0(\I_L^2(2))=6$. Hence
$$
H^0(\I_R(2))\twoheadrightarrow H^0(\mathcal K(2))\simeq\kk^2,
\qquad h^0(\I_R(2))=8.
$$
Choose two lifts $Q_0,Q_1$ of a basis of $H^0(\mathcal K(2))$. Their normal
differentials are pointwise independent along $L$. Thus
$X=V(Q_0,Q_1)$ is smooth along $L$; general choices of the lifts and Bertini
make $X$ a smooth complete intersection away from $L$ as well. The conormal
sequence identifies
$$
N^*_{L/X}\simeq
N^*_{L/\PP^4}/\mathcal K\simeq\mathcal O_L(1),
$$
and the inverse image of $\mathcal K$ in $\I_L$ is therefore the Cartier
double divisor $2L$ on $X$. In particular, the general ribbon determined by
\eqref{eq:ribbon-surjection} occurs as $R=2L$ on a smooth $(2,2)$ surface.
Adjunction on $X$, where $K_X=-H$, gives
$$
p_a(R)=1+\frac{(2L)^2+K_X\cdot2L}{2}=-2,
$$
so
$$
0\longrightarrow\mathcal O_L(1)\longrightarrow\mathcal O_R
\longrightarrow\mathcal O_L\longrightarrow0,
\qquad P_R(t)=2t+3.
$$

Now choose three general ordered points $q_1,q_2,q_3$ disjoint from $R$ and
put
$$
W=H^0\bigl(\I_{R\cup\{q_1,q_2,q_3\}}(2)\bigr).
$$
They impose three independent conditions on the eight-dimensional space
$H^0(\I_R(2))$, so $\dim W=5$. The condition that the composite
$$
W\longrightarrow H^0(\I_R(2))
\longrightarrow H^0(\mathcal K(2))
$$
be surjective is open. It is non-empty: choose first a smooth surface $X$
as above and then three points of $X\setminus L$ whose images under the
birational morphism $\beta_L:X\to\PP^2$ of
Lemma~\ref{lem:delpezzo-line-ribbon} are non-collinear. Hence it holds for
general $(R,q_1,q_2,q_3)$. Choose $Q_0,Q_1\in W$ mapping to a basis. Then
$X=V(Q_0,Q_1)$ is again a smooth $(2,2)$ surface containing the three points,
and restriction gives
$$
0\longrightarrow H^0(\I_X(2))\longrightarrow W
\longrightarrow
H^0\bigl(X,\mathcal O_X(2H-2L)\otimes\I_{q_1+q_2+q_3}\bigr)
\longrightarrow0.
$$
The last vector space has dimension three. Under
$|H-L|=\beta_L^*|\mathcal O_{\PP^2}(1)|$ it is the complete net of conics
through the three non-collinear points $\beta_L(q_i)$, hence the standard
quadratic Cremona net. It follows that the base of $W$ is exactly
$R\sqcup\{q_1,q_2,q_3\}$ and that the induced map of $\PP^4$ is birational.

The curve Segre formula gives
$
 s(R,\PP^4)=2h^3-4h^4.
$
After adding the three reduced points,
$s(B,\PP^4)=2h^3-h^4$, so
\eqref{eq:segre-degrees-curves} gives $d_3=6$ and $d_4=1$.

Finally, the surjections in \eqref{eq:ribbon-surjection} form a non-empty
open subset of a $\PP^8$-bundle over the six-dimensional Grassmannian of
lines. The three ordered disjoint points contribute $12$ parameters. Thus
the parameter space is irreducible of dimension $6+8+12=26$. The general
complete base scheme recovers its ribbon and the three reduced points up to
finite permutation, so the map to the Grassmannian is generically finite and
its image has dimension $26$.
\end{proof}

\begin{Lemma}\label{lem:binary-cubics-dejonquieres}
Let $K=(f_0,f_1,f_2)\subset\kk[[s,t]]$, where every $f_i$ has order three
and the initial cubics span a three-dimensional space without a common
projective zero. Then
$
H_{\kk[[s,t]]/K}=(1,2,3,1),\quad
\length(\kk[[s,t]]/K)=7,\quad e(K)=9,
$
and the integral closure of $K$ is $(s,t)^3$.
\end{Lemma}

\begin{proof}
Let $S=\kk[s,t]$, let $J$ be the ideal generated by the three initial
cubics, and let $\mathfrak m=(s,t)$. Hilbert--Burch gives
$$
0\longrightarrow S(-5)\oplus S(-4)
\longrightarrow S(-3)^3\longrightarrow J\longrightarrow0.
$$
Thus $S/J$ has Hilbert function $(1,2,3,1)$, length $7$, and
$\mathfrak m^4\subset J$. We claim that the full initial ideal of $K$ equals $J$. Let
$f=\sum_i a_i f_i$ be an arbitrary element of $K$, with
$a_i\in\kk[[s,t]]$, and let $m=\min_i\operatorname{ord}(a_i)$. If the terms
of degree $m+3$ do not cancel, the initial form of $f$ is a homogeneous
combination of the three initial cubics and belongs to $J$. If they cancel,
then $\operatorname{ord}(f)\geq m+4\geq4$; since
$\mathfrak m^4\subset J$, its initial form again belongs to $J$. Thus every
initial form of an element of $K$ lies in $J$, while the reverse inclusion is
clear. Hence $\operatorname{in}(K)=J$, proving the Hilbert function and the
length.

Two general elements of $K$ form a minimal reduction. Their initial cubics
have no common point on $\PP^1$, so their intersection multiplicity is
$3\cdot3=9$. Therefore $e(K)=9=e(\mathfrak m^3)$. Since
$K\subset\mathfrak m^3$ are $\mathfrak m$-primary ideals in the
formally equidimensional regular local ring $\kk[[s,t]]$, Rees's multiplicity
theorem \cite{Rees61} implies that $\mathfrak m^3$ is integral over $K$.
Thus $K$ and $\mathfrak m^3$ have the same integral closure.
\end{proof}

\begin{Proposition}\label{prop:dejonquieres-family}
There exists an irreducible family of dimension $24$ whose general base
scheme has positive-dimensional part a reduced line $L$ and isolated
residual scheme supported at two points $p,q\notin L$. At $p$ the local
algebra has Hilbert function $(1,2,3,1)$, length $7$, and Hilbert--Samuel
multiplicity $9$, while $q$ is reduced. The Hilbert polynomial and
multidegree are $P_B(t)=t+9,\, \multideg=(1,2,4,7,1).$ \end{Proposition}

\begin{proof}
Let $X=\Bl_{z_1,\ldots,z_5}\PP^2$, $H=-K_X=3\ell-e_1-\cdots-e_5,$ where the points $z_i$ are general. The anticanonical system embeds $X$ as a
smooth complete intersection of two quadrics in $\PP^4$. The class $L=2\ell-e_1-\cdots-e_5$ has $H\cdot L=1$ and $L^2=-1$, hence its image is a line. Set $D=2H-L=4\ell-e_1-\cdots-e_5.$ For general points $p,q\in X\setminus L$, the net $\PP H^0(\OO_X(D)\otimes\I_p^3\otimes\I_q)$ is the plane homaloidal system of type $(4;3,1^6)$. Its Hudson reductions are $(4;3,1^6)\longmapsto(3;2,1^4) \longmapsto(2;1^3)\longmapsto(1;0),$ so it defines a birational map \cite{Alberich02,Hudson27}.

Let $A=H^0(\I_X(2))$. Restriction gives
$$
0\longrightarrow A\longrightarrow H^0(\I_L(2))
\longrightarrow H^0(X,2H-L)\longrightarrow0.
$$
The inverse image of the preceding net is a five-dimensional system $W$ of
quadrics. Since $A\subset W$, the base scheme is contained in $X$. On $X$
its positive-dimensional part is the reduced line $L$, and its residual base
scheme is supported at $p$ and $q$. At $p$, after eliminating the two smooth
equations of $X$, the local ideal is generated by three general binary
cubics, so Lemma~\ref{lem:binary-cubics-dejonquieres} applies. At $q$ the
base point is simple.

The Segre class of the reduced line is $s(L,\PP^4)=h^3-3h^4.$ For an isolated base ideal, the coefficient of its local zero-dimensional
Segre class is its Hilbert--Samuel multiplicity. Hence the points $p$ and $q$
contribute respectively $9h^4$ and $h^4$, and $s(B,\PP^4)=h^3+7h^4.$ Formula~\eqref{eq:segre-degrees-curves} yields $d_3=7$ and $d_4=1$.
Remark~\ref{rem:segre-birationality-criterion} therefore shows that the
associated map is birational, with multidegree
$(1,2,4,7,1)$. The Hilbert polynomial is $P_B(t)=(t+1)+7+1=t+9.$ Finally, choose first the line $L$ and then a pencil of quadrics through it
whose intersection is smooth. Since $h^0(\I_L(2))=12$, pairs $(X,L)$ form an
irreducible family of dimension $6+\dim\Gr(2,12)=6+20=26.$ The points $p,q\in X$ contribute four further parameters. For a fixed
resulting system $W$, the forgotten pencil $A=H^0(\I_X(2))\subset W$ varies
in an open subset of $\Gr(2,5)$, of dimension $6$. Therefore, the image in the
Grassmannian has dimension $26+4-6=24$.
\end{proof}

\section{Smooth pencils and irreducible components}\label{sec:curve-components}

In this section, starting from the explicit constructions in Section~\ref{sec:curves}, we study the irreducible components of the corresponding loci. First, we consider the open locus of quadratic Cremona systems containing a pencil whose complete intersection is a smooth del Pezzo surface $X$ of degree four. Restricting to $X$, the residual web induces a birational net, and the Noether equalities together with nefness and effectivity reduce the possible divisor classes to seven $W(D_5)$-orbits, yielding the seven components of the smooth-pencil locus. Then, projecting from a rank-one base point and using the classification of Pan--Ronga--Vust in $\PP^3$, we prove that the conic family gives a further component, which is distinct both from the punctual component and from the smooth-pencil components.

Set $G=\Gr(5,V)$, and let $\mathscr S$ be the tautological bundle on $G$.
Consider the relative Grassmannian $\rho:\Gr_G(2,\mathscr S)\to G$. A point
of $\Gr_G(2,\mathscr S)$ is a pair $(W,A)$ with $A\subset W$ and $\dim A=2$.
Let $\Omega$ be the locus of pairs for which $V(A)\subset\PP^4$ is a smooth
complete intersection of two quadrics, and set
$\mathscr U=\rho(\Omega)\cap\Bir_2(\PP^4)$. Finally, set
$\mathscr U_1=\{W\in\mathscr U\mid\dim\Bs(W)=1\}$.

\begin{Lemma}\label{lem:smooth-pencil-open}
The locus $\mathscr U$ is open in $\Bir_2(\PP^4)$. If $W\in\mathscr U_1$ and
$A\subset W$ is general among the smooth pencils, then $X=V(A)$ is a smooth
del Pezzo surface of degree four and the net induced by $W/A$ on $X$ is
birational to $\PP^2$.
\end{Lemma}

\begin{proof}
Smoothness of $V(A)$ is open, hence $\Omega$ is open. The morphism $\rho$ is
smooth and therefore open, proving the first assertion. Smooth pencils form a
non-empty open subset of $\Gr(2,W)$. We may thus choose $A$ so that the plane
in $\PP(W^*)$ corresponding to $A$ meets the isomorphism locus of the
Cremona transformation. The surface $X$ is smooth, connected, and hence
irreducible. Lemma~\ref{lem:strict-transform-general-plane} identifies it
with the strict transform of that plane and proves the last assertion.
\end{proof}

We classify $\mathscr U_1$. Let $W\in\mathscr U_1$, choose $A$ as in
Lemma~\ref{lem:smooth-pencil-open}, and set $H=\OO_X(1)=-K_X$. If $C$ is the
fixed divisorial part of the induced birational net, its movable part has
class $D=2H-C$. Resolve its proper and infinitely near base points, with
multiplicities $m_1,\ldots,m_s$. Since the resolved net is the pull-back of
lines in $\PP^2$, one has
\stepcounter{thm}\begin{equation}\label{eq:del-pezzo-net}
D^2-\sum_jm_j^2=1,\qquad H\cdot D-\sum_jm_j=3.
\end{equation}
Choose a marking
$X\simeq\Bl_{z_1,\ldots,z_5}\PP^2$, with
$H=3\ell-e_1-\cdots-e_5$, and write $D=a\ell-\sum_{i=1}^5b_i e_i$.
Then \eqref{eq:del-pezzo-net} are the Noether equalities for the proper plane
homaloidal type $(a;b_1,\ldots,b_5,m_1,\ldots,m_s)$. Moreover, $D$ is nef,
$2H-D$ is effective. Since $\ell$ is nef and $2H-D$ is a non-zero effective divisor, one has $(2H-D)\cdot\ell=6-a\geq0$; hence $a\leq6$.

For the finite check below we use the notation
$[a;b_1,\ldots,b_5\mid m_1,\ldots,m_s]$, and write $-$ when the right-hand
side is empty.

\begin{Lemma}\label{lem:del-pezzo-enumeration}
Up to the action of $W(D_5)$, the possible movable classes and residual
multiplicities are exactly the seven rows in the following table. Here
$L_0=2\ell-e_1-\cdots-e_5$.

\begin{center}
\footnotesize
\begin{tabular}{c|c|c|c}
$D$ & $C=2H-D$ & $(m_j)$ & general base scheme \\
\hline
$4\ell-\sum e_i$ & $L_0$ & $(3,1)$ & $L+Z_7+q$ \\
$2\ell$ & $2L_0$ & $(1,1,1)$ & $R_{-2}+q_1+q_2+q_3$ \\
$3\ell-e_1-e_2-e_3$ & $L_0+(\ell-e_4-e_5)$ & $(2,1)$ & $L_1\sqcup L_2+p^{\PP^2}+q$ \\
$3\ell-e_1-e_2-e_3-e_4$ & $H-e_5$ & $(2)$ & $T+p^{\PP^2}$ \\
$2\ell-e_1$ & $L_0+(2\ell-e_2-e_3-e_4-e_5)$ & $(1,1)$ & $C\sqcup L+q_1+q_2$ \\
$2\ell-e_1-e_2$ & $4\ell-e_1-e_2-2e_3-2e_4-2e_5$ & $(1)$ & $R_4+q$ \\
$\ell$ & $5\ell-2\sum e_i$ & $-$ & $E_5$
\end{tabular}
\end{center}
\end{Lemma}

\begin{proof}
The Hudson test gives the following proper homaloidal types of degree at most
six:
\begin{footnotesize}
$$
\begin{aligned}
&(1;0),\ (2;1^3),\ (3;2,1^4),\ (4;3,1^6),\ (4;2^3,1^3),\\
&(5;4,1^8),\ (5;3,2^3,1^3),\ (5;2^6),\\
&(6;5,1^{10}),\ (6;4,2^4,1^3),\ (6;3^3,2,1^4),\ (6;3^2,2^4,1).
\end{aligned}
$$
\end{footnotesize}
Choose five multiplicities, allowing zeros, for the marked points
$z_1,\ldots,z_5$; the remaining ones are the $m_j$. Nefness is equivalent to
$b_i\geq0$, $a-b_i-b_j\geq0$, and $2a-\sum b_i\geq0$, since the $16$
$(-1)$-curves have classes $e_i$, $\ell-e_i-e_j$, and
$2\ell-e_1-\cdots-e_5$. The effective cone is generated by the same curves.
Checking the twelve types gives the following complete list of surviving
assignments; each row is one Weyl orbit.

\begin{center}
\begin{footnotesize}
\begin{tabular}{c|l}
standard representative & all surviving assignments \\
\hline
$[1;0^5\mid-]$ &
$[1;0^5\mid-],\ [2;1^3,0^2\mid-],\ [3;2,1^4\mid-]$ \\[1mm]
$[2;0^5\mid1^3]$ &
$[2;0^5\mid1^3],\ [4;2^3,0^2\mid1^3],\ [6;4,2^4\mid1^3]$ \\[1mm]
$[2;1,0^4\mid1^2]$ &
$[2;1,0^4\mid1^2],\ [3;2,1^2,0^2\mid1^2],\ [4;3,1^4\mid1^2],\ [4;2^3,1,0\mid1^2],\ [5;3,2^3,1\mid1^2]$ \\[1mm]
$[2;1^2,0^3\mid1]$ &
$[2;1^2,0^3\mid1],\ [3;2,1^3,0\mid1],\ [4;2^3,1^2\mid1]$ \\[1mm]
$[3;1^3,0^2\mid2,1]$ &
$[3;1^3,0^2\mid2,1],\ [4;2^2,1^2,0\mid2,1],\ [5;3,2^2,1^2\mid2,1],\ [6;3^2,2^3\mid2,1]$ \\[1mm]
$[3;1^4,0\mid2]$ &
$[3;1^4,0\mid2],\ [4;2^2,1^3\mid2],\ [5;2^5\mid2]$ \\[1mm]
$[4;1^5\mid3,1]$ &
$[4;1^5\mid3,1],\ [5;2^3,1^2\mid3,1],\ [6;3,2^4\mid3,1]$
\end{tabular}
\end{footnotesize}
\end{center}

Quadratic reflection in the first three marked points sends
$(a;b_1,b_2,b_3,b_4,b_5)$ to
$
(2a-b_1-b_2-b_3; a-b_2-b_3, a-b_1-b_3,a-b_1-b_2, b_4, b_5).
$
Together with permutations, these reflections reduce every row to its
standard representative. The residual multiplicity multisets distinguish the
seven representatives. Translating the standard representatives into
$C=2H-D$ gives the table in the statement. The geometric identifications
follow from the intersection numbers on $X$ and from
Lemma~\ref{lem:binary-cubics-dejonquieres}; the corresponding general members
are precisely the families constructed in
Propositions~\ref{prop:curve-families}, \ref{prop:ribbon-family}, and
\ref{prop:dejonquieres-family}.
\end{proof}

We will need the following standard persistence statement.

\begin{Lemma}\label{lem:persistence-components}
Let $\mathcal H$ be a projective family in the Hilbert scheme of $\PP^4$.
The locus of $W\in G$ such that $\Bs(W)$ contains a member of $\mathcal H$
is closed. The same holds for a projective family of cycles in the Chow
variety.
\end{Lemma}

\begin{proof}
The incidence of pairs $(W,Z)$ satisfying $Z\subset\Bs(W)$ is closed in
$G\times\mathcal H$. Its projection to $G$ is closed since $\mathcal H$ is
projective. The proof for cycles is identical.
\end{proof}

\begin{Proposition}\label{prop:seven-smooth-pencil-components}
The locus $\mathscr U_1$ has exactly seven irreducible components. Their
general base schemes, multidegrees, and dimensions are

\begin{center}
\begin{footnotesize}
\begin{tabular}{c|c|c}
multidegree & general base scheme & dimension \\
\hline
$(2,4,7)$ & $L+Z_7+q$ & $24$ \\
$(2,4,6)$ & $R_{-2}+q_1+q_2+q_3$ & $26$ \\
$(2,4,6)$ & $L_1\sqcup L_2+p^{\PP^2}+q$ & $24$ \\
$(2,4,5)$ & $T+p^{\PP^2}$ & $24$ \\
$(2,4,5)$ & $C\sqcup L+q_1+q_2$ & $25$ \\
$(2,4,4)$ & $R_4+q$ & $25$ \\
$(2,4,3)$ & $E_5$ & $25$
\end{tabular}
\end{footnotesize}
\end{center}

\end{Proposition}

\begin{proof}
Fix one row of Lemma~\ref{lem:del-pezzo-enumeration}. Marked smooth del
Pezzo surfaces of degree four form an irreducible family. On such a surface the fixed divisor varies in the projective space
$|2H-D|$, after removing the forced $(-1)$-components, and an ordered cluster
of residual proper or infinitely near points varies in the irreducible space
obtained by successive universal blow-ups. The dimensions of these spaces of
fixed divisors are constant in each row by Riemann--Roch and the standard
vanishing obtained after removing the forced $(-1)$-components. In particular,
in the fourth row $C=H-e_5$ moves in a projective plane; it is not rigid. Let $\pi:Y\to X$ resolve the cluster. The resolved net of
a system $W\in\mathscr U_1$ is $M=g^*\OO_{\PP^2}(1)$ for a birational
morphism $g:Y\to\PP^2$. Since $g_*\OO_Y=\OO_{\PP^2}$, the projection
formula gives $h^0(Y,M)=3$. Therefore, the net induced by $W/A$ is the complete
linear system $|M|$.

For the general cluster in each row the Hudson test gives a complete
homaloidal net, hence again $h^0(Y,M)=3$. The condition $h^0(Y,M)=3$ is open
on the irreducible cluster space by upper semicontinuity, and the preceding
projection-formula argument shows that every cluster arising from a system in
$\mathscr U_1$ lies in this open locus. The locus of general proper clusters
is dense in the same irreducible successive-blow-up space. Therefore every
system in a fixed row lies in the closure of the general family of that row. It follows that $\mathscr U_1$ is covered by the
closures of the seven irreducible families in the statement. Their
irreducibility and dimensions were proved in
Propositions~\ref{prop:curve-families}, \ref{prop:ribbon-family}, and
\ref{prop:dejonquieres-family}. We prove that none is contained in the closure
of another.

The ribbon family has the persistent cycle $2[L]$. No other general base in
the table contains such a cycle; hence no other family is contained in its
closure. The ribbon family has the largest dimension, so it is not contained
in any other closure.

The three $25$-dimensional families have distinct general values of $d_3$,
and the same holds for the three $24$-dimensional families. An inclusion
between the closures of two irreducible families of the same dimension would
force equality of the closures. Both families would then contain a non-empty
open subset of the same irreducible variety, contradicting the different
general values of $d_3$. Therefore no containment occurs among families of
the same dimension. To conclude, it is enough to exclude the containment of a
$24$-dimensional family in a $25$-dimensional one. Let
$\mathcal H_{2m+1}^{\rm pl}$ be the projective Hilbert family of plane conics
in $\PP^4$ and let $\mathcal H_{m+1}$ be the Grassmannian of lines. The
incidence of triples $(W,C,L)$ with
$C\in\mathcal H_{2m+1}^{\rm pl}$, $L\in\mathcal H_{m+1}$, and
$C\cup L\subset\Bs(W)$ is closed and has proper projection to the
Grassmannian of systems. Hence every specialization of the conic-plus-line
family still contains a plane conic scheme and a line. A general de
Jonqui\`eres base has only one reduced curve component; a general twisted
cubic is irreducible and contains no line; and two skew lines contain no
plane degree-two subscheme, since any degree-two subscheme with both generic
points has non-planar support. Thus none of the three $24$-dimensional
families lies in that closure.

Likewise, the projective Chow incidence shows that a specialization of the
rational quartic or elliptic quintic family contains an effective curve cycle
of degree four or five. The general bases of the three $24$-dimensional
families have total curve degree at most three. Lemma~\ref{lem:persistence-components}
therefore excludes all remaining containments. Hence the seven closures are
precisely the irreducible components of $\mathscr U_1$.
\end{proof}

\begin{Lemma}\label{lem:closed-rank-one-locus}
The locus
$\mathscr R_{\leq1}=\{W\in G\mid r_p(W)\leq1\text{ for some }p\in\Bs(W)\}$
is closed, and $\mathscr U\cap\mathscr R_{\leq1}=\varnothing$.
\end{Lemma}

\begin{proof}
On $\PP^4\times G$, impose the vanishing of the universal system and the
rank condition on its universal first-jet map. This is a closed determinantal
incidence, and its projection to $G$ is closed since $\PP^4$ is proper. If
$W\in\mathscr U$ and $A\subset W$ defines a smooth complete intersection,
the two differentials coming from $A$ are independent at every base point.
Therefore, $r_p(W)\geq2$ for all $p\in\Bs(W)$.
\end{proof}

Let
$
\mathscr H_{\mathrm{II}},\mathscr H_{\mathrm{III}}
\subset\Gr(4,H^0(\PP^3,\OO_{\PP^3}(2)))
$
be the irreducible components of the locally closed quadratic Cremona locus whose general members have
base schemes, respectively, a line and three points, and four points with a
prescribed tangent plane at one of them. By
\cite[Proposition~2.4.1]{PanRongaVust01}, they are distinct irreducible
components of dimensions $13$ and $14$. Set $\mathcal F_C$ to be the family in Proposition~\ref{prop:curve-families}\textup{(ii)}.

\begin{Lemma}\label{lem:conic-rank-one-reduction}
For a general $W\in\mathcal F_C$, the point $p$ occurring in
$C+p_C^{\PP^3}+q_1+q_2+q_3$ is the unique point of $\Bs(W)$ with
$r_p(W)=1$. The system obtained from $W$ by
Proposition~\ref{prop:geometric-reduction-p3} is a general point of
$\mathscr H_{\mathrm{II}}$.
\end{Lemma}

\begin{proof}
The base scheme of a general member of $\mathcal F_C$ is exactly
$C+p_C^{\PP^3}+q_1+q_2+q_3$. Its tangent space at $p$ is the prescribed
hyperplane, so $r_p(W)=1$. At a point of $C\setminus\{p\}$ the base is the
smooth curve $C$, hence the first jets generate its three-dimensional
conormal space, while the $q_i$ are reduced base points. Therefore, $p$ is the
unique rank-one point.

Write $W$ as in \eqref{eq:rank-one-normal-form-geometric} and set
$V=\langle q_1,q_2,q_3,q_4\rangle$. Projection from $p$ maps $C$ onto a line
$L\subset\PP^3$ and the points $q_i$ onto three general points $a_i$. Hence
$V\subset H^0(\I_{L\cup\{a_1,a_2,a_3\}}(2))$. A line and three general
points impose six independent conditions on quadrics in $\PP^3$, so the
space on the right has dimension four. Therefore equality holds. Conversely, a general line and three general
points in $\PP^3$ can be lifted by choosing $p\in\PP^4$, a smooth conic
through $p$ projecting isomorphically to the line, and three general lifted
points. Hence the map from the incidence defining $\mathcal F_C$ to
$\mathscr H_{\mathrm{II}}$ is dominant, and $V$ is a general point of
$\mathscr H_{\mathrm{II}}$.
\end{proof}

\begin{Lemma}\label{lem:conic-not-punctual-component}
The family $\mathcal F_C$ is not contained in the punctual component
$\mathcal P=\overline{\mathcal C^0_{248}}$.
\end{Lemma}

\begin{proof}
Assume the contrary and choose a general $W_0\in\mathcal F_C$. Since the
zero-dimensional locus is dense in $\mathcal P$, there is, after a finite base
change, a family $W_t\in\mathcal P$ over a smooth pointed curve such that
$W_0$ is the special fiber and $W_t$ is a general punctual transformation for
$t\neq0$. The latter has a unique rank-one point $p_t$. The closure of the
corresponding section in the proper incidence of pairs $(W,p)$ with
$r_p(W)\leq1$ gives a point $p_0\in\Bs(W_0)$ with $r_{p_0}(W_0)\leq1$.
Lemma~\ref{lem:first-jet-normal-form} excludes rank zero, and
Lemma~\ref{lem:conic-rank-one-reduction} gives $p_0=p$.

After shrinking the curve and trivializing the relevant bundles, the kernels
$V_t=W_t\cap H^0(\I_{p_t}^2(2))$ form a family of four-dimensional quadratic Cremona systems of $\PP^3$. For
$t\neq0$, the base of $V_t$ is zero-dimensional, so
Lemma~\ref{lem:prv-finite-component} gives
$V_t\in\mathscr H_{\mathrm{III}}$. After deleting finitely many points of the punctured curve, every
$V_t$, including $V_0$, lies in the locally closed quadratic Cremona locus of
$\PP^3$. The subset $\mathscr H_{\mathrm{III}}$ is an irreducible component
and is therefore closed relative to that locus. Since $V_t\in
\mathscr H_{\mathrm{III}}$ for $t\neq0$, relative closedness gives
$V_0\in\mathscr H_{\mathrm{III}}$. On the other hand,
Lemma~\ref{lem:conic-rank-one-reduction} says that $V_0$ is a general point
of $\mathscr H_{\mathrm{II}}$, contradicting the fact that these are distinct
irreducible components.
\end{proof}

Let $\overline{\mathcal H}_{-2}$ be the closure in
$\Hilb^{2m+3}(\PP^4)$ of the family of ribbons occurring in
Proposition~\ref{prop:ribbon-family}, and set
$$
\mathscr B_{-2}=
\{W\in G\mid R\subset\Bs(W)\text{ for some }
R\in\overline{\mathcal H}_{-2}\}.
$$

\begin{Lemma}\label{lem:conic-not-ribbon-locus}
The locus $\mathscr B_{-2}$ is closed, and a general member of
$\mathcal F_C$ does not belong to it.
\end{Lemma}

\begin{proof}
Closedness follows from Lemma~\ref{lem:persistence-components}. Under the
Hilbert--Chow morphism every ribbon in the family maps to $2[L]$ for its
supporting line $L$. The locus of cycles of the form $2[L]$ is closed, so
every member of $\overline{\mathcal H}_{-2}$ is supported on a line. The
positive-dimensional support of the base of a general member of
$\mathcal F_C$ is a smooth conic, which contains no line. Hence such a base
cannot contain a member of $\overline{\mathcal H}_{-2}$.
\end{proof}

\begin{Proposition}\label{prop:nine-distinct-components}
The space $\Bir_2(\PP^4)$ has at least nine irreducible components.
\end{Proposition}

\begin{proof}
The fiber dimension of the universal base scheme is upper semicontinuous. A
quadratic birational map cannot be base-point-free, since a morphism
$f:\PP^4\to\PP^4$ with $f^*\OO_{\PP^4}(1)=\OO_{\PP^4}(2)$ has degree $16$.
Hence the locus $\mathcal C^0$ of maps with zero-dimensional base scheme is
open in $\Bir_2(\PP^4)$. By
Theorem~\ref{thm:zero-dimensional-unique-component-geometric}, it is
irreducible; its closure $\mathcal P$ is therefore an irreducible component of
$\Bir_2(\PP^4)$. Proposition~\ref{prop:geometric-rank-one-point} and
Lemma~\ref{lem:closed-rank-one-locus} give
$\mathcal P\subset\mathscr R_{\leq1}$.

Let $\mathcal D_1,\ldots,\mathcal D_7$ be the components of $\mathscr U_1$
from Proposition~\ref{prop:seven-smooth-pencil-components}, and let
$\mathcal K_i$ be an irreducible component of $\Bir_2(\PP^4)$ containing
$\mathcal D_i$. Since $\mathcal D_i\cap\mathscr R_{\leq1}=\varnothing$, one
has $\mathcal K_i\neq\mathcal P$.

The general base dimension on $\mathcal K_i$ is at most one. It cannot be
zero, since $\mathcal P$ is the unique component meeting the irreducible open
zero-dimensional locus. Hence the general base dimension on $\mathcal K_i$
is one. The closed locus where the base dimension is at least one contains the
generic point of $\mathcal K_i$ and therefore all of $\mathcal K_i$. Therefore, $\mathcal K_i\cap\mathscr U_1$ is a non-empty open irreducible subset of
$\mathcal K_i$. It contains $\mathcal D_i$; since $\mathcal D_i$ is a
maximal irreducible subset of $\mathscr U_1$, one has
$\mathcal K_i\cap\mathscr U_1=\mathcal D_i$. Therefore the $\mathcal K_i$
are pairwise distinct and $\dim\mathcal K_i=\dim\mathcal D_i$.

Let $\mathcal K_C$ be an irreducible component containing the irreducible
$26$-dimensional family $\mathcal F_C$. By
Lemma~\ref{lem:conic-not-punctual-component}, one has
$\mathcal K_C\neq\mathcal P$. Six of the components $\mathcal K_i$ have
dimension at most $25$, so they cannot contain $\mathcal F_C$. The remaining
one is the $26$-dimensional ribbon component. Its intersection with
$\mathscr U_1$ is the ribbon family, which is dense in it; hence it is
contained in the closed locus $\mathscr B_{-2}$. By
Lemma~\ref{lem:conic-not-ribbon-locus}, a general member of $\mathcal F_C$
does not belong to $\mathscr B_{-2}$. Therefore, $\mathcal K_C$ is distinct from
all the preceding eight components, proving the claim.
\end{proof}

\section{Two-dimensional base schemes}\label{sec:surfaces}

In this section we consider the cases with $d_2<4$, that is precisely the cases in which the base scheme contains a surface component.

\begin{Lemma}\label{lem:surface-segre-classes}
Let $h=c_1(\mathcal O_{\PP^4}(1))$. The following Segre classes hold:
\begin{enumerate}
\item \textit{if $P\subset\PP^4$ is a plane, then $s(P,\PP^4)=h^2-2h^3+3h^4$;}
\item \textit{if $Q\subset\PP^4$ is a smooth quadric surface spanning a hyperplane and $q$ is a point outside $Q$, then
$
s(Q\cup q,\PP^4)=2h^2-6h^3+15h^4;
$}
\item \textit{if $P\subset\PP^4$ is a plane, $L$ is a line meeting $P$ in one point and not contained in $P$, and $q_1,q_2$ are general points, then
$
s(P\cup L\cup q_1\cup q_2,\PP^4)=h^2-h^3-h^4;
$}
\item \textit{if $P\subset\PP^4$ is a plane and $L_1,L_2$ are skew lines, both meeting $P$ in one point, then
$
s(P\cup L_1\cup L_2,\PP^4)=h^2-9h^4.
$}
\end{enumerate}
\end{Lemma}

\begin{proof}
For a plane, $N_{P/\PP^4}\simeq\mathcal O_P(1)^{\oplus2}$, so
$
 s(P,\PP^4)=(1+h)^{-2}[P]=h^2-2h^3+3h^4.
$
For a smooth quadric $Q\subset H\simeq\PP^3$,
$
N_{Q/\PP^4}\simeq\mathcal O_Q(2)\oplus\mathcal O_Q(1).
$
Since $[Q]=2h^2$, inversion of its Chern class gives
$
 s(Q,\PP^4)=2h^2-6h^3+14h^4.
$
A disjoint reduced point adds $h^4$, proving (ii).

We compute the incident-line correction explicitly. Blow up $P$ and let $E$
be the exceptional divisor. If $L$ meets $P$ transversely at $p$, its strict
transform $\widetilde L$ is smooth and
$
I_{P\cup L}\mathcal O_{\Bl_P\PP^4}
 =\mathcal O(-E)I_{\widetilde L}.
$
Blow up $\widetilde L$, with exceptional divisor $F$, and compare the
divisors $E+F$ and $E$. The codimension-three push-forward formulas
\eqref{eq:codim-three-blowup-push} give
\stepcounter{thm}\begin{equation}\label{eq:plane-line-segre-difference}
 s(P\cup L,\PP^4)-s(P,\PP^4)
 =[L]-\bigl(\deg c_1(N_{\widetilde L/\Bl_P\PP^4})
              +4E\cdot\widetilde L\bigr)h^4.
\end{equation}
Here $E\cdot\widetilde L=1$. Since
$N_{L/\PP^4}\simeq\mathcal O_L(1)^{\oplus3}$ has degree $3$ and blowing up
the transverse codimension-two center $P$ decreases the normal degree by one,
$
\deg c_1(N_{\widetilde L/\Bl_P\PP^4})=2.
$
Thus the correction is $h^3-6h^4$, and
$
 s(P\cup L)=h^2-h^3-3h^4.
$
Adding two disjoint points proves (iii). If two skew lines meet $P$ at
distinct points, the two local corrections are supported at disjoint points
and add. Hence
$
 s(P\cup L_1\cup L_2)=s(P)+2(h^3-6h^4)=h^2-9h^4,
$
which proves (iv).
\end{proof}

\begin{Proposition}\label{prop:234}
Let $P\subset\PP^4$ be a plane and let $q_1,\ldots,q_4$ be general points outside $P$. Then $H^0(\I_{P\cup q_1\cup\cdots\cup q_4}(2))$ has dimension $5$ and defines a quadratic Cremona transformation of multidegree $(1,2,3,4,1)$. The corresponding family has dimension $22$.
\end{Proposition}

\begin{proof}
A plane imposes $6$ independent conditions on quadrics, since $h^0(\mathcal O_P(2))=6$. The four points impose four further independent conditions for general choices. Hence
$
h^0(\I_{P\cup q_1\cup\cdots\cup q_4}(2))=15-10=5.
$
For an exact coordinate witness, set $P=V(x_3,x_4)$ and take
$
q_1=e_3,\quad q_2=e_4,\quad q_3=[1:0:0:1:1],\quad
q_4=[0:1:1:1:-1].
$
The complete system is
\begin{align*}
\langle &x_0(x_3-x_4),\ x_1(x_3+x_4),\ x_3(x_1-x_2),\\
        &x_1x_3+x_2x_4,\ x_3(x_0-x_1-x_4)\rangle.
\end{align*}
On $x_3\neq0$, putting $r=x_4/x_3$ gives precisely the three solutions
$r=0,1,-1$, namely $q_1,q_3,q_4$; on $x_3=0$, $x_4\neq0$, the only solution
is $q_2$. Along $P$, the $2\times2$ minors of the normal coefficient matrix
generate the irrelevant ideal in $\kk[x_0,x_1,x_2]$, so $P$ is reduced and
there is no embedded component. Thus the saturated base ideal is exactly the
stated union. The same holds for general data by openness.
The parameter space is irreducible, and its dimension is
$
\dim\Gr(3,5)+4\cdot4=6+16=22.
$

The Segre class is
$
s(P\cup q_1\cup\cdots\cup q_4,\PP^4)=h^2-2h^3+7h^4.
$
Substituting this class in \eqref{eq:segre-degrees} gives $d_1=2$, $d_2=3$, $d_3=4$ and $d_4=1$. Remark~\ref{rem:segre-birationality-criterion} shows that the rational map is dominant and birational. Therefore the multidegree is $(1,2,3,4,1)$.
\end{proof}

\begin{Proposition}\label{prop:233}
Let $P\subset\PP^4$ be a plane, let $L$ be a line meeting $P$ in one point and not contained in $P$, and let $q_1,q_2$ be general points. Then $H^0(\I_{P\cup L\cup q_1\cup q_2}(2))$ has dimension $5$ and defines a quadratic Cremona transformation of multidegree $(1,2,3,3,1)$. The corresponding family has dimension $19$.
\end{Proposition}

\begin{proof}
The incidence variety of pairs $(P,L)$ with $L\cap P$ a point and $L\not\subset P$ is irreducible. Indeed, one first chooses $P$, then a point of $P$, and then a line through that point not contained in $P$. This gives
$
6+2+3=11
$
parameters. Adding two general points gives dimension $11+8=19$.

The plane imposes $6$ conditions on quadrics. A line meeting $P$ adds only $2$ further conditions, since a quadric containing $P$ already vanishes at the point $L\cap P$. Finally $q_1$ and $q_2$ impose two further independent conditions. Hence the total number of conditions is $10$, and
$
h^0(\I_{P\cup L\cup q_1\cup q_2}(2))=5.
$
Take $P=V(x_3,x_4)$, $L=\PP\langle e_0,e_3\rangle$,
$q_1=e_4$, and $q_2=[1:1:1:1:1]$. The complete system is
\begin{align*}
\langle &-x_0x_4+x_1x_3,\ -x_4(x_0-x_1),
 -x_0x_4+x_2x_3,\\
 &-x_4(x_0-x_2),\ -x_4(x_0-x_3)\rangle.
\end{align*}
On $x_4\neq0$ its affine ideal is
$
(x_1-x_0,x_2-x_0,x_3-x_0,x_0(x_0-1)),
$
which gives $q_1$ and $q_2$. On $x_4=0$ outside $P$, one has $x_3\neq0$ and
the equations give $x_1=x_2=0$, namely $L$. Localization at the two
intersection strata shows that all components are reduced. Thus the
saturated base ideal is exactly $P\cup L\cup q_1\cup q_2$, and the same holds
generally by openness.

By Lemma~\ref{lem:surface-segre-classes},
$
s(P\cup L\cup q_1\cup q_2,\PP^4)=h^2-h^3-h^4.
$
Substitution in \eqref{eq:segre-degrees} gives
$
d_2=4-1=3,\, d_3=8-6+1=3,\, d_4=16-24+8+1=1.
$
Remark~\ref{rem:segre-birationality-criterion} shows that the map is dominant and birational, and its multidegree is $(1,2,3,3,1)$.
\end{proof}

\begin{Proposition}\label{prop:232}
	Let $P\subset\PP^4$ be a plane and let $L_1,L_2$ be skew lines, each meeting $P$ in one point. Then $H^0(\I_{P\cup L_1\cup L_2}(2))$ has dimension $5$ and defines a quadratic Cremona transformation of multidegree $(1,2,3,2,1)$. The corresponding family has dimension $16$.
\end{Proposition}

\begin{proof}
	The parameter space is irreducible. We choose the plane $P$, and then each line $L_i$ by choosing its intersection point with $P$ and a direction not tangent to $P$. This gives
	$
	6+2\cdot5=16
	$
	parameters, with the condition that $L_1$ and $L_2$ are skew being open.
	
	The plane imposes $6$ conditions on quadrics. Each line $L_i$ adds $2$ further conditions, since quadrics through $P$ already vanish at the point $L_i\cap P$. Therefore, the total number of conditions is $10$, and
	$
	h^0(\I_{P\cup L_1\cup L_2}(2))=5.
	$
	For $P=V(x_3,x_4)$,
$L_1=\PP\langle e_0,e_3\rangle$, and
$L_2=\PP\langle e_1,e_4\rangle$, the complete system is
$
\langle x_0x_4,x_1x_3,x_2x_3,x_2x_4,x_3x_4\rangle.
$
Its monomial ideal is exactly
$
I_P\cap I_{L_1}\cap I_{L_2}.
$
This gives an exact rational witness, and the assertion for general data
follows by openness.
	
	By Lemma~\ref{lem:surface-segre-classes},
	$
	s(P\cup L_1\cup L_2,\PP^4)=h^2-9h^4.
	$
	Substitution in \eqref{eq:segre-degrees} gives
	$
	d_2=4-1=3,\, d_3=8-6=2,\, d_4=16-24+9=1.
	$
	Remark~\ref{rem:segre-birationality-criterion} shows that the map is dominant and birational, and its multidegree is $(1,2,3,2,1)$.
\end{proof}

\begin{Proposition}\label{prop:222}
Let $Q\subset\PP^4$ be a smooth quadric surface spanning a hyperplane and let $q\in\PP^4$ be a general point. Then $H^0(\I_{Q\cup q}(2))$ has dimension $5$ and defines a quadro-quadric Cremona transformation of multidegree $(1,2,2,2,1)$. The corresponding family has dimension $17$.
\end{Proposition}

\begin{proof}
The quadric surface $Q$ is contained in a unique hyperplane $H\simeq\PP^3$. Since $Q\simeq\PP^1\times\PP^1$ and $\mathcal O_Q(1)\cong\mathcal O_{\PP^1\times\PP^1}(1,1)$, we have
$
h^0(\mathcal O_Q(2))=h^0(\mathcal O_{\PP^1\times\PP^1}(2,2))=9.
$
Therefore, $Q$ imposes $9$ independent conditions on quadrics in $\PP^4$. The general point $q$ imposes one further condition, so
$
h^0(\I_{Q\cup q}(2))=15-10=5.
$
Take $H=V(x_4)$,
$Q=V(x_4,x_0x_3-x_1x_2)$, and $q=e_4$. Then
$
H^0(\I_{Q\cup q}(2))=
\langle x_0x_4,x_1x_4,x_2x_4,x_3x_4,x_0x_3-x_1x_2\rangle,
$
and its ideal is exactly $I_Q\cap I_q$. This is an exact rational witness;
the general assertion follows by openness.

The family of smooth quadric surfaces in hyperplanes of $\PP^4$ has dimension $4+9=13$, and the point $q$ contributes $4$ further parameters. Hence the total dimension is $17$.

By Lemma~\ref{lem:surface-segre-classes},
$
s(Q\cup q,\PP^4)=2h^2-6h^3+15h^4.
$
Using \eqref{eq:segre-degrees}, we get
$
d_2=4-2=2,\, d_3=8-12+6=2,\, d_4=16-48+48-15=1.
$
Remark~\ref{rem:segre-birationality-criterion} shows that the map is dominant and birational. Since $d_3=2$, its inverse is also quadratic, and the multidegree is $(1,2,2,2,1)$.
\end{proof}

\section{The upper bound}\label{sec:upper-bound}

In this section we prove that the nine components found in Proposition~\ref{prop:nine-distinct-components} exhaust $\Bir_2(\PP^4)$. We begin by considering systems with a surface in the base scheme and systems with a rank-one base point. Then, we study one-dimensional base schemes whose first-jet rank is at least two at every point.


\begin{Lemma}\label{lem:surface-component-degree-two}
Let $W\in\Bir_2(\PP^4)$ and let $S$ be the unmixed two-dimensional part of
$\Bs(W)$. Then $S$ is contained in a hyperplane and has degree at most two
in that hyperplane.
\end{Lemma}

\begin{proof}
Since $W\subset H^0(\I_S(2))$, one has $h^0(\I_S(2))\geq5$. Assume that
$S$ is non-degenerate. Let $C$ be a general hyperplane section. Even when $S$ is non-reduced or
not arithmetically Cohen--Macaulay, the hyperplane exact sequence gives
$\operatorname{HF}_S(2)\geq\operatorname{HF}_S(1)+\operatorname{HF}_C(2)$.
The curve $C\subset\PP^3$ is non-degenerate. A further general hyperplane
section has length at least three, and the Hilbert function of three points in
$\PP^2$ in degree two is at least three; together with non-degeneracy this
gives $\operatorname{HF}_C(2)\geq6$. Since
$\operatorname{HF}_S(1)=5$, we obtain
$\operatorname{HF}_S(2)\geq11$. Therefore
$h^0(\I_S(2))\leq4$, a contradiction.

Therefore, $S\subset H\simeq\PP^3$. Since $S$ is unmixed of codimension one in
$H$, it is defined in $H$ by one form of degree $d$. If $d\geq3$, every
quadric containing $S$ vanishes on $H$, so $H$ is a fixed component of $W$.
This contradicts the assumption that the minimal algebraic degree is
exactly two. Hence $d\leq2$.
\end{proof}

We begin with a plane. Let $P\subset\PP^4$ be contained in $\Bs(W)$ and let
$\pi:\Bl_P(\PP^4)\to\PP^1$ be the morphism induced by the pencil of
hyperplanes through $P$. After removing $P$, the restrictions of quadrics to
the fibers of $\pi$ form the bundle
$
\mathcal E=\OO_{\PP^1}(1)^{\oplus3}\oplus\OO_{\PP^1}(2).
$
The system $W$ induces $\alpha_W:W\otimes\OO_{\PP^1}\to\mathcal E$.

\begin{Lemma}\label{lem:torsion-quot-irreducible}
Let $C$ be a smooth irreducible curve and let $\mathcal E$ be a vector bundle
of rank $r$ on $C$. For every $\ell\geq0$, the scheme
$\operatorname{Quot}^{\ell}_C(\mathcal E)$ of length-$\ell$ torsion
quotients is smooth and irreducible of dimension $r\ell$.
\end{Lemma}

\begin{proof}
At a quotient
$
0\longrightarrow\mathcal K\longrightarrow\mathcal E
\longrightarrow T\longrightarrow0,
$
the kernel $\mathcal K$ is locally free. The tangent and obstruction spaces
are
$
\operatorname{Hom}(\mathcal K,T),\qquad
\operatorname{Ext}^1(\mathcal K,T)=H^1(\mathcal K^*\otimes T)=0.
$
Thus the Quot scheme is smooth, and the tangent dimension is
$r\ell$. The open locus of quotients supported at $\ell$ distinct points is
irreducible of dimension
$
\ell+\ell(r-1)=r\ell.
$
It remains to prove density. Locally at a support point, trivialize
$\mathcal E$ over the DVR $R=\kk[[t]]$. Smith normal form writes the quotient
as
$
\bigoplus_{i=1}^rR/(t^{\lambda_i}).
$
Choose pairwise distinct constants $c_{ij}$ and replace $t^{\lambda_i}$ by
the monic polynomial
$
\prod_{j=1}^{\lambda_i}(t-c_{ij}\tau).
$
The resulting diagonal quotient over $\kk[[\tau]][[t]]$ is flat; for
$\tau\neq0$ all support points are distinct and each fiber quotient is
one-dimensional. Applying this construction in disjoint formal
neighborhoods and keeping the quotient fixed elsewhere smooths every point of
the Quot scheme into the distinct-support locus. Hence that open locus is
dense, and smoothness implies irreducibility.
\end{proof}

\begin{Proposition}\label{prop:all-plane-systems}
Every quadratic Cremona transformation whose base contains a plane belongs to
the closure of the component whose general base is
$L_1\sqcup L_2+p^{\PP^2}+q$.
\end{Proposition}

\begin{proof}
The map $\alpha_W$ has generic rank four. Otherwise the restriction of
$\vf_W$ to a general fiber of $\pi$ has image of dimension at most two, and
$\vf_W$ is not dominant. Set $\mathcal F=\operatorname{Im}(\alpha_W)$. There
are exact sequences
$$
0\longrightarrow\OO_{\PP^1}(-a)\longrightarrow
W\otimes\OO_{\PP^1}\longrightarrow\mathcal F\longrightarrow0 \quad\text{  and  }\quad
0\longrightarrow\mathcal F\longrightarrow\mathcal E\longrightarrow T
\longrightarrow0,
$$
where $T$ is a torsion sheaf. If $\ell=\length T$, comparison of degrees gives
$a+\ell=5$.

The kernel defines a morphism $\kappa:\PP^1\to\PP(W)$ with
$\kappa^*\OO_{\PP(W)}(1)=\OO_{\PP^1}(a)$. For a general parameter $t$, the
restriction of $W$ to the corresponding $\PP^3$ is a complete system of
linear forms and maps it isomorphically onto
$\kappa(t)^\perp\subset\PP(W^*)$. Hence the inverse images of a general
$y\in\PP(W^*)$ are indexed by the zeros of
$\langle y,\kappa(t)\rangle$. Therefore $\deg(\vf_W)=a$, and birationality
gives $a=1$, $\ell=4$, and
$\mathcal F\simeq\OO_{\PP^1}(1)\oplus\OO_{\PP^1}^{\oplus3}$.

Consequently every such $W$ is obtained from a quotient
$\mathcal E\twoheadrightarrow T$ of length four. By Lemma~\ref{lem:torsion-quot-irreducible}, the scheme
$\operatorname{Quot}_{\PP^1}^4(\mathcal E)$ is smooth and irreducible of
dimension $16$. For a general quotient
the kernel is globally generated and the corresponding system is
$H^0(\I_{P\cup\{q_1,q_2,q_3,q_4\}}(2))$. Therefore, every system containing $P$
belongs to the closure of the irreducible $22$-dimensional family with base
$P+q_1+q_2+q_3+q_4$.

To conclude, it is enough to place this family in one of the components already constructed. Set
coordinates $[x_0:x_1:x_2:u:v]$ and $P=V(u,v)$. For $c$ in a non-empty open subset of $\mathbb A^1$ and a deformation
parameter $t$, set $L_{1,t}=\langle a_t(c),c_t\rangle$,
$L_{2,t}=\langle b_t,d_t\rangle$, let $p_t^{\Lambda_t}$ be the first order
scheme with $\Lambda_t=\PP\langle p_t,\delta_{1,t},\delta_{2,t}\rangle$, and
let $q_t$ be given by
\begin{footnotesize}
$$
\begin{aligned}
a_t(c)&=e_0+t(-2e_0+e_1+e_2+ce_3-e_4),&
 c_t&=e_2+t(e_0+2e_2+e_3-3e_4),\\
b_t&=e_1+t(e_0-3e_1+3e_2-e_4),&
 d_t&=e_2+t(e_0-2e_1-2e_2+2e_3),\\
p_t&=e_0+e_1+e_2+t(e_0+3e_1+e_2),&
 q_t&=e_0+e_1+e_3+2e_4+t(-3e_0+2e_1+3e_2-3e_3-2e_4),\\
\delta_{1,t}&=e_0-e_2+t(2e_0+3e_1-2e_2-2e_3+2e_4),&
\delta_{2,t}&=e_1-e_2+t(-2e_0+3e_1+e_2+2e_4).
\end{aligned}
$$
\end{footnotesize}
For $t\neq0$ general these data define a member of the
$L_1\sqcup L_2+p^{\PP^2}+q$ family. Let $M_c(t)$ be the $10\times15$ matrix
of their linear conditions on quadrics. The Smith form of $M_c(t)$ over
$\QQ(c)[t]$ gives a polynomial basis of its kernel and therefore its flat
limit at $t=0$. The exceptional values of $c$ form a proper closed subset.
After setting $u=sz$, $v=rz$ and removing $z$, the greatest common
divisors of the maximal minors of the two special matrices are
$$
g_{-3}=(r-3s)(r-2s)(2r-5s)(3r-s), \quad g_{-1}=(r-3s)(r-2s)(r+5s)(2r-s).
$$
The quotient minors have no common zero on either quartic, so the special
base schemes are $P$ together with four reduced points outside $P$. The exact Smith form and minor computations are reproduced by the Magma
file described in Section~\ref{appendix:magma}.

For a binary quartic $as^4+bs^3r+cs^2r^2+dsr^3+er^4$, set
$I=12ae-3bd+c^2$ and
$J=72ace+9bcd-27ad^2-27b^2e-2c^3$. One obtains
$$
\frac{J(g_{-3})^2}{I(g_{-3})^3}=\frac{47628}{79507},
\qquad
\frac{J(g_{-1})^2}{I(g_{-1})^3}=\frac{1244819524}{887503681}.
$$
The formulas defining the family are algebraic in $c$, and the two displayed
specializations $c=-3,-1$ have different values of the absolute binary
quartic invariant $J^2/I^3$. Therefore this invariant, and hence the
cross-ratio of the four residual points, is non-constant as $c$ varies. A
block calculation after fixing $P$ and four general points shows that the connected stabilizer in $\PGL_5$ has dimension three. Hence the orbit has dimension $21$, while the family has dimension $22$. The $\PGL_5$-saturation of the
preceding one-parameter family is therefore dense in it. This proves the claim.
\end{proof}

\begin{Proposition}\label{prop:quadric-surface-boundary-rational-quartic}
The family whose general base is a smooth quadric surface in a hyperplane and
one point is contained in the closure of the $R_4+q$ family.
\end{Proposition}

\begin{proof}
Let $Q\subset H\simeq\PP^3$ be smooth and identify
$Q\simeq\PP^1\times\PP^1$, with $\mathcal O_Q(1)=\mathcal O_Q(1,1)$. A
smooth curve $R_0\in|\mathcal O_Q(1,3)|$ is rational of degree four and spans
$H$. A quadric of $H$ containing $R_0$ contains $Q$, because a divisor of
class $(2,2)$ cannot contain one of class $(1,3)$. Hence
$
H^0(\I_{R_0/\PP^4}(2))=H^0(\I_Q(2))
$
and both spaces have dimension six.

Let $s_0,\ldots,s_3$ be the four sections of
$H^0(\mathcal O_{\PP^1}(4))$ defining the embedding $R_0\subset H$, and
complete them by $s_4$ to a basis. Define
\stepcounter{thm}\begin{equation}\label{eq:quartic-flat-family}
 f_t:\PP^1\longrightarrow\PP^4,\qquad
 [s_0:s_1:s_2:s_3:t s_4].
\end{equation}
Projection onto the first four coordinates is the fixed embedding of
$\PP^1$ as $R_0$. Consequently the relative map
$\PP^1\times\mathbb A^1\to\PP^4\times\mathbb A^1$ is a closed immersion;
its image $\mathcal R$ is isomorphic to $\PP^1\times\mathbb A^1$ and is flat.
For $t\neq0$, \eqref{eq:quartic-flat-family} is the complete linear system
$|\mathcal O_{\PP^1}(4)|$, hence $R_t$ is a rational normal quartic.

For every fiber, the restriction map
$
H^0(\PP^4,\mathcal O(2))\longrightarrow
H^0(\PP^1,\mathcal O(8))
$
is surjective: this is classical for the rational normal quartic, and at
$t=0$ it follows from
$H^0(\I_{R_0}(2))=H^0(\I_Q(2))$ of dimension six. Thus its kernels form a
rank-six vector bundle near $0$. Choose a general point $q$ outside
$\mathcal R$ such that evaluation at $q$ is non-zero on the special kernel.
The kernel of this evaluation is then a rank-five subbundle and
\begin{align*}
\lim_{t\to0}H^0(\I_{R_t\cup\{q\}}(2))
 &=H^0(\I_{R_0\cup\{q\}}(2))\\
 &=H^0(\I_{Q\cup\{q\}}(2)).
\end{align*}
This is the required closure relation.
\end{proof}

\begin{Lemma}\label{lem:line-degree-two-cancellation}
Let $a,b$ be non-zero linear forms on $\PP^1$, let $g$ be quadratic, and let
$c_0,\ldots,c_m$ be scalars, not all zero. The rational map
$
[c_0ab:\cdots:c_mab:g]
$
has degree $2-\deg\gcd(ab,g)$ onto its image, unless it is constant. In
particular, it is birational onto its image if and only if
$\deg\gcd(ab,g)=1$.
\end{Lemma}

\begin{proof}
After canceling $d=\gcd(ab,g)$, the image is a line and the map is defined by
two coprime forms $ab/d,g/d$ of degree $2-\deg d$.
\end{proof}

\begin{Proposition}\label{prop:all-quadric-surface-systems}
Every quadratic Cremona transformation whose base has a two-dimensional
quadric component belongs to the closure of the component whose general base
is $R_4+q$.
\end{Proposition}

\begin{proof}
Let $Q=V_H(F)$, where $H=V(h)\simeq\PP^3$ and $F$ has degree two. Then
$H^0(\I_Q(2))=hH^0(\OO_{\PP^4}(1))\oplus\langle F\rangle$. Since $W$ has
dimension five and has no fixed hyperplane,
$W\cap hH^0(\OO_{\PP^4}(1))=hU$, where $U$ is a hyperplane of linear forms.
Let $q$ be the base point of $U$. After changing the fifth generator,
$W=\langle hU,F+hx\rangle$.

Let $L$ be a general line through $q$. Write $a=0$ for $q$ on $L$ and
$b=h|_L$. The restrictions of the first four generators are
$c_0ab,\ldots,c_3ab$, while the last one is $g=(F+hx)|_L$. After cancelling
the common factor $h$, the ratios of the first four coordinates are the
linear projection from $q$. Hence they recover the general line $L$, and the
generic fiber of $\varphi_W$ is contained in such a line. Its degree is
therefore the degree of the restriction to $L$. By
Lemma~\ref{lem:line-degree-two-cancellation}, birationality is equivalent to
$\deg\gcd(ab,g)=1$.

Assume first $q\notin H$, so $a,b$ are independent. The condition $a\mid g$
is the single global condition $g(q)=0$ and therefore does not depend on the
chosen line. If it fails, birationality forces $b\mid g$ on every general
line through $q$. At the point $L\cap H$ this says $F=0$; as $L$ varies,
these points fill a dense open subset of $H$, contradicting the fact that
$F$ is a non-zero quadratic equation on $H$. Thus $a\mid g$, so
$q\in\Bs(W)$ and, by dimension,
$W=H^0(\I_{Q\cup\{q\}}(2))$.

Assume $q\in H$. Then $a=b$ and the same lemma gives
$\operatorname{ord}_q(g)=1$. Therefore, $q\in Q$ and the residual condition is the
first-order condition determined by the non-zero class of $g$ in
$\mathfrak m_q/\mathfrak m_q^2$. In affine coordinates centered at $q$, let
$h$ be the normal coordinate. If the linear part of $F$ is non-zero, choose
an arc $q_t$ with $h(q_t)=t$ and
$F(q_t)+h(q_t)x(q_t)=O(t^2)$. If it vanishes, after the base change $t=s^2$
choose $q_s=q+sv+s^2n$, with $v\in T_qH$, $h(n)=1$, and the quadratic term
of $F$ at $v$ equal to $-x(q)$. In both cases a linear form $x_t\to x$ can
be chosen so that $F(q_t)+h(q_t)x_t(q_t)=0$. Deforming $U$ to the hyperplanes
of linear forms vanishing at $q_t$ shows that $W$ is a limit of systems
$H^0(\I_{Q\cup\{q_t\}}(2))$ with $q_t\notin H$. Singular, reducible,
and double quadrics are limits of smooth quadrics in their hyperplanes, and
the point deforms simultaneously. Therefore, every system is in
the closure of the family $Q+q$. Proposition~\ref{prop:quadric-surface-boundary-rational-quartic}
gives $\mathcal F_{Q+q}\subset\overline{\mathcal F}_{R_4+q}$.
\end{proof}

Lemmas~\ref{lem:surface-component-degree-two} and
Propositions~\ref{prop:all-plane-systems} and
\ref{prop:all-quadric-surface-systems} show that a component not among the
nine already found cannot have a two-dimensional general base.

Let $\mathscr H_{\rm I},\mathscr H_{\rm II},\mathscr H_{\rm III}$ be the
three irreducible components of the locally closed quadratic Cremona locus of $\PP^3$ in
\cite{PanRongaVust01}. They are closed relative to that locus. Their dimensions are $11,13,14$, and their general
bases are, respectively, a smooth conic and one point, a line and three
points, and a finite scheme of type $a_1+a_2+a_3+b^{\PP^2}$.

\begin{Proposition}\label{prop:all-rank-one-systems}
If $W\in\Bir_2(\PP^4)$ has a point $p\in\Bs(W)$ with $r_p(W)=1$, then $W$
belongs to the punctual component, to the component whose general base is
$C+p_C^{\PP^3}+q_1+q_2+q_3$, or to the closure of the component whose general
base is $R_4+q$.
\end{Proposition}

\begin{proof}
By Proposition~\ref{prop:geometric-reduction-p3} one has
$
W=\langle x_0\ell+q_0,q_1,q_2,q_3,q_4\rangle,
$
where $V=\langle q_1,q_2,q_3,q_4\rangle$ defines a quadratic Cremona
transformation of $\PP^3$. For each $\alpha\in\{\rm I,II,III\}$, the
incidence of lifts of $\mathscr H_\alpha$ is irreducible. Its dimension is
$4+3+\dim\mathscr H_\alpha+6$. The first terms choose $p$ and
$[\ell]\in\PP(T_p^*\PP^4)$. Once the coefficient of $x_0\ell$ is normalized,
$q_0$ is an element of the affine space
$H^0(\PP^3,\OO_{\PP^3}(2))/V$, which has dimension six. Hence the three
incidences have dimensions $24,26,27$. For a general lift, $p$ is the unique
rank-one point, so the maps to the Grassmannian are generically finite.

For type III the general lift has finite base and belongs to the punctual
component by Theorem~\ref{thm:zero-dimensional-unique-component-geometric}.
For type II the general lift has base
$C+p_C^{\PP^3}+q_1+q_2+q_3$ by
Lemma~\ref{lem:conic-rank-one-reduction}; its closure is the ninth component.
Finally, we consider type I.

Set coordinates $[X_0:X_1:X_2:X_3:Z]$. Let $\mathcal C$ be the saturation
with respect to $t$ of the family generated by
\begin{footnotesize}
$$
\begin{aligned}
&(tX_0+Z)X_2-2tX_1^2,
\, (tX_0+Z)X_3-2tX_1X_2,\\
&t^2X_0^2-Z^2-4t^2X_1X_3,
\, X_1X_3-X_2^2,\\
&X_1(tX_0-Z)-2tX_2X_3,
\, X_2(tX_0-Z)-2tX_3^2,\\
&X_0X_2-X_1^2-X_3^2.
\end{aligned}
$$
\end{footnotesize}
For $t\neq0$ its fiber is the rational normal quartic parametrized by
$
[s:r]\mapsto[s^4+r^4:s^3r:s^2r^2:sr^3:t(s^4-r^4)].
$
The special ideal is
$
I_{C_0}=(ZX_1,ZX_2,ZX_3,Z^2,X_1X_3-X_2^2,
X_0X_2-X_1^2-X_3^2).
$
Its reduction in $Z=0$ is a complete intersection of two quadrics with
Hilbert polynomial $4m$, while the nilpotent class of $Z$ contributes
$\kk[X_0](-1)$. Therefore, $P_{C_0}(m)=4m+1$, and the family is flat.

Set $q=[1:1:1:2:3]$ and
$W_0=H^0(\I_{C_0\cup\{q\}}(2))$. At
$p=[1:0:0:0:0]$ the first-jet rank is one. The subsystem independent of
$X_0$ is the complete system through the conic
$V(Z,X_1X_3-X_2^2)$ and the projection $[1:1:2:3]$ of $q$, hence it is a
type I system. Keeping $q$ fixed gives a flat family from $W_0$ to systems
with base $R_4+q$.

The Lie algebra of the stabilizer of $W_0$ is the kernel of
$
\mathfrak{gl}_5\rightarrow
\operatorname{Hom}(W_0,\operatorname{Sym}^2((\kk^5)^*)/W_0).
$
In the monomial basis the corresponding rational matrix has rank $24$, so its
kernel consists of the scalar matrices. The calculation is reproduced by
\texttt{TypeIStabilizerCertificate} in Section~\ref{appendix:magma}. Therefore, the
stabilizer of $W_0$ in $\PGL_5$ is finite and its orbit has dimension $24$,
equal to the dimension of the irreducible type I lift incidence. The orbit is
therefore dense in its image. Since
$\overline{\mathcal{F}}_{R_4+q}$ is $\PGL_5$-invariant and contains $W_0$, it
contains every type I lift.
\end{proof}

We may therefore assume from now on that
$
\dim\Bs(W)=1,\, r_p(W)\geq2\text{ for every }p\in\Bs(W).
$ Choose symmetric matrices $A_0,\ldots,A_4$ for a basis of $W$ and set
$
A_W(y)=\sum_{i=0}^4y_iA_i,
\, F_W(y)=\det A_W(y).
$

\begin{Lemma}\label{lem:reduced-discriminant-smooth-pencil}
If $F_W\neq0$ is reduced, a general pencil in $W$ cuts out a smooth complete
intersection of two quadrics.
\end{Lemma}

\begin{proof}
A general line in $\PP(W)$ meets $V(F_W)$ transversally at rank-four
quadrics. If $V(Q,R)$ were singular at $p$, some $S\in\langle Q,R\rangle$
would have vertex $p$. For an independent $T$ one has
$d(\det)_S(T)=cT(p)$ with $c\neq0$. Since $p\in V(Q,R)$, this vanishes, so
the line is tangent to $V(F_W)$, a contradiction.
\end{proof}

\begin{Proposition}\label{prop:four-residual-branches}
Assume that $W$ contains no pencil cutting out a smooth complete intersection.
Then one of the following occurs:
\begin{enumerate}
\item[(H1)] \textit{$F_W$ has a multiple linear factor whose rank-four vertex curve is
a line;}
\item[(H2)] \textit{$F_W$ has a multiple linear factor whose rank-four vertex curve is
a smooth conic;}
\item[(Q)] \textit{$F_W$ has a multiple irreducible quadratic factor;}
\item[(R)] \textit{$F_W\equiv0$, or the general member of a multiple factor has rank
at most three.}
\end{enumerate}
\end{Proposition}

\begin{proof}
If $F_W\neq0$ is not reduced, write $F_W=g^mh$ with $m\geq2$, $g\nmid h$.
Since $\deg F_W=5$, one has $\deg g\leq2$. Suppose that the general member of
$D=V(g)$ has rank four and let $p$ be its vertex. Since $D$ is multiple,
$dF_W=0$ at its general point. The identity
$d(\det)_Q(T)=cT(p)$ gives $p\in\Bs(W)$.

Let $B$ be the reduced image of the vertex map and set
$K_p=\ker(W\to\mathfrak m_p/\mathfrak m_p^2)$. Non-reduced structures on
this image do not affect the following incidence. The incidence of pairs
$(p,Q)$ with $Q\in D$ and $p=\operatorname{Vert}(Q)$ has dimension three and
maps to the one-dimensional base. Its general fiber is contained in
$\PP(K_p)$, which has dimension $4-r_p(W)\leq2$. Hence $r_p(W)=2$ and
$
D=\overline{\bigcup_p\PP(K_p)}.
$ Assume first that $D$ is a hyperplane. The planes $\PP(K_p)$ form a curve in
the dual $\PP^3$. A general point of $D$ lies on one of them, since a
rank-four quadric has a unique vertex. The degree of the dual curve is
therefore one, so these planes form a pencil. Their common line contains two
independent quadrics singular at every point of the vertex curve $B$. Therefore, $B$ is contained in the intersection of their singular linear spaces. This
intersection has projective dimension at most two, since two independent
quadrics singular along the same $\PP^3$ are proportional. Hence $B$ is
planar. If $\deg B\geq3$, the restriction to its plane of every quadric in
$W$ vanishes, and the plane is contained in the base. Therefore $B$ is a line
or a smooth conic. Indeed, $B$ is the irreducible image of the parameter
curve; a degree-two image which is not a line is a smooth conic.

Assume that $D$ is an irreducible quadric. Since it is covered by planes, its
rank is three or four, and the planes $\PP(K_p)$ form a ruling. This gives
(Q). All remaining cases give (R).
\end{proof}

If $W$ contains a smooth pencil, it belongs to one of the seven components of
Proposition~\ref{prop:seven-smooth-pencil-components}. Now, we treat the four
branches of Proposition~\ref{prop:four-residual-branches}. We begin with (H1). Set $L=V(u,v,w)$.

\begin{Lemma}\label{lem:H1-normal-forms}
The generic multiplicity of the one-dimensional base along $L$ is at most
three. If it is three, after linear changes one has one of the two forms
$$
\begin{aligned}
W_4&=\langle ux+vw,uy-v^2,uw,uv,xv+yw+w^2+\lambda u^2\rangle,\\
W_5(\mu)&=\langle u^2,uv,ux+vw,uy-v^2,xv+yw+\mu w^2\rangle.
\end{aligned}
$$
If the multiplicity is two, the unmixed part is the genus $-1$ ribbon on
$L$ with ideal $(w,u^2,uv,v^2,yu-xv)$ after coordinates.
\end{Lemma}

\begin{proof}
The pencil of planes in the multiple hyperplane and its vertex line can be
simultaneously normalized. Solving the first-jet equations gives
$
Q_0=xu+a,\, Q_1=yu+b,\, Q_2=c,\, Q_3=d,
\, Q_4=xv+yw+e,
$
where $a,b,c,d,e\in\kk[u,v,w]_2$. At $p=[s:t:0:0:0]\in L$, the first-jet
ideal is $(u,sv+tw)$. Set $(v,w)=(t\zeta,-s\zeta)$ modulo this ideal. The
quadratic terms of $tQ_0-sQ_1,Q_2,Q_3$ vanish for every $[s:t]$ precisely
when
$
va+wb\in(u),\, c,d\in(u).
$
Hence
$
a=wk+ua_1,\, b=-vk+ub_1,\, c=uc_1,\, d=ud_1
$
for linear forms $k,a_1,b_1,c_1,d_1$. If $k\in\langle u\rangle$, the plane
$u=0$ contains a surface component of the base. Therefore, $u,k$ are independent;
set $k=v$ and absorb $a_1,b_1$ into $x,y$. This gives
$
W=\langle ux+vw,uy-v^2,uc_1,ud_1,xv+yw+e\rangle.
$
If the classes of $c_1,d_1$ modulo $u$ have rank two, set them equal to
$w,v$. Additions of the first four generators and linear changes of $x,y$
remove all terms of $e$ except $w^2+\lambda u^2$; the coefficient of $w^2$
can be set equal to one by replacing $y$ with $y+w$, which changes $Q_1$ by
$uw$. This gives $W_4$. If the rank is one, set $c_1=u$, $d_1=v$ and remove all terms of $e$ except $\mu w^2$, obtaining $W_5(\mu)$.
Rank zero contradicts the independence
of $Q_2,Q_3$.

If the generic multiplicity is two, the same coefficient comparison gives,
after row operations,
\[
I_Y=(w,u^2,uv,v^2,yu-xv).
\]
This is scheme-theoretic: on the chart $x\neq0$, putting $\tau=y/x$ gives
\[
I_Y=(w,v-\tau u,u^2),
\]
so $\mathcal O_Y\simeq\kk[\tau,u]/(u^2)$ and the support is the reduced line
$L$. The analogous computation on $y\neq0$ glues to the stated genus $-1$
ribbon.

We are left with the bound on the multiplicity. On the chart $t=1$, set
$z=sv+w$. After eliminating the two first jets, write
$k=\alpha u+\gamma z+\beta(s)v$, where $\beta(s)\neq0$ for general $s$.
The equation $Q_1$ gives $u=\beta(s)v^2+O(v^3)$. The first non-zero terms of
$Q_2,Q_3$ are
$\beta(s)c_1(0,1,-s)v^3$ and
$\beta(s)d_1(0,1,-s)v^3$. Multiplicity at least four would make both vanish
for general $s$, so $c_1,d_1\in\langle u\rangle$, a contradiction. This also proves that the multiplicity-two unmixed structure is exactly the displayed ribbon.
\end{proof}

\begin{Proposition}\label{prop:H1-absorbed}
Every system in $\text{(H1)}$ belongs to the closure of one of the components with
general base $L_1\sqcup L_2+p^{\PP^2}+q$, $R_4+q$, or $E_5$.
\end{Proposition}

\begin{proof}
Assume first that the multiplicity along $L$ is two and let $Y$ be the
ribbon. The first-jet image bundle along $L$ is
$\mathcal O_L(1)\oplus\mathcal O_L$: in the coordinates used in
Lemma~\ref{lem:H1-normal-forms} it is generated by the $u$-direction and by
$xv+yw$. Hence the determinant of the first jets of a general pencil is a
general section of $\mathcal O_L(1)$. It has one simple zero $p$ and is
non-zero elsewhere, so the complete intersection $X$ cut out by the pencil
is smooth along $L\setminus\{p\}$.

Work on the chart $x=1$, write $\tau=y/x$, and put $z=v-\tau u$. Locally
$I_Y=(w,z,u^2)$. At the simple zero of the determinant, row operations on the
pencil and elimination of the equation having non-zero first jet reduce the
other equation to
$$
\tau z+a u^2+\text{terms of order at least three},
$$
where $a\neq0$ for a general pencil. Indeed, the local base ideal modulo the
first equation is $(z,u^2)$, so the vanishing of $a$ is a proper closed
condition on the pencil. The quadratic form $\tau z+a u^2$ is
non-degenerate. The formal Morse lemma, followed by rescaling, therefore
gives
$$
\widehat{\mathcal O}_{X,p}\simeq
\kk[[\tau,z,u]]/(\tau z+u^2).
$$
Inside this ring the ideal $(z,u^2)$ equals $(z)$, because
$u^2=-\tau z$. Thus $Y$ is Cartier on $X$ and
$$
I_{Y/X}=(z),\qquad
\widehat{\mathcal O}_{Y,p}\simeq\kk[[\tau,u]]/(u^2).
$$
Consequently a general pencil cuts out a normal $(2,2)$ surface with exactly
one $A_1$ singularity along $Y$.
Since $H\cdot Y=2$ and $p_a(Y)=-1$, adjunction gives $Y^2=-2$. The residual
net has class $D=2H-Y$, hence $D^2=6$ and $K_X\cdot D=-6$. The minimal
resolution of the $A_1$ point is crepant. After resolving the residual base
points, the Noether equalities are
$
\sum m_i^2=5,\qquad \sum m_i=3.
$
Their only positive integral solution is $(2,1)$. For a general system these
are distinct proper points, so the residual scheme is
$p^{T_pX}\cup\{q\}=p^{\PP^2}+q$.

The required smoothing is explicit. With parameter $\tau$, set
\begin{align*}
\mathcal I_\tau=(w,\ yu-xv,\ u^2-\tau^2x^2,
 uv-\tau^2xy,\ v^2-\tau^2y^2).
\end{align*}
In the polynomial ring $\kk[\tau,x,y,u,v,w]$ one has the exact equality
$$
\mathcal I_\tau=
(w,u-\tau x,v-\tau y)\cap(w,u+\tau x,v+\tau y).
$$
At $\tau=0$ this is $I_Y$, while for $\tau\neq0$ the two ideals define skew
lines. Moreover, for the lexicographic order
$u>v>w>x>y>\tau$ the five displayed generators are a Gr\"obner basis with
initial ideal
$$
(u^2,uv,uy,v^2,w),
$$
which is independent of $\tau$. Every graded piece of the quotient is
therefore a free $\kk[\tau]$-module, and the family is flat with Hilbert
polynomial $2m+2$. The double and simple residual points may be prolonged
as disjoint sections. Since all fibers impose ten independent conditions on
quadrics on a non-empty open set, cohomology and base change gives the
closure relation with the family
$L_1\sqcup L_2+p^{\PP^2}+q$. Collisions and infinitely near residual points
lie in the same closure.

Consider $W_4$. Set $q_i=\rho y_i$ and work on $u\neq0$. If
$r=\rho/u^2$, the coordinates $uw,uv,ux+vw,uy-v^2$ give
$$
\frac wu=ry_2,\qquad \frac vu=ry_3,\qquad
\frac xu=ry_0-r^2y_2y_3,\qquad
\frac yu=ry_1+r^2y_3^2.
$$
The last coordinate becomes
$
(y_0y_3+y_1y_2+y_2^2)r^2-y_4r+\lambda=0.
$
If $\lambda\neq0$, this equation has two distinct non-zero roots on a dense
open subset of the target, and each root gives one source point. Therefore, the map
has degree two. Hence $\lambda=0$. For $\lambda=0$, the six quadrics containing the curve $C$ are the minors of
$$
\begin{pmatrix}
u&-v&-y-w&x\\0&u&v&w\end{pmatrix}.
$$
Its radical is $(u,v,w)\cap(u,v,y+w)$, so the ideal has height
three and the Eagon--Northcott resolution is
$$
0\longrightarrow\OO_{\PP^4}(-4)^{\oplus3}\longrightarrow
\OO_{\PP^4}(-3)^{\oplus8}\longrightarrow
\OO_{\PP^4}(-2)^{\oplus6}\longrightarrow\I_C\longrightarrow0.
$$
Therefore, $P_C(m)=4m+1$ and $h^0(\I_C(2))=6$. The missing minor $u^2$ does not
vanish at $q=[0:0:1:0:0]$, while the five generators of $W_4$ do. Therefore
$W_4=H^0(\I_{C\cup\{q\}}(2))$. A general deformation of the matrix has
height three and defines a rational normal quartic. The relative
Eagon--Northcott complex gives flatness, so $W_4$ belongs to the closure of
the $R_4+q$ component.

For $W_5(\mu)$ work on $u\neq0$ and set $r=v/u$, $z=w/u$. After
dividing the three remaining coordinates by $u^2$, set
$A=x/u+rz$, $B=y/u-r^2$, and let $E$ denote the last coordinate. Then
$E=rA+Bz+\mu z^2$. If $\mu\neq0$, a general target has two distinct
values of $z$, each determining $x/u,y/u$ uniquely. Therefore, birationality gives
$\mu=0$. Finally, $W_5=W_5(0)$ is generated by the principal Pfaffians of
$$
\begin{pmatrix}
0&0&0&u&v\\
0&0&u&-w&x\\
0&-u&0&v&y\\
-u&w&-v&0&0\\
-v&-x&-y&0&0
\end{pmatrix}.
$$
Its radical is $(u,v,w)\cap(u,v,y)$, so the Pfaffian ideal has
height three. The
Buchsbaum--Eisenbud complex is therefore exact and gives
$$
0\longrightarrow\OO_{\PP^4}(-5)\longrightarrow
\OO_{\PP^4}(-3)^{\oplus5}\longrightarrow
\OO_{\PP^4}(-2)^{\oplus5}\longrightarrow\I\longrightarrow0.
$$
Hence its Hilbert polynomial is $5m$ and its quadratic part is exactly $W_5$.
For a general alternating deformation the Pfaffian ideal still has height
three; the same relative complex gives a flat family whose general fiber is a
smooth elliptic normal quintic. Therefore, $W_5$ belongs to the closure of the
$E_5$ component.
\end{proof}

Now, consider (H2). Set $C=V(x_0x_2-x_1^2,u,v)$.

\begin{Lemma}\label{lem:H2-normal-form}
After linear changes every system in $\text{(H2)}$ has the form
$$
\begin{aligned}
Q_0&=x_0u+x_1v+a(u,v),&
Q_1&=x_1u+x_2v+b(u,v),\\
Q_2&=c(u,v),&Q_3&=d(u,v),\\
Q_4&=x_0x_2-x_1^2+x_0n_0+x_1n_1+x_2n_2+e,
\end{aligned}
$$
where $a,b,c,d,e$ are quadratic and the $n_i$ are linear in $u,v$.
\end{Lemma}

\begin{proof}
Parametrize the vertex conic by
$
[s:t]\mapsto[s^2:st:t^2:0:0].
$
Its first-jet bundle for quadrics is
\begin{align*}
N^*_{C/\PP^4}(2)
&\simeq
\bigl(\mathcal O_{\PP^1}(-4)\oplus
      \mathcal O_{\PP^1}(-2)^{\oplus2}\bigr)
      \otimes\mathcal O_{\PP^1}(4)\\
&\simeq\mathcal O\oplus\mathcal O(2)^{\oplus2}.
\end{align*}
The repeated-hyperplane incidence says that the component in the two normal
directions to the plane of $C$ has rank one at every $[s:t]$. Thus its
$2\times2$ matrix of binary quadrics has rank one and no common zero. Factoring
its rows and columns, the two factors must both have degree one; after changes
of bases the matrix is
$
\begin{pmatrix}s^2&st\\ st&t^2\end{pmatrix}.
$
Equivalently, two generators have first jets
$
s(su+tv),\qquad t(su+tv),
$
which lift to $x_0u+x_1v$ and $x_1u+x_2v$. The remaining independent first jet must have non-zero component in the
trivial summand $\mathcal O$. Indeed, a quotient of positive degree of
$\mathcal O(2)^{\oplus2}$ has a zero for every section, so without the
trivial component the total first-jet rank would drop to one somewhere.
Normalize this nowhere-zero component to the equation
$x_0x_2-x_1^2$ of the conic. The other two
generators have zero first jets and are therefore binary quadrics in $u,v$.
Finally, adding the first four generators to the fifth and translating the
normal coordinates removes all terms except the displayed
$x_i n_i$ and $e$. This gives the asserted normal form.
\end{proof}

The first jets define a double structure $Y$ on $C$ with
$$
0\longrightarrow\OO_{\PP^1}(-1)\longrightarrow\OO_Y
\longrightarrow\OO_C\longrightarrow0
$$
and
$
I_Y=(x_0x_2-x_1^2,x_0u+x_1v,x_1u+x_2v,u^2,uv,v^2).
$
The graded parametrization
$x_0=s^2$, $x_1=st$, $x_2=t^2$, $u=\epsilon t$, $v=-\epsilon s$,
$\epsilon^2=0$, identifies the degree $m$ part of the coordinate ring with
$H^0(\OO_{\PP^1}(2m))\oplus\epsilon H^0(\OO_{\PP^1}(2m-1))$. Hence its
Hilbert function is $4m+1$ for $m\geq0$ and the restriction maps from
$H^0(\OO_{\PP^4}(m))$ are surjective. Therefore, $Y$ is ACM. By the main theorem
of \cite{MartinDeschampsPiene97}, the ACM curves with Hilbert polynomial
$4m+1$ form an irreducible locus containing the rational normal quartics.

\begin{Proposition}\label{prop:H2-absorbed}
Every system in (H2) belongs to the closure of the component with general base
$R_4+q$ or $E_5$.
\end{Proposition}

\begin{proof}
If $c,d$ are coprime, $[u:v]\mapsto[c:d]$ is a morphism of degree two.
For a fixed value of $[u:v]$, the first two equations leave an affine line of
solutions in $(x_0,x_1,x_2)$ with direction
$
(v^2,-uv,u^2).
$
The quadratic form $x_0x_2-x_1^2$ vanishes on this direction, because
$v^2u^2-(-uv)^2=0$. Hence $Q_4$ restricts linearly to that affine line and,
on a dense open set, selects a unique source point. A general target has two
preimages under $[u:v]\mapsto[c:d]$, and each gives one source point.
Therefore the map has degree two.

Hence $c,d$ have a common linear factor. Since they are independent, after
changes one may set $c=u^2$, $d=uv$. On $u=0$ the residual base is contained
in a line $R$ meeting $C$ at one point $p$. If $Q_4|_R\neq0$, its divisor is $p+q$ and the base is
$Y\cup\{q\}$. The locus of ACM curves with Hilbert polynomial $4m+1$ is
irreducible by \cite{MartinDeschampsPiene97}; its smooth locus is non-empty,
since it contains rational normal quartics, and is therefore dense. Thus
there is a flat family deforming $Y$ to a rational normal quartic. After a
finite base change, the isolated point $q$ prolongs to a section disjoint from
the curve. The restriction maps on quadrics are surjective for both the
special ACM curve and a rational normal quartic, so cohomology and base change
identifies the limiting five-dimensional complete systems. This proves the
closure relation with $R_4+q$.

If $Q_4|_R=0$, coordinate changes give
$
I_{Y\cup R}=(u^2,uv,x_0u+x_1v,x_1u+x_2v,x_0x_2-x_1^2).
$
These are the principal Pfaffians of
$$
\begin{pmatrix}
0&0&x_1&-x_0&v\\
0&0&x_2&-x_1&-u\\
-x_1&-x_2&0&u&0\\
x_0&x_1&-u&0&0\\
-v&u&0&0&0
\end{pmatrix}.
$$
The ideal has height three. The Buchsbaum--Eisenbud resolution gives Hilbert
polynomial $5m$ and shows that the five displayed quadrics are the complete
quadratic system. A general alternating deformation remains of height three,
and the relative resolution gives a flat family with general fiber a smooth
elliptic normal quintic.
\end{proof}


\begin{Lemma}\label{lem:Q-ruling-normal-forms}
In branch \textup{(Q)}, the reduced vertex curve is a line. After source and
target changes, with that line equal to $B=V(u,v,z)$, the following first-jet
normal forms hold.
\begin{enumerate}
\item If the parameter quadric has rank four, then
\begin{align*}
Q_0&=x_1u+a,&Q_1&=-x_0u+b,\\
Q_2&=x_1v+c,&Q_3&=-x_0v+d,&Q_4&=e,
\end{align*}
where $a,b,c,d,e\in\kk[u,v,z]_2$.
\item If the parameter quadric has rank three, then
\begin{align*}
Q_0&=x_1u+a,&Q_1&=-x_0u+x_1v+b,&Q_2&=-x_0v+c,\\
Q_3&=d,&Q_4&=e,
\end{align*}
where $a,b,c,d,e\in\kk[u,v,z]_2$.
\end{enumerate}
\end{Lemma}

\begin{proof}
Let $\nu:\PP^1\to B$ be the normalization and put
$d=\deg\nu^*\mathcal O_{\PP^4}(1)$. For a rank-four parameter quadric, a
ruling plane has tautological bundle
$
\mathcal K\simeq\mathcal O(-1)^{\oplus2}\oplus\mathcal O;
$
for rank three it has
$
\mathcal K\simeq\mathcal O(-2)\oplus\mathcal O^{\oplus2}.
$
In both cases the Pl\"ucker degree is $2$. If $A_i$ are symmetric matrices for
a basis of $W$ and $x(t)$ represents $\nu(t)$, the first-jet matrix has
columns $A_i x(t)$, whose entries have degree $d$. Its $2\times2$ minors have
degree $2d$ and no common zero, because $r_p(W)=2$ along the vertex curve.
They are the Pl\"ucker coordinates of the kernel plane. Hence $2d=2$, so
$d=1$ and $B$ is a line.

The quotient bundle
$
\mathcal Q=(W\otimes\mathcal O_{\PP^1})/\mathcal K
$
has rank two and degree two. The first-jet map embeds it in
$
N^*_{B/\PP^4}(2)\simeq\mathcal O(1)^{\oplus3}.
$
Since $\mathcal Q$ is globally generated, its two summands have non-negative
degree; the embedding forces each of them to have degree at most one.
Consequently $\mathcal Q\simeq\mathcal O(1)^{\oplus2}$, and, after twisting
by $\mathcal O(-1)$, its image is a constant two-dimensional subspace of
$\mathcal O^{\oplus3}$. A constant change of the normal coordinates and a
change of basis of $W$ therefore put the first jets in the forms
\begin{align*}
&(x_1u,-x_0u,x_1v,-x_0v,0) &&\text{in rank four},\\
&(x_1u,-x_0u+x_1v,-x_0v,0,0) &&\text{in rank three}.
\end{align*}
A quadric with zero first jet along $B$ is a quadratic form in $u,v,z$.
This gives the two displayed normal forms.
\end{proof}

\begin{Proposition}\label{prop:Q-absorbed}
Every system in \textup{(Q)} belongs to the closure of one of the two
components whose general bases are
$
C\sqcup L+q_1+q_2
$
and
$
R_4+q.
$
\end{Proposition}

\begin{proof}
By Lemma~\ref{lem:Q-ruling-normal-forms}, the vertex curve is a line. Assume
first that the parameter quadric has rank four. Write
\begin{align*}
Q_0&=x_1u+a,&Q_1&=-x_0u+b,\\
Q_2&=x_1v+c,&Q_3&=-x_0v+d,&Q_4&=e,
\end{align*}
with $a,b,c,d,e\in S_2$, where $S=\kk[u,v,z]$, and set
$
F=va-uc,\qquad G=vb-ud.
$
For a general target point $y=[y_0:\dots:y_4]$, a normal direction
$[u:v:z]$ with $e\neq0$ gives a point of the inverse image precisely when
\stepcounter{thm}\begin{equation}\label{eq:Q-rank-four-inverse-equations}
\begin{aligned}
(y_0v-y_2u)e-y_4F&=0,\\
(y_1v-y_3u)e-y_4G&=0.
\end{aligned}
\end{equation}
Once $[u:v:z]$ is known and $(u,v)\neq(0,0)$, the four original linear
equations determine $x_0,x_1$ uniquely.

Consider the rational map
\stepcounter{thm}\begin{equation}\label{eq:Q-plane-map-four-coordinates}
\Psi:\PP^2\dashrightarrow\PP^3,
\qquad [u:v:z]\longmapsto[ue:ve:F:G].
\end{equation}
The two equations in \eqref{eq:Q-rank-four-inverse-equations} say that
$\Psi([u:v:z])$ lies on the line
\begin{align*}
y_0X_1-y_2X_0-y_4X_2&=0,\\
y_1X_1-y_3X_0-y_4X_3&=0.
\end{align*}
As $y$ varies with $y_4\neq0$, these are exactly the lines in the dense open
chart of $\Gr(1,3)$ which are graphs over the $(X_0,X_1)$-plane. A general
such line avoids the indeterminacy locus of $\Psi$ and the fixed line
$X_0=X_1=0$. Therefore, the number of points of a general inverse image is
$
\deg(\Psi)\deg\overline{\Psi(\PP^2)}.
$
Birationality of the original map forces both factors to be one. Hence
$\Psi$ is birational onto a plane.

There is consequently a relation
$
\lambda_0ue+\lambda_1ve+\lambda_2F+\lambda_3G=0,
$
with $(\lambda_2,\lambda_3)\neq(0,0)$. A simultaneous linear change of the
coordinates $x_0,x_1$ and of the two pairs $(Q_0,Q_2)$, $(Q_1,Q_3)$ makes the
last two coefficients $(1,0)$. Adding constant multiples of $Q_4$ to
$Q_0,Q_2$ then removes the term $(\lambda_0u+\lambda_1v)e$. Thus we may
assume $F=0$. Since $va=uc$ and $u,v$ are coprime, there is a linear form
$R$ with $a=uR$ and $c=vR$; replacing $x_1$ by $x_1+R$ gives
\stepcounter{thm}\begin{equation}\label{eq:Q-rank-four-reduced-form}
W=\langle x_1u,-x_0u+b,x_1v,-x_0v+d,e\rangle.
\end{equation}
The map \eqref{eq:Q-plane-map-four-coordinates} is now the plane map
\stepcounter{thm}\begin{equation}\label{eq:Q-plane-cremona-map}
\chi=[ue:ve:G]:\PP^2\dashrightarrow\PP^2.
\end{equation}
It is birational.

Let $H=\gcd(e,G)$, write $e=He_0$, $G=Hg_0$, and put $h=\deg H$.
After cancelling $H$, the first two coordinates of \eqref{eq:Q-plane-cremona-map}
still have ratio $u:v$. On the chart $u\neq0$, with
$r=v/u$ and coefficient field $K=\kk(r)$, birationality is equivalent to
\stepcounter{thm}\begin{equation}\label{eq:Q-rational-function-degree-one}
K(z)=K\bigl(g_0(1,r,z)/e_0(1,r,z)\bigr).
\end{equation}
Since $e_0$ and $g_0$ are coprime, the rational function on the right has
degree one. In particular, both numerator and denominator have degree at
most one in $z$.

Suppose first that $h=0$. Then the coefficients of $z^2$ in both $e$ and
$G$ vanish. If $b_2,d_2$ are the coefficients of $z^2$ in $b,d$, the
coefficient of $z^2$ in $G=vb-ud$ is $vb_2-ud_2$, so $b_2=d_2=0$.
Consequently every generator in \eqref{eq:Q-rank-four-reduced-form} vanishes
on the plane $V(u,v)$, contrary to $\dim\Bs(W)=1$. Hence $h\geq1$.

Assume that $h=1$. Write
$
H\in S_1,\quad e=Hl,\quad G=Hq,
$
where $l\in S_1$ and $q\in S_2$. By
\eqref{eq:Q-rational-function-degree-one}, $q$ has no $z^2$ term, hence
$q\in(u,v)$. Choose linear forms $b_1,d_1$ such that
$
q=vb_1-ud_1.
$
The equality
$
v(b-Hb_1)-u(d-Hd_1)=0
$
and exactness of the Koszul complex of $u,v$ give a linear form $R$ with
$
b=Hb_1+uR$, $d=Hd_1+vR$. Replacing $x_0$ by $x_0-R$ yields
\stepcounter{thm}\begin{equation}\label{eq:Q-rank-four-h-one}
W=\langle x_1u,-x_0u+Hb_1,x_1v,-x_0v+Hd_1,Hl\rangle,
\qquad q=vb_1-ud_1.
\end{equation}
After cancelling $H$, the associated plane map is the quadratic Cremona map
$[ul:vl:q]$.

We first treat the dense open set on which $H(o)l(o)\neq0$ for
$o=V(u,v)$, the scheme $Z=V(l,q)\subset\PP^2$ consists of two distinct
points, and it is disjoint from $V(H)\cup\{o\}$. In these coordinates set
\begin{align*}
Y&=V(u,v,Hl)=V(u,v,z^2),\\
M&=V(x_0,x_1,H).
\end{align*}
Thus $Y$ is a plane double line, with Hilbert polynomial $2m+1$, and $M$ is
a line skew to its plane. The two points of $Z$ lift uniquely to points
$\widetilde Z\subset V(x_1)$ by
$
 x_0u=Hb_1,\ x_0v=Hd_1.
$
Every quadric in \eqref{eq:Q-rank-four-h-one} vanishes on
$Y\cup M\cup\widetilde Z$. After setting $H=z$, the quadrics through
$Y\cup M$ are
\stepcounter{thm}\begin{equation}\label{eq:Q-seven-quadrics}
\langle x_0u,x_1u,x_0v,x_1v,zu,zv,z^2\rangle.
\end{equation}
On the stated dense open set the two reduced points impose two independent
conditions on this seven-dimensional space. Hence
\stepcounter{thm}\begin{equation}\label{eq:Q-complete-h-one}
W=H^0\bigl(\mathcal I_{Y\cup M\cup\widetilde Z}(2)\bigr).
\end{equation}
The conic $Y$ is smoothed in its plane by
$
z^2-\tau x_0x_1=0,
$
and the length-two scheme $\widetilde Z$ is smoothable. The incidence of a
smooth plane conic, a line skew to its plane, and two exterior points is
irreducible, and its general complete quadratic system belongs to the
component with base $C\sqcup M+q_1+q_2$. Cohomology and base change applied to
\eqref{eq:Q-complete-h-one} therefore place $W$ in its closure. The
generators in
\eqref{eq:Q-rank-four-h-one} depend polynomially on
$H,l,b_1,d_1$, and the preceding open conditions are dense. Every special
system with $h=1$ is consequently a specialization of systems in the same
closed component.

It remains to consider $h=2$. Then $G=eL$ for a linear form $L$, and the
cancelled map is $[u:v:L]$. It is an automorphism, so $u,v,L$ are independent.
Adding suitable multiples of $e$ to $b,d$ changes $L$ by an arbitrary element
of $\langle u,v\rangle$; hence we may set $L=z$. Since $ze=G\in(u,v)$ and
$(u,v)$ is prime, $e\in(u,v)$. Write
$
e=z\ell(u,v)+q(u,v).
$
If $\ell=0$, the coefficient of $z^2$ in $G$ is zero, hence the coefficients
of $z^2$ in $b,d$ vanish; the plane $V(u,v)$ would again be contained in the
base. Thus $\ell\neq0$. A change of $u,v$, followed by replacing $z$ by
$z+\alpha u+\beta v$, gives
$
e=uz+cv^2
$
with $c=0$ or $1$. The identity $vb-ud=z(uz+cv^2)$ and the Koszul complex
then allow a final change of $x_0$ which gives
\stepcounter{thm}\begin{equation}\label{eq:Q-rank-four-h-two}
W_c=\langle x_1u,-x_0u+cvz,x_1v,-x_0v-z^2,uz+cv^2\rangle.
\end{equation}
The case $c=0$ is a specialization of $c=1$.

We now exhibit $W_1$ as an explicit limit of complete systems through a
rational normal quartic and one point. Put
\stepcounter{thm}\begin{equation}\label{eq:Q-Y-coordinates}
Y_0=x_1,\qquad Y_1=-x_0,\qquad Y_2=-z,\qquad
Y_3=v,\qquad Y_4=u.
\end{equation}
For a parameter $\tau$, let $C_\tau$ be defined by the six quadrics
\begin{align*}
A_\tau&=Y_0Y_2-\tau Y_1^2,&
B_\tau&=Y_0Y_3-\tau Y_1Y_2,&
C_\tau&=Y_0Y_4-\tau Y_1Y_3,\\
D&=Y_1Y_3-Y_2^2,&
E&=Y_1Y_4-Y_2Y_3,&
F&=Y_2Y_4-Y_3^2.
\end{align*}
This is the Rees degeneration of the rational normal quartic ideal for the
weight assigning weight one to $Y_0$ and weight zero to the other variables,
so it is flat. Equivalently, for $\tau\neq0$ the substitution
$Y_0=\tau Y_0'$ gives the six $2\times2$ minors of
$
\begin{pmatrix}Y_0'&Y_1&Y_2&Y_3\\Y_1&Y_2&Y_3&Y_4\end{pmatrix},
$
whereas at $\tau=0$ the fiber is the union of the line
$V(Y_2,Y_3,Y_4)$ and the twisted cubic
$V(Y_0,D,E,F)$, meeting in one point; both fibers have Hilbert polynomial
$4m+1$.

Set
$
q_\tau=[2:1:\tau:\tau^2:\tau^3].
$
For $\tau\neq0$, this point is not on $C_\tau$, and
\begin{align*}
A_\tau(q_\tau)&=\tau,&B_\tau(q_\tau)&=\tau^2,&
C_\tau(q_\tau)&=\tau^3,&D(q_\tau)&=E(q_\tau)=F(q_\tau)=0.
\end{align*}
Therefore
\stepcounter{thm}\begin{equation}\label{eq:Q-rational-quartic-limit-system}
H^0\bigl(\mathcal I_{C_\tau\cup\{q_\tau\}}(2)\bigr)
=\langle B_\tau-\tau A_\tau,
C_\tau-\tau^2A_\tau,D,E,F\rangle.
\end{equation}
At $\tau=0$, the right-hand side becomes
$\langle B_0,C_0,D,E,F\rangle$, which is exactly $W_1$ under
\eqref{eq:Q-Y-coordinates}. For $\tau\neq0$ the pair
$(C_\tau,q_\tau)$ belongs to the irreducible incidence of a rational normal
quartic and an exterior point; its general complete quadratic system
belongs to the $R_4+q$ component. Hence $W_1$, and therefore every system with $h=2$, lies
in the closure of that component.

Assume now that the parameter quadric has rank three. By
Lemma~\ref{lem:Q-ruling-normal-forms},
$
Q_0=x_1u+a,\quad Q_1=-x_0u+x_1v+b,\quad Q_2=-x_0v+c,
\quad Q_3=d,\quad Q_4=e,
$
with $a,b,c,d,e\in\kk[u,v,z]_2$. Set
$F=v^2a-uvb+u^2c$. On the conic $C_\tau=V(e-\tau d)$ the inverse problem is
the moving linear series generated by
$
F,\, dv^2,\,-duv,\, du^2.
$
Once a point of $C_\tau$ is fixed and $[u:v]\neq[0:0]$, the first three
quadrics determine $x_0,x_1$ uniquely. Hence the topological degree is the
degree of this moving series. If it were one, its image would be a line. Its
projection to the last three coordinates lies in the Veronese conic
$[v^2:-uv:u^2]$, so this projection would be constant. Therefore, $[u:v]$ is
constant on every moving irreducible component of $C_\tau$.

If $d,e$ are coprime and the general $C_\tau$ is irreducible, this
contradicts the preceding constancy. If it is reducible, constancy forces both
components to be lines through $o=V(u,v)$. Hence $d,e$ are binary quadrics and
$[u:v]\mapsto[d:e]$ has degree two. The moving degree is again not one.
Therefore $d,e$ have a common linear factor. Write $d=kd_1$, $e=ke_1$ and set
$r=V(d_1,e_1)$. The moving component of $C_\tau$ is the line
$V(e_1-\tau d_1)$. Constancy of $[u:v]$ on all members forces this pencil to
be centered at $o=V(u,v)$. After changes $d=ku$, $e=kv$.

On $v=\tau u$, write $F=u^2H_\tau(u,z)$. After removing $u^2$, the moving
pencil is generated by $H_\tau$ and $uk(u,\tau u,z)$. Degree one forces a
common linear factor. If the factor is $u$ for general $\tau$, the
coefficient of $z^2$ in $H_\tau$ vanishes identically, so
$a,b,c\in(u,v)$ and the plane $V(u,v)$ is in the base. Therefore, $k\notin\langle u,v\rangle$; set $k=z$. The common factor is then $z$, so
$z\mid F$.

The syzygy module of $(v^2,-uv,u^2)$ is generated by $(u,v,0)$ and
$(0,-u,-v)$. Replacing $x_0,x_1$ and adding $zu,zv$, one obtains
$
W_{\mathbf r}=\langle x_1u+r_2z^2,-x_0u+x_1v-r_1z^2,
-x_0v+r_0z^2,zu,zv\rangle.
$
The base contains $Y=V(u,v,z^2)$ and $M=V(x_0,x_1,z)$. Their quadratic
ideal is the seven-dimensional space
\stepcounter{thm}\begin{equation}\label{eq:Q-rank-three-seven-space}
V_0=\langle x_0u,x_1u,x_0v,x_1v,zu,zv,z^2\rangle.
\end{equation}
We now realize the two conditions cutting out $W_{\mathbf r}$ from $V_0$
as limits of evaluations at two reduced points. It is enough to treat the
dense open set on which $r_0r_2\neq0$ and
\stepcounter{thm}\begin{equation}\label{eq:Q-rank-three-root-equation}
\alpha^2+r_1\alpha+r_0r_2=0
\end{equation}
has two distinct roots $\alpha_1,\alpha_2$; all other values of
$\mathbf r$ are specializations.

For a parameter $t$, let
\stepcounter{thm}\begin{equation}\label{eq:Q-rank-three-conic-family}
C_t=V\bigl(u,v,z^2-t^3x_0x_1\bigr),\qquad
M=V(x_0,x_1,z),
\end{equation}
and, after the harmless finite base change which orders the two roots, set
\stepcounter{thm}\begin{equation}\label{eq:Q-rank-three-point-sections}
p_{i,t}=
[\alpha_i:-r_2:t^2:(r_0/\alpha_i)t^2:t]
\qquad (i=1,2).
\end{equation}
For $t\neq0$, $C_t$ is a smooth conic, $M$ is skew to its plane, and the two
points are distinct and disjoint from $C_t\cup M$. The quadrics through
$C_t\cup M$ form the free rank-seven module
\stepcounter{thm}\begin{equation}\label{eq:Q-rank-three-seven-family}
V_t=\langle x_0u,x_1u,x_0v,x_1v,zu,zv,
 z^2-t^3x_0x_1\rangle.
\end{equation}
Divide evaluation at $p_{i,t}$ by $t^2$. In the ordered basis displayed in
\eqref{eq:Q-rank-three-seven-family}, its limit is
\stepcounter{thm}\begin{equation}\label{eq:Q-rank-three-limit-functional}
\ell_i=(\alpha_i,-r_2,r_0,-r_0r_2/\alpha_i,0,0,1).
\end{equation}
Equation~\eqref{eq:Q-rank-three-root-equation} shows directly that every
generator of $W_{\mathbf r}$ is annihilated by both $\ell_1$ and $\ell_2$.
The two functionals are independent, and both their common kernel and
$W_{\mathbf r}$ have dimension five; hence
\stepcounter{thm}\begin{equation}\label{eq:Q-rank-three-kernel-limit}
W_{\mathbf r}=\ker(\ell_1)\cap\ker(\ell_2)
 =\lim_{t\to0}H^0\bigl(\mathcal I_{C_t\cup M\cup
 \{p_{1,t},p_{2,t}\}}(2)\bigr).
\end{equation}
For $t\neq0$ these are complete systems attached to the irreducible
incidence of a smooth conic, a line skew to its plane, and two exterior
points; its general complete quadratic system belongs to the
$C\sqcup L+q_1+q_2$ component. Thus the general $W_{\mathbf r}$, and by
specialization every $W_{\mathbf r}$, lies in its closure.
\end{proof}

\begin{Lemma}\label{lem:rank-one-matrix-spaces}
Let $\mathcal T\subset\operatorname{Hom}(A,B)$ be a linear space whose
elements have rank at most one. Then the non-zero maps in $\mathcal T$ have
a common kernel hyperplane or a common image line.
\end{Lemma}

\begin{proof}
Fix a non-zero map $T_0=b_0\otimes\alpha_0$. Every non-zero
$T=b\otimes\alpha$ in $\mathcal T$ must satisfy either
$b\in\langle b_0\rangle$ or $\alpha\in\langle\alpha_0\rangle$, since
otherwise $T_0+T$ has rank two. If all maps satisfy the first alternative,
they have the common image line $\langle b_0\rangle$. Otherwise choose
$T_1=b_1\otimes\alpha_0$ with $b_1$ independent of $b_0$. For any further
$T=b\otimes\alpha$, applying the same rank-two test to both $T_0+T$ and
$T_1+T$ forces $\alpha\in\langle\alpha_0\rangle$. Hence all maps have the
common kernel hyperplane $\ker\alpha_0$.
\end{proof}

\begin{Lemma}\label{lem:symmetric-rank-one-space}
Let $\mathcal N\subset\operatorname{Sym}^2(E^*)$ be a linear subspace such
that every element has rank at most one. Then $\dim\mathcal N\leq1$; in
particular, all non-zero elements are multiples of one fixed square.
\end{Lemma}

\begin{proof}
Every non-zero rank-one symmetric form is a square $\ell^2$. If
$\ell^2,m^2\in\mathcal N$ are independent, then $\ell,m$ are independent and
$\ell^2+m^2$ has rank two, a contradiction.
\end{proof}

\begin{Lemma}\label{lem:bounded-rank-symmetric}
Let $U\subset\operatorname{Sym}^2(V_0^*)$, with $\dim V_0=5$, and assume that
every element of $U$ has rank at most three. Then either all forms have a
common kernel of dimension at least two or there is a subspace $P\subset V_0$
of dimension at least three on which every form vanishes.
\end{Lemma}

\begin{proof}
Choose a form of maximal rank $r\leq3$ and write, with $V_0=E\oplus K$,
$$
A=\begin{pmatrix}A_E&0\\0&0\end{pmatrix},
\qquad
B=\begin{pmatrix}C_B&D_B\\D_B^t&E_B\end{pmatrix},
$$
where $A_E$ is non-degenerate and $\dim K=5-r\geq2$. The coefficients of the
Schur complement of $A+tB$ give $E_B=0$ and
$D_B^tA_E^{-1}D_B=0$. Therefore, the linear space of maps
$T_B=A_E^{-1}D_B:K\to E$ consists of maps of rank at most one, and the span
of the images of any two maps is totally isotropic. Lemma~\ref{lem:rank-one-matrix-spaces}
gives a common kernel or a common image. If all $T_B$ vanish, $K$ is a common
kernel. In the common-kernel case the maps factor through a one-dimensional
quotient of $K$; their images span a totally isotropic subspace of $E$, hence
a line. Therefore, unless all maps vanish, they have a common isotropic image
$\langle u\rangle$. The next coefficient of the Schur complement gives
$C_B(u,u)=0$. Since $D_B=A_Eu\otimes\lambda_B$ and $A_E(u,u)=0$, every $B$
vanishes on $K\oplus\langle u\rangle$.
\end{proof}

\begin{Lemma}\label{lem:rank-two-pencils}
Let $q_{s,t}$ and $r_{s,t}$ be two pencils of quadratic forms on a
four-dimensional vector space. Assume that, for general $[s:t]$, they have a
common varying two-dimensional kernel. Then the four forms have a common
linear factor.
\end{Lemma}

\begin{proof}
First consider the pencil $q_{s,t}$. Choose a rank-two member and put it in
the form $x_0x_1$. Writing the second member as a symmetric matrix and
setting the $3\times3$ minors of $x_0x_1+\tau q$ equal to zero gives two
possibilities: either the kernel $V(x_0,x_1)$ is fixed, or
$$
q_{s,t}=\lambda(s\mu_0+t\mu_1)
$$
for fixed independent linear forms $\lambda,\mu_0,\mu_1$. The first
possibility is excluded by the hypothesis, and the common kernel is therefore
$$
V(\lambda,s\mu_0+t\mu_1).
$$
Modulo $\lambda$, a quadratic form singular along this plane is a multiple
of $(s\mu_0+t\mu_1)^2$. Write
$r_{1,0}\equiv a\mu_0^2$ and
$r_{0,1}\equiv b\mu_1^2$ modulo $\lambda$. Since
$r_{1,1}=r_{1,0}+r_{0,1}$ must be a multiple of
$(\mu_0+\mu_1)^2$, comparison of the coefficient of
$\mu_0\mu_1$ gives that this multiple is zero, and then $a=b=0$.
Consequently both members of the second pencil are divisible by $\lambda$.
The four forms have the common linear factor $\lambda$.
\end{proof}

\begin{Lemma}\label{lem:R-one-point-normal-forms}
Assume that an irreducible quadratic component $D$ of the discriminant has
general rank three and that the general singular line contains one base
point. Then $W$ is not dominant.
\end{Lemma}

\begin{proof}
The associated maximal planes form one ruling of $D$. The same Pl\"ucker
degree argument used in Lemma~\ref{lem:Q-ruling-normal-forms} shows that the base
point curve is a line. If $D$ has rank four we have
$
Q_0=x_1u+a,\, Q_1=-x_0u+b,\, Q_2=x_1v+c,
\, Q_3=-x_0v+d,\, Q_4=e,
$
where the last terms are quadratic in $u,v,w$. Let $(m_{ij})$ be the symmetric
matrix of a general normal part. The condition that every member of a ruling
plane have rank at most three gives
$
\lambda^2(m_{23}^2-m_{22}m_{33})
+2\lambda\mu(m_{12}m_{33}-m_{13}m_{23})
+\mu^2(m_{13}^2-m_{11}m_{33})=0.
$
These identities hold for every element of the linear space of normal
parts. If a diagonal entry is non-zero, they express all entries as the
corresponding rank-one outer product; changing the distinguished normal
coordinate gives all the $2\times2$ minors. If every diagonal entry vanishes,
polarization forces the whole normal part to vanish. Thus every normal part
has rank at most one. By Lemma~\ref{lem:symmetric-rank-one-space}, the normal
parts span at most one fixed square. If they vanish, all quadrics miss one
normal coordinate and the map is not generically finite. Otherwise choose
that square to be $w^2$ and use row operations to remove it from the first
four generators. The system becomes
$
\langle x_1u,-x_0u,x_1v,-x_0v,w^2\rangle,
$
whose image satisfies $y_0y_3-y_1y_2=0$. It is therefore not dominant.

Assume that $D$ has rank three. In the normal form
$
Q_0=x_1u+a,\, Q_1=-x_0u+x_1v+b,\, Q_2=-x_0v+c,
\, Q_3=d,\, Q_4=e,
$
the maximal plane indexed by $[s:t]$ is generated by
$Q_{s,t}=s^2Q_0+stQ_1+t^2Q_2,Q_3,Q_4$. Its mixed part is
$(sx_1-tx_0)(su+tv)$. Set $l=su+tv$. For every
$R\in\langle d,e\rangle$, the coefficient of the square of $R$ in the
$4\times4$ determinant of $Q_{s,t}+R$ is
$\det(R|_{l=0})$; hence it vanishes for every $[s:t]$. Writing a symmetric
matrix in the basis $u,v,w$, the equations
$l^{\mathsf t}\operatorname{adj}(R)l=0$ for all $l\in\langle u,v\rangle$ say that
$R$ either has rank one or has zero last row and column. The Veronese surface
of rank-one forms contains no line. Since $d,e$ are independent, every form
in their pencil has zero last row and column, so $d,e\in\kk[u,v]_2$.

Restrict now $Q_{s,t}+\alpha d+\beta e$ to $l=0$. The determinant of this
binary form vanishes for all $\alpha,\beta$. Since the restrictions of
$d,e$ span a non-zero line for general $[s:t]$, comparison of the coefficient
of this line shows that the $w^2$ and $mw$ coefficients of the restriction of
the normal part of $Q_{s,t}$ vanish, where $m$ complements $l$ in
$\langle u,v\rangle$. Therefore, the $w^2$ coefficients of $a,b,c$ vanish and
their terms linear in $w$ have the form
$(\alpha s+\beta t)l w$. One fixed change of $x_0,x_1$ removes these terms.
All five quadrics then miss $w$, so the map is not generically finite.
\end{proof}

\begin{Lemma}\label{lem:R-same-plane}
Assume that the general singular line contains two base points determining
the same maximal plane of $D$. Then all quadrics contain a plane or have a
common singular point.
\end{Lemma}

\begin{proof}
Let $L_t\subset\PP^4$ be the line joining the two base points and let
$\Pi_t\subset D$ be their common maximal plane. Every quadratic form in
$\Pi_t$ has matrix kernel containing the two-dimensional vector space
underlying $L_t$. The planes $\Pi_t$ span $W$.

Suppose first that the parameter quadric $D$ has rank four. Choose a basis
$Q_0,\ldots,Q_4$ of $W$ so that one ruling is
$$
\Pi_{[s:t]}=
\langle sQ_0+tQ_2,\ sQ_1+tQ_3,\ Q_4\rangle.
$$
The fixed form $Q_4$ kills every $L_t$. Let $S$ be the vector span of all
$L_t$. If the lines have a common point $p$, then every $\Pi_t$, and hence
every form in $W$, is singular at $p$. Assume that they have no common point.
Then $3\leq\dim S\leq4$, since $Q_4\neq0$ kills $S$.

If $\dim S=3$, the restrictions to $S$ of each of the pencils
$sQ_0+tQ_2$ and $sQ_1+tQ_3$ have rank at most one and kernel containing
$L_t$. A linear space of symmetric forms of rank at most one has dimension at
most one by Lemma~\ref{lem:symmetric-rank-one-space}; a non-zero such pencil
would have fixed kernel. Since $L_t$ varies, both restrictions vanish
identically. The restriction of $Q_4$ also vanishes, so all five quadrics
contain the plane $\PP(S)$.

Let $\dim S=4$. If one of the two pencils has a fixed kernel of vector
dimension at least three, restrict the other pencil to that kernel. Its
members have rank at most one with varying kernel, so the preceding argument
shows that it vanishes there; again all forms contain a plane. Otherwise the
two pencils have the common varying two-dimensional kernel $L_t$ on $S$.
Lemma~\ref{lem:rank-two-pencils} gives a common linear factor
$\lambda\in S^*$. Thus $Q_0,Q_1,Q_2,Q_3$ vanish on the plane
$\PP(V(\lambda))\subset\PP(S)$, and $Q_4$ vanishes on all of $\PP(S)$.
This is a common base plane.

Suppose now that $D$ has rank three. Normalize its ruling as
$$
\Pi_{[s:t]}=
\langle s^2Q_0+stQ_1+t^2Q_2,\ Q_3,\ Q_4\rangle.
$$
Both $Q_3$ and $Q_4$ kill the span $S$ of the lines $L_t$. As before, a
common point of the lines is a common singular point of all forms. If there
is no common point, then $\dim S\geq3$. One cannot have $\dim S\geq4$:
the space of symmetric forms killing a fixed four-dimensional subspace is
one-dimensional, whereas $Q_3,Q_4$ are independent. Hence $\dim S=3$.
The restriction
$$
(s^2Q_0+stQ_1+t^2Q_2)|_S
$$
has rank at most one and kernel $L_t$. If it is non-zero, it is the square of
a linear form $s\lambda_0+t\lambda_1$; its kernels are then the pencil of
lines through the fixed point $V(\lambda_0,\lambda_1)$, contrary to the
assumption. Therefore it vanishes identically. Together with $Q_3|_S=Q_4|_S=0$
this shows that every quadric contains the plane $\PP(S)$.
\end{proof}

\begin{Proposition}\label{prop:R-empty}
Branch (R) contains no system satisfying
$\dim\Bs(W)=1$ and $r_p(W)\geq2$ at every base point.
\end{Proposition}

\begin{proof}
Assume first $F_W\equiv0$. If the general quadric has rank four, its vertex
belongs to the base, since the differential of the determinant vanishes on
$W$. The map from $\PP(W)$ to the base has fibers of dimension at least three,
while the fiber over $p$ is contained in $\PP(K_p)$ of dimension at most two.
This is impossible. If the general rank is at most three,
Lemma~\ref{lem:bounded-rank-symmetric} gives a common singular line or a
common plane in the base, again a contradiction.

Let a multiple linear factor have general rank three and let $U\subset W$ be
the corresponding four-dimensional space. Lemma~\ref{lem:bounded-rank-symmetric}
applies to $U$. In the common-kernel case the fifth quadric has a zero on the
common line; at this point the first-jet rank of $W$ is at most one. In the
common-plane case, write $V=\langle x_0,x_1\rangle\oplus P$, with
$P=\langle z_0,z_1,z_2\rangle$. The cross-term matrices of the forms in $U$ are $2\times3$ matrices
of rank at most one. A linear space of such matrices has a common kernel or a common image.
The first alternative gives the preceding common-kernel case. Assume a common
image and let $k$ be the dimension of the cross-term space. If $k\leq1$, its
common kernel in $P$ has vector dimension at least two, again giving the
preceding case. If $k=2$, coordinates give
$
U=\langle x_0z_0,x_0z_1,q_0(x_0,x_1),q_1(x_0,x_1)\rangle.
$
Set $r=[0:0:0:0:1]$. Every member of $U$ and every first jet of a member of
$U$ vanish at $r$. If the fifth quadric $Q$ also vanishes at $r$, then
$r\in\Bs(W)$ and $r_r(W)\leq1$, contrary to the standing assumption. If
$Q(r)\neq0$, restrict the five quadrics to a general line through $r$. The
four members of $U$ become constant multiples of $s^2$, whereas
$
Q|_L=as^2+bst+ct^2,\qquad c=Q(r)\neq0.
$
The resulting map of the line is given by two coprime forms of degree two and
has degree two onto its image. A general line through $r$ meets the open set
on which a birational map is an isomorphism in a dense subset, so this is
impossible. Hence $k=3$, and after
coordinates
$
U=\langle x_0z_0,x_0z_1,x_0z_2,q(x_0,x_1)\rangle.
$
For a general point of the target of the map defined by $U$, the fiber is the
conic parametrized by
$
[s:t]\longmapsto[s^2:st:a_0q(s,t):a_1q(s,t):a_2q(s,t)].
$
The four coordinates of $U$ have the common factor $s^2q(s,t)$. If the fifth
quadric does not vanish on $P$, its restriction to the conic is non-zero at
$s=0$, so the induced rational function has a pole of order two and degree at
least two. If it vanishes on $P$, then $P$ is contained in the base. Therefore, no
birational system occurs.

Finally, consider $F_W=g^2\ell$, where $g$ is irreducible and the
general $Q\in D=V(g)$ has rank three. Let $K_Q=\ker Q$. In a decomposition
$V=E\oplus K_Q$, the degree-two term of $\det(Q+tR)$ gives
$
\det(R|_{K_Q})=c_Q\bigl(dg_Q(R)\bigr)^2
$
for every $R\in W$ and some $c_Q\neq0$. The linear form $dg_Q$ factors
through the image $U_Q$ of $W\to\operatorname{Sym}^2(K_Q^*)$. Since the
determinant is non-degenerate on the three-dimensional space of binary
quadrics, $U_Q$ has dimension one or two. In dimension two its projective
line is tangent to the discriminant conic and its common zero is one reduced
point. In dimension one it is generated by a non-degenerate binary form and
has two distinct zeros. These are the only possibilities on a dense open
subset of $D$.

Let $\mathcal I\to D$ be the incidence of these zeros. After restricting to a
dense open subset it is finite flat of degree one or two. Its image in the
base is a curve. For a general image point $p$, the fiber is contained in
$\PP(K_p)$. Dimension counting gives $r_p(W)=2$ and identifies this fiber
with the maximal plane $\PP(K_p)\subset D$.

If $\mathcal I\to D$ has degree one,
Lemma~\ref{lem:R-one-point-normal-forms} applies. If it has degree two and
the two points determine the same maximal plane,
Lemma~\ref{lem:R-same-plane} applies. Assume that they determine distinct maximal planes. Then $D$ has rank four
and the planes belong to its two rulings. Neither source-point map is
constant: a constant point $p$ would make a positive-dimensional family of
maximal planes lie in $\PP(K_p)$, contradicting $r_p(W)=2$. Parametrize the
two rulings by $\PP^1\times\PP^1$, and let the two source curves have degrees
$d,e\geq1$. Their tautological line bundles give
\stepcounter{thm}\begin{equation}\label{eq:two-rulings-quotient}
0\longrightarrow\mathcal O(-d,0)\oplus\mathcal O(0,-e)
\longrightarrow\mathcal O^{\oplus5}
\longrightarrow\mathcal Q\longrightarrow0.
\end{equation}
The point of $D$ obtained by intersecting the two maximal planes depends
bilinearly on the ruling parameters. Its quadratic form therefore has values
in $\mathcal O(1,1)$, kills the two tautological source lines, and descends to
$
\mathcal Q\longrightarrow\mathcal Q^*\otimes\mathcal O(1,1).
$
It is generically non-degenerate because the general member has rank three.
Since $\det\mathcal Q=\mathcal O(d,e)$, its determinant is a non-zero section
of
$
(\det\mathcal Q)^{-2}\otimes\mathcal O(3,3)
=\mathcal O(3-2d,3-2e).
$
Thus $d,e\leq1$, and non-constancy gives $d=e=1$.

The two source curves are therefore lines. They are skew: if they met at
$p$, the two distinct maximal planes corresponding to the intersection
parameters would both lie in $\PP(K_p)$; since $r_p(W)=2$, this is one plane,
a contradiction. Write the source vector space as
$U\oplus U'\oplus\langle z\rangle$. For $a\in\PP(U)$ and
$b\in\PP(U')$, let $X_a$ and $Y_b$ be linear forms vanishing at the
corresponding points. The universal form has bidegree $(1,1)$ in $(a,b)$.
Terms $X_a^2$ or $Y_b^2$ would have bidegrees $(2,0)$ or $(0,2)$ and hence
cannot occur. Therefore the most general form with both source points in its
kernel is
$
cX_aY_b+z\lambda(b)X_a+z\mu(a)Y_b+\beta(a,b)z^2,
$
where $\lambda,\mu$ are linear and $\beta$ is bilinear. If $c=0$, the induced
form on the three-dimensional quotient has rank at most two, so $c\neq0$.
Constant translations of $U$ and $U'$ by the $z$-direction remove the two
middle terms. Finally, changing the parameter coordinate complementary to
$D$ removes $\beta$. Thus
$
W=\langle x_0x_2,x_0x_3,x_1x_2,x_1x_3,z^2\rangle.
$
Its image satisfies $y_0y_3-y_1y_2=0$, so it is not dominant.

Collisions of the two zeros occur over a proper closed subset of $D$ and do
not give another generic case. If the whole singular line were in the base,
$U_Q=0$, contradicting $dg_Q\neq0$.
\end{proof}

\begin{thm}\label{thm:nine-components}
The space $\Bir_2(\PP^4)$ has exactly nine irreducible components.
\end{thm}

\begin{proof}
Proposition~\ref{prop:nine-distinct-components} gives nine distinct
components. Let $\mathcal K$ be any irreducible component and let $W$ be its
general point. By the definition of $\Bir_2(\PP^4)$, the quadrics in $W$ have no
common divisor. Hence the base has dimension zero, one, or two. In
dimension zero, Theorem~\ref{thm:zero-dimensional-unique-component-geometric}
gives the punctual component. Dimension two is treated by
Propositions~\ref{prop:all-plane-systems} and
\ref{prop:all-quadric-surface-systems}. If the base has dimension one and a
rank-one point, Proposition~\ref{prop:all-rank-one-systems} applies. Otherwise
all first-jet ranks are at least two. A smooth pencil places $W$ in one of the
seven components of Proposition~\ref{prop:seven-smooth-pencil-components}; if
there is no smooth pencil, Proposition~\ref{prop:four-residual-branches} and
Propositions~\ref{prop:H1-absorbed}, \ref{prop:H2-absorbed},
\ref{prop:Q-absorbed}, and \ref{prop:R-empty} apply. Therefore, $\mathcal K$ is one
of the nine components already constructed.
\end{proof}

\section{Magma implementation}\label{appendix:magma}

In this section we describe the computational supplement, which consists of
the files
\path{P4_cremona_all_types.m} and \path{Cremona_P4_checks.m}. The source
archive distributed with this manuscript contains the versions used here;
the public repository is
\begin{center}
\url{https://github.com/msslxa/Cremona_P4}.
\end{center}
The first script uses Magma \cite{BosmaCannonPlayoust97}, works over a large finite field,
and computes the projective degrees of the explicit systems by residual
intersections. For the $k$-th projective degree it intersects $k$ general
linear combinations of the five quadrics with $4-k$ general hyperplanes,
saturates by the base ideal, and computes the degree of the residual
zero-dimensional scheme. It also computes the saturated base ideals. These calculations provide independent finite-field checks of the examples and their multidegrees; they are not used as substitutes for the geometric arguments in the text. \texttt{Cremona\_P4\_checks.m} contains exact calculations over $\QQ$ used in the upper bound. The procedure
\texttt{PlaneBoundaryCertificate} constructs the ten conditions defining the
family in Proposition~\ref{prop:all-plane-systems}, computes their Smith form
over $\QQ[t]$, specializes the saturated kernel, and verifies the two residual
binary quartics, the ranks at their roots, and the two values of $J^2/I^3$.
The procedures \texttt{TypeIFlatLimitCertificate} and
\texttt{TypeIStabilizerCertificate} check, respectively, the Hilbert
polynomial of the type I limit in Proposition~\ref{prop:all-rank-one-systems}
and that the infinitesimal stabilizer of the displayed system consists only
of scalar matrices. The procedure
\texttt{DeterminantalPfaffianCertificates} checks the minors, principal
Pfaffians, and Hilbert polynomials used in
Propositions~\ref{prop:H1-absorbed} and \ref{prop:H2-absorbed}.
Finally, \texttt{QuadraticBranchCertificates} verifies an auxiliary
resultant identity and the associated coefficient comparisons, the syzygies
of $(v^2,-uv,u^2)$, and the Hilbert polynomials of representative residual
schemes in the quadratic branch. These calculations are supplementary: the
exhaustive plane-Cremona reduction \eqref{eq:Q-plane-cremona-map}, the
rational-quartic degeneration
\eqref{eq:Q-rational-quartic-limit-system}, and the rank-three conic
specialization \eqref{eq:Q-rank-three-kernel-limit} are proved directly in
the text and do not rely on the script. The procedure
\texttt{LowRankNormalFormCertificates} verifies the image relation in the
last normal form of Proposition~\ref{prop:R-empty}. Calling
\texttt{RunChecks()} runs all these certificates.

We stress that the scripts only verify the finite calculations listed above. The reductions to the corresponding normal forms, the flatness arguments, and the geometric classification are proved in the text.

\bibliographystyle{amsalpha}
\bibliography{Biblio}
\end{document}